\documentclass[11pt]{amsart}
\usepackage{amsmath, amssymb}
\usepackage[colorlinks=true,linkcolor=blue,urlcolor=blue,citecolor=blue, pagebackref]{hyperref}
\usepackage[alphabetic]{amsrefs}
\usepackage{amsfonts, epsfig,color,wrapfig,epstopdf}
\usepackage{ xypic,verbatim,amscd,color}
\usepackage{youngtab, ytableau}
\usepackage{graphicx}
\usepackage{colordvi}
\newcommand\op{\operatorname}
\newcommand\bull{{\bullet}}

\newcommand\frh{\mathfrak{h}}

\newcommand\symbo{\sqrt{-1}}

\newcommand\Sh{\operatorname{Sh}}

\newcommand\mf{\mathcal{F}}

\newcommand\mk{\mathcal{K}}
\newcommand\pone{\Bbb{P}^1}

\newcommand\Pic{\operatorname{Pic}}

\newcommand\me{\mathcal{E}}

\newcommand\mt{\mathcal{T}}
\newcommand\rigid{\mathcal{T}}

\newcommand\Fl{\operatorname{Fl}}

\newcommand\mv{\mathcal{V}}
\newcommand\tensor{\otimes}
\newcommand\ml{\mathcal{L}}

\newcommand\ma{\mathcal{A}}
\newcommand\shom{\mathcal{H}{\operatorname{om}}}
\newcommand\codim{\operatorname{codim}}
\newcommand\mb{\mathcal{B}}

\newcommand\mg{\mathcal{G}}

\newcommand\im{\operatorname{im}}
\newcommand\mw{\mathcal{W}}

\newcommand\mq{\mathcal{Q}}
\newcommand\rk{\operatorname{rk}}
\newcommand\ms{\mathcal{S}}

\newcommand\mn{\mathcal{N}}
\newcommand\ParL{\mathcal{P}\text{ar}_L}

\newcommand\Hom{\operatorname{Hom}}

\newcommand\Par{\mathcal{P}\text{ar} }

\newcommand\Gr{\operatorname{Gr}}
\newcommand\SL{\operatorname{SL}}

\newcommand{\leto}[1]{\stackrel{#1}{\to}}

\newcommand\Bun{\operatorname{Bun}}
\newcommand\Quot{\operatorname{Quot}}

\newtheorem{theorem}{Theorem}[section]

\newtheorem{remark}[theorem]{ Remark}

\newtheorem{corollary}[theorem]{Corollary}
\newtheorem{question}[theorem]{Question}
\newtheorem{proposition}[theorem]{Proposition}

\newtheorem{lemma}[theorem]{Lemma}

\newtheorem{definition}[theorem]{Definition}
\newtheorem{definition/lemma}[theorem]{Definition/Lemma}
\newtheorem{defi}[theorem]{Definition}

\numberwithin{theorem}{subsection}

\newtheorem{example}[theorem]{\bf Example}

\begin{document}
\title[Vertices of eigenpolytopes and rigid local systems]{Rigid local systems and  the multiplicative eigenvalue problem}
\author{Prakash Belkale}
\maketitle

\begin{abstract}
We give a construction which produces irreducible complex rigid local systems on $\Bbb{P}_{\Bbb{C}}^1-\{p_1,\dots,p_s\}$ via quantum Schubert calculus and strange duality. These local systems are unitary and arise from a study of vertices in the polytopes controlling  the multiplicative eigenvalue problem for the special unitary groups $\op{SU}(n)$ (i.e., determination of the possible eigenvalues of a product of unitary matrices given the eigenvalues of the matrices). Roughly speaking, we show that the strange duals of the simplest vertices  of these polytopes give all possible unitary irreducible rigid local systems. As a consequence we obtain  that the ranks of unitary irreducible rigid local systems, including those with finite global monodromy, on $\Bbb{P}^1-S$ are bounded above if we fix the cardinality of the set $S=\{p_1,\dots,p_s\}$ and require that the local monodromies have orders which divide $n$, for a fixed $n$.  Answering a question of N. Katz, we show that there are no irreducible rigid local systems  of rank greater than one, with finite global monodromy, all of whose local monodromies have orders dividing $n$, when $n$ is a prime number.

We also show that all unitary irreducible rigid local systems on $\Bbb{P}^1_{\Bbb{C}} -S$ with finite local monodromies arise as solutions to the Knizhnik-Zamalodchikov  equations on conformal blocks for the special linear group. Along the way, generalising previous works of the author and J. Kiers, we give an inductive mechanism for determining all vertices in the multiplicative eigenvalue problem for $\op{SU}(n)$.

\end{abstract}

\setcounter{tocdepth}{1}
\tableofcontents

\section{Introduction}
\subsection{Rigid local systems}
A rank $\ell$ local system of complex vector spaces on $\Bbb{P}^1_{\Bbb{C}}-S$ where $S=\{p_1,\dots,p_s\}$, a collection of distinct points,  is equivalent to an $s$-tuple $(A_1,\dots,A_s)$ of matrices $A_i\in \op{GL}(\ell,\Bbb{C})$ with product $A_1A_2\cdots A_s= I_{\ell}$, the identity. Such tuples are taken up to conjugacy, i.e.,  $(A_1,\dots,A_s)\sim (CA_1C^{-1}, CA_2C^{-1},\dots, CA_sC^{-1})$ where $C\in  \op{GL}(\ell,\Bbb{C})$.

The connection to local systems arises from the standard description of the fundamental group of an $s$-punctured Riemann sphere $\pi_1=\pi_1(\Bbb{P}^1-\{p_1,\dots,p_s\},b)$ at any base point $b$ as the free group on generators $\gamma_{1},\dots,\gamma_{s}$ given by suitable loops around the punctures, modulo the relation $\gamma_{1}\gamma_{2}\cdots\gamma_s=1$.  Rank $\ell$ local systems can then be identified with $\ell$-dimensional representations $\rho:\pi_1\to \op{GL}(\ell,\Bbb{C})$, with $\rho(\gamma_i)=A_i$. The conjugacy class of $A_i$ is
called the local monodromy at the point $p_i$ of the local system, and the image of $\rho$ is called  the (global) monodromy group of the local system.

Such a local system is {\em irreducible} if the representation of $\pi_1$ is irreducible. It is called {\em rigid} if any other local system with the same conjugacy classes of  local monodromies is isomorphic to it, i.e., if
$B_1B_2\cdots B_s=I_{\ell}$ and there exist $C_i\in \op{GL}(\ell,\Bbb{C})$ such that $B_i=C_iA_iC_i^{-1}$ then there exists a $C\in \op{GL}(\ell,\Bbb{C})$ such that $B_i=CA_iC^{-1}$ for $i=1,\dots,s$ \cite{Katz}.

Several classical examples of differential equations have local systems of solutions which are rigid and irreducible, namely the standard generalization of the hypergeometric functions  ${}_{\ell}F_{\ell-1}$ and Pochhammer differential equations (see e.g., \cite{BH} and \cite{Haraoka}).  Katz  \cite{Katz} has given an inductive procedure of constructing all  rigid irreducible local systems on  $\Bbb{P}^1-S$ via two operations, called middle convolution and tensoring. Simpson has classified rigid irreducible local systems with the conjugacy class of at least one of the matrices semisimple with generic distinct eigenvalues \cite{SimOld}. In this paper we produce a new family of rigid irreducible local systems from enumerative geometry (quantum Schubert calculus), see Corollary \ref{existence} and Sections \ref{qSC} and \ref{qSC2}. These local systems are unitary (see below).

Any irreducible rigid local system on $\Bbb{P}^1-\{p_1,\dots,p_s\}$ with quasi-unipotent local monodromies (i.e., the $A_i$ have roots of unity as eigenvalues)  is conjugate to the (usual) dual of its complex conjugate, and hence carries a non-degenerate hermitian form
\cite[Section 9.5] {Katz}. The signature of this hermitian form  is determined  in the corresponding classical cases in \cite{BH} and \cite{Haraoka}. There is a unique complex variation of Hodge structures that underlies an irreducible rigid local system, and the corresponding Hodge numerical data have been determined by Dettweiler and Sabbah \cite{DS} in a complemented Katz algorithm (which keeps track of additional local and global Hodge data), which then determines the signature of the form in general in an algorithmic procedure.

The case of unitary irreducible rigid local systems, i.e., those for which the form is positive or negative definite, is of special interest because it includes all irreducible rigid local systems with finite global monodromy (with a suitable converse, applied to all Galois conjugates, see e.g., \cite[Theorem 11]{BLocal}). We note here that for an irreducible unitary local system, it is the same to require that it is rigid in the sense of Katz as it is to require that it is rigid as a unitary representation, see Remark \ref{RigidRem} (and these are both equivalent to a numerical condition \eqref{eqRi}).

The problem of listing all rigid local systems with finite  global monodromy groups has classical origins. The global monodromy group of an irreducible local system is finite  if and only if  the corresponding differential equation (with regular singularities) has algebraic solutions. For  Gaussian hypergeometric functions, which correspond to rigid local systems of rank $2$, the finiteness problem was solved in the $19$th century by Schwarz \cite{Schwarz}. More recently \cite{BH} and \cite{Haraoka} solved the finiteness problems for  the hypergeometric functions  ${}_{\ell}F_{\ell-1}$ and Pochhammer differential equations, respectively. For Goursat-II rigid local systems \cite{EGoursat} of rank $4$ the question of when the global monodromy group is finite has been analysed recently in \cite{Goursat}.

Fix a natural number $n$. In Theorem \ref{egregium} (resp. Corollary \ref{egregiumprime}) we show there are only finitely many unitary (resp. finite global monodromy) irreducible rigid local systems (without fixing the ranks) with eigenvalues of local monodromy $n$th roots of unity, and also parametrize them. Answering a question of N. Katz, we are able to rule out the existence of rigid local systems of rank $>1$ with  finite global monodromy when $n$ is prime (Theorem \ref{KatzQ} below). We do this by relating the existence of such local systems via strange duality to the multiplicative eigenvalue problem for the special unitary group $\op{SU}(n)$, and quantum Schubert calculus, which are described in the next section. In Section
\ref{earlyfeedback}  we discuss what we know about the place of unitary rigid local systems in Katz's algorithm.

\subsection{Vertices in the multiplicative eigenvalue problem}\label{partita}
 Conjugacy classes in the special unitary group $\operatorname{SU}(n)$ are in bijective correspondence with points of the simplex
\begin{equation}\label{defD}
\Delta_n=\Delta(\operatorname{SU}(n))=\{{a}=(a_1,\dots,a_n)\in \Bbb{R}^n\mid  a_1\geq a_2\geq\dots \geq a_n\geq a_1-1,\  \sum_{i=1}^n a_i=0\}
\end{equation}
(The correspondence takes $a$ to the conjugacy class of the diagonal matrix with entries $\exp(2\pi \sqrt{-1}a_i$).

Define $P_n(s)\subset \Delta_n^s$ to be the set of tuples  $\vec{a}=(a^1,a^2,\dots,a^s)$ such that there exist $A_{1}, A_{2}, \dots, A_{s}\in \operatorname{SU}(n)$ with $A_{i}$ in the conjugacy class corresponding to $a^i$ and
$A_{1}A_{2}\cdots A_{s}=I_n\in  \operatorname{SU}(n).$

It is known that $P_n(s)$ is a compact polytope cut out by a finite set of inequalities controlled by quantum Schubert calculus of Grassmannians (Theorem \ref{ABW} below).  A point $\vec{a}=(a^1,a^2,\dots,a^s)$ is in $P_n(s)$ iff there is a representation of the fundamental group $\rho:\pi_1(\Bbb{P}^1-\{p_1,\dots,p_s\},b)\to\operatorname{SU}(n)$ such that the local monodromy around $p_j$, $A_{j}=\rho(\gamma_{j})$ is conjugate to $a^j$, $j=1,\dots,s$, and hence the problem  of determination of $P_n(s)$ is the same as the following: Does there exist a unitary local system on $\Bbb{P}^1-\{p_1,\dots,p_s\}$ with given local monodromy data $\vec{a}$? The polytope $P_n(s)$ also controls non-zero structure constants in the quantum cohomology of Grassmannians, and in the fusion ring of $\operatorname{SL}(n)$, see e.g., \cite{BHorn}. The rational cone over the polytope $P_n(s)$ is a generalization of the fundamental Littlewood-Richardson cone of algebraic combinatorics \cite{FBull,Z} (also Remark \ref{classicall}).

 The problem of determining the vertices of $P_n(s)$ was suggested to the author by Nori in 1996-97. It was noted in the author's PhD thesis (see \cite[Section 7]{BLocal}) that there are vertices of $P_n(s)$ other than the ``trivial'' ones (i,.e,  when $a^1,\dots,a^s$ correspond to $\zeta^{m_i} I_n$ with $\zeta=\zeta_n=\exp(2\pi\sqrt{-1}/n)$  with $n$ dividing $\sum_{i=1}^s m_i$). Thaddeus has  recently given interesting examples of many other non-trivial vertices \cite{thaddy}. The Mehta-Seshadri theorem \cite{MS}  can be used to derive a system of inequalities
 determining $P_n(s)$ \cite{Biswas,AW,BLocal}, see Theorem \ref{ABW} below.

 Note that if $\vec{a}\in P_n(s)$ is a vertex and $\rho$ any unitary representation of the fundamental group
 with local monodromies given by $\vec{a}$, then $\rho$ is reducible (see Lemma \ref{pinky}).
\begin{definition}\label{defFprime}
A  vertex $\vec{a}$ of $P_n(s)$ is called a F-vertex if there is a unique unitary local system with  local monodromies given by $\vec{a}$. In other words the corresponding matrix equation  $A_1A_2\cdots A_s=I_n$, $A_i\in \op{SU}(n)$ has a unique (possibly block diagonal) solution up to conjugacy.
Not all vertices of $P_n(s)$ have this property (see Section \ref{Rex} for an example).
\end{definition}
\subsection{Parametrization of rigid and irreducible unitary local systems}\label{newRigid}
 Suppose we are given an $s$-tuple of semisimple  conjugacy classes $\mathcal{A}=(\bar{A}_1,\dots,\bar{A}_s)$ in $\op{GL}(\ell,\Bbb{C})$, which are all $n$th roots of the identity (i.e., $\bar{A}_i^n=I_{\ell}$) such that  $\prod_{i=1}^s\det \bar{A}_i=1$. Assume  that  the eigenvalues of $\bar{A}_i$ are $\exp(2\pi\sqrt{-1}\mu_j^i/n)$ with integers $0\leq \mu_j^i<n$, so that the $\mu^i=(\mu^i_1,\dots,\mu^i_{\ell})$ are Young diagrams that fit in an $\ell\times n$ box, i.e.,  for $i\in [s], n> \mu^i_1\geq \mu^i_2\geq \dots \geq \mu^i_{\ell}\geq 0$. We will describe below a construction that assigns to $\mathcal{A}$ a point $v(\mathcal{A})$ of $\Delta_{n}^s$.

Consider the vector of transpose Young diagrams $\vec{\lambda}=(\lambda^1,\dots,\lambda^s)$ with $\lambda^i=(\mu^i)^T$. These are Young diagrams that fit in an $n\times\ell$ box.  Define points
${a}^i\in\Delta_{n}$ as follows
$$a^i=\frac{1}{\ell}\bigl(\lambda^i -\frac{1}{n}|\lambda^i|(1,1,\dots,1)\bigr).$$
Here $|\lambda^i|$ is the number of boxes in the Young diagram $\lambda^i$, i.e., $|\lambda^i|=\sum_{j=1}^n \lambda^i_j$. Note that the transpose operation is related to strange duality (see Section \ref{etrange}, and Lemma \ref{reprem}) which is an important technique in our work.
\begin{defi}\label{finiteness}
Define $v(\mathcal{A})=(a^1,\dots,a^s)\in \Delta_{n}^s$.
\end{defi}
\begin{theorem}\label{egregium}
Let $n$ be fixed. Then $\mathcal{A}\mapsto v(\mathcal{A})$  is a bijective correspondence between the following sets
\begin{enumerate}
\item The set of all $\ma$ with varying $\ell$, but fixed $n$ (i.e., eigenvalues of $\bar{A}_i$ are $n$th roots of unity) such that there is a rigid, irreducible unitary local system with local monodromy data $\mathcal{A}$.
\item The set  of F-vertices in $P_n(s)$.
\end{enumerate}
Therefore, since there are only finitely many vertices of $P_n(s)$, there are only finitely many unitary irreducible rigid local systems (without fixing the ranks) with eigenvalues of local monodromy $n$th roots of unity. The ranks $\ell$ of such local systems are therefore bounded above by a constant $C(|S|,n)$.
(Remark \ref{Hilberty} gives a stronger finiteness statement.)
\end{theorem}
(We do not know whether the conditions of irreducibility, rigidity, and semisimplicity of local monodromies, but without unitarity, imply an upper bound for $\ell$ in terms of $n$, see Section \ref{afterN}.)

The quantum versions of the saturation theorem \cite{BHorn}, a Fulton conjecture  type scaling result (see \cite[Remark 8.5]{BKq}), as well as  a numerical version of strange duality (see Section \ref{etrange}) are essential ingredients in the proof of the above result.  The mapping $\mathcal{A}\mapsto v(\mathcal{A})$ has the following preliminary  property (see Proposition \ref{schubertt}): There is a unitary representation of $\pi_1$ with  local monodromy data $\mathcal{A}$ if and only if
 $v(\ma)$ is a point of $P_n(s)$. Theorem \ref{egregium} is a finer comparison of properties under $\mathcal{A}\mapsto v(\mathcal{A})$.


For rigid local systems with finite global monodromy there is also an exact parametrization:
\begin{defi}\label{aliA}
The simplex  $\Delta_n$ is identified with $\Delta'_n= \{{x}=(x_1,\dots, x_{n-1})\mid x_i\geq 0, \sum x_i=1\}$. The correspondence takes $(a_1,\dots, a_n)$ in Definition \ref{defD} to ${x}$ with $x_i=a_i-a_{i+1}, i\in[n-1]$.  If $\op{gcd}(m,n)=1$, then there is a map $T_m:\Delta'_n \to\Delta'_n$ given by ${x}\mapsto {x}'$ with $x'_{mi \pmod{n}}=x_i$.   We will denote the resulting map $\Delta^s_n\to \Delta^s_n$ also by $T_m$.
\end{defi}
\begin{corollary}\label{egregiumprime}
Let $n$ be fixed. Then  $\mathcal{A}\mapsto v(\mathcal{A})$  is a bijective correspondence between the following sets
\begin{enumerate}
\item The set of all $\mathcal{A}$ with varying $\ell$, but fixed $n$ (i.e., eigenvalues of $\bar{A}_i$ are $n$th roots of unity) such that there is a rigid, irreducible  local system with local monodromy data $\mathcal{A}$ with finite global monodromy.
\item The set  of F-vertices $\vec{a}$ in $P_n(s)$ with the property $T_m(\vec{a})$ is a F-vertex of $P_n(s)$ for all integers $1\leq m\leq n$ relatively prime to $n$.
\end{enumerate}
\end{corollary}
We use Corollary \ref{egregiumprime} to answer a question of Katz (personal communication) in the affirmative:
\begin{theorem}\label{KatzQ}
Let $n$ be a prime number. There are no rigid irreducible  local systems $\mt$ of rank $\ell>1$ with finite global monodromy such that the local monodromies of $\mt$ are $n$th roots of unity (i.e., the local monodromies $\bar{A}_i$ satisfy $A_i^n=I_{\ell}$).
\end{theorem}
Theorem \ref{KatzQ} is proved  in Section \ref{KatzSection} by establishing the following property (P) true of all unitary irreducible rigid local systems. Let $n$ be arbitrary, and $\mt$ a rank $\ell$ unitary irreducible rigid local system with
local monodromy data $\ma=(\bar{A}_1,\dots,\bar{A}_s)$ with $\bar{A}_i^n=I_{\ell}$ for each $i\in[s]$.
Then,
\begin{itemize}
\item[(P)] If $1$ is an eigenvalue of $\bar{A}_1$, then $\zeta_n=\exp(2\pi\sqrt{-1}/n)$ is not an eigenvalue of $\bar{A}_1$.
\end{itemize}
In the above context if $\mt$ has finite global monodromy, then we can apply  (P) to all Galois conjugates (since they are all unitary), and obtain a stronger property that if  $1$ is an eigenvalue of $\bar{A}_1$, then any primitive $n$th root of unity cannot  be an eigenvalue of $\bar{A}_1$, which immediately gives
Theorem \ref{KatzQ} (by the action of permutations and $\tau_n(s)$  below, one can show that any local monodromy has to be a scalar when $n$ is prime).

In Section \ref{smalln}, we give a complete classification of all unitary rigid irreducible local systems, and those with finite global monodromy,  for $n\leq 6$ and $|S|=3$.
\begin{defi}\label{symmetry}
 Use the notation $a\mapsto \alpha\cdot a$ to denote the natural action of $n$th roots of unity $\alpha\in\mu_n=Z(\op{SU}(n))$ on a conjugacy class $a\in \Delta_n$. Define a finite group $\tau_n(s)$ as follows. Here $\zeta=\zeta_n=\exp(\frac{2\pi\sqrt{-1}}{n})\in \Bbb{C}^*$,
$$\tau_n(s)=\{(\zeta^{m_1},\zeta^{m_2},\dots,\zeta^{m_s})\mid \sum_{i=1}^s m_i\equiv 0 \pmod{n}\}.$$
\end{defi}
There is a coordinate by coordinate action of $\tau_n(s)$ on $P_n(s)$. The permutation group of $\{1,\dots,s\}$, denoted $S_s$ also acts on $P_n(s)$. The semidirect product $\tau_n(s)\rtimes S_s$ therefore acts on $P_n(s)$ permuting its vertices.

There is a coordinate by coordinate action of $\tau_n(s)$ also on the set of possible $\mathcal{A}$ with product of $\det(\bar{A}_i)$ equal to one. The correspondences above are compatible under the inverse automorphism $\tau_n(s)\to \tau_n(s)$. There is a similar compatibility  with the semidirect product.
\subsubsection{Determination of vertices}\label{review}
There is a system of inequalities that cuts out the compact polytope $P_n(s)$ (Theorem \ref{ABW}
below). But the nature of the vertices of $P_n(s)$, whether they are ``structured'' quantities, is a priori not clear
from such an approach. Generalising the solution of the classical version of this problem, which is the determination of the extremal rays in the Hermitian eigenvalue problem for arbitrary type \cite{BHermit,BKiers,Kiers}, we approach the problem of determining vertices using the following:
\begin{enumerate}
\item The cone over $P_n(s)$ is the effective cone of a moduli stack of parabolic bundles (Theorem \ref{blacktea} and Proposition \ref{b9} below). The main point is that the Mehta-Seshadri theorem
recasts the description of $P_n(s)$  as existence of semistable points in the moduli stack of parabolic bundles.
\item One can then construct points of $P_n(s)$ by constructing divisors (suitable degeneracy loci)  on moduli stacks of parabolic bundles.
\end{enumerate}
We consider a particular inequality from the system of inequalities given by Theorem \ref{ABW}, and construct some basic vertices (all are F-vertices)  on the corresponding facet  by describing some degeneracy loci on moduli stacks of parabolic bundles (Theorem \ref{Adagio} and \ref{brahms}). An algorithm describing the construction of the corresponding rigid local systems is isolated in Section \ref{qSC}. The local monodromies, and even the ranks, of these local systems are given by Gromov-Witten numbers and quantum Schubert calculus. The set of all F-vertices can be determined without solving the system of inequalities given in Theorem \ref{ABW}, see Proposition \ref{practical}.

The remaining vertices on the facet will arise from a numerically explicit induction process from Levi subgroups (see Section \ref{schumann} below).

\subsection{The inequalities defining $P_n(s)$}\label{ABWin}
Denote the Grassmannian of $r$-dimensional vector subspaces of a vector space $W$ by $\operatorname{Gr}(r,W)$, and  let
$\Gr(r,n)=\operatorname{Gr}(r,\Bbb{C}^n)$. Also fix a set of  $s$ distinct points $S=\{p_1,\dots,p_s\}$ on $\Bbb{P}^1$.
\begin{defi}\label{ineffable}
Let $$E_{\bull}:E_0=0\subsetneq E_1\subsetneq E_2\subsetneq\dots\subsetneq E_n=W$$  be a complete flag of vector subspaces in an $n$-dimensional vector space $W$. Let $I=\{i_1<\dots< i_r\}\subset [n]=\{1,\dots,n\}$ be a subset of cardinality $r$. The open Schubert variety $\Omega^0_I(E_{\bull})\subseteq \Gr(r,W)$ is defined as follows ($i_0=0$, and $i_{r+1}=n$):
$$\Omega^0_I(E_{\bull})=\{V\in \Gr(r,W)\mid \rk (V\cap E_j)=a,\text { for } i_a\leq j<i_{a+1}, a=0,\dots,r\}.$$
The closure, in $\Gr(r,W)$, of $\Omega^0_I(E_{\bull})$ is the closed Schubert variety
$$\Omega_I(E_{\bull})=\{V\in \Gr(r,W)\mid \rk (V\cap E_{i_a})\geq a, a=1,\dots,r\}.$$
\end{defi}
Let $\sigma_I\in H^{2|\sigma_I|}(\Gr(r,W),\Bbb{Z})$ be the cycle class of $\Omega_I(E_{\bull})$. Here
$|\sigma_I|=\sum_{a=1}^r(n-r+a-i_a)$
is the complex codimension of $\Omega_I(E_{\bull})$ in $\Gr(r,W)$.

\begin{defi}\label{GW}
Let $I^1,\dots ,I^s$ be subsets of $[n]$ each of cardinality $r$ and $d$ an integer $\geq 0$.
The Gromov-Witten number $\langle\sigma_{I^1}, \sigma_{I^2},\dots,\sigma_{I^s}\rangle_d$
is the number of maps (count as zero if infinite) $f:\Bbb{P}^1\to \operatorname{Gr}(r,n)$ of degree $d$ such that for $i=1,\dots,s$,
$f(p_i)\in\Omega_{I^j}(F^j_{\bullet})\subseteq \operatorname{Gr}(r,n)$. Here $F^1_{\bullet},\dots,F^s_{\bullet}$ are complete flags on $\Bbb{C}^n$ in general position. It is known that the number does not depend upon the choice of the distinct points $p_1,\dots,p_s$. If $\sum_{a=1}^n\codim \sigma_{I^s}\neq dn +r(n-r)$, then $\langle\sigma_{I^1}, \sigma_{I^2},\dots,\sigma_{I^s}\rangle_d=0$. 

Since maps $f:\Bbb{P}^1\to \operatorname{Gr}(r,n)$ are in one-one correspondence with rank $r$ subbundles $\mv\subseteq \mathcal{O}^{\oplus n}$ on $\Bbb{P}^1$, we can view Gromov-Witten numbers as  enumerative counts of subbundles of a fixed bundle (see Definition \ref{GWgen}).
\end{defi}

The system of inequalities that cuts out $P_n(s)$ inside $\Delta_n$ is the following \cite{AW,BLocal}, which is obtained using the theory of parabolic bundles and the Mehta-Seshadri theorem \cite{MS}:
\begin{theorem}\label{ABW}
Let $\vec{a}=(a^1,\dots,a^s)\in \Delta_n^s$.
Then, $\vec{a}\in P_{n}(s)$ if an only if for any collection of data
$(d,r,n,I^1,\dots,I^s)$ as in Definition \ref{GW}, with $\langle\sigma_{I^1}, \sigma_{I^2},\dots,\sigma_{I^s}\rangle_d= 1$, the inequality
\begin{equation}\label{wall}
\sum_{j=1}^s\sum_{k\in I^j} a_k^{j}\leq d
\end{equation}
holds.
\end{theorem}
The system of inequalities in \eqref{wall} is irredundant, see \cite[Remark 8.6]{BKq}, and hence equality in each inequality \eqref{wall} defines a facet of $P_n(s)$. Such facets are called {\em regular}, these are exactly those facets of $P_n(s)$ not contained in one of the alcove walls (i.e., a wall of $\Delta_n^s$).  In Lemma \ref{pinky}, we show that every vertex of $P_n(s)$ lies on  such a facet.

\subsection{Stacks of parabolic bundles}
As noted in Section \ref{review}, the Mehta-Seshadri approach towards describing $P_n(s)$ relies
on studying line bundles on moduli of parabolic bundles, which we review in this section.
\subsubsection{Lie theoretic definitions}\label{loudvoice}
Let $\Delta_{n,\Bbb{Q}}\subseteq \Delta_n$ be the subset of points with rational coordinates and $P_{n,\Bbb{Q}}(s)\subseteq\Delta^s_{n,\Bbb{Q}}$, the rational polytope
 $P_{n}(s)\cap \Delta^s_{n,\Bbb{Q}}$. The convex polytope $P_{n,\Bbb{Q}}(s)\subseteq \frh_{n,\Bbb{Q}}^s$ (which has the same vertex set as $P_n(s)$) is related to the effective cone of a suitable stack, see Proposition \ref{b9}.

Let $\frh_n$ (respectively $\frh_{n,\Bbb{Q}}$)  be the real (resp. rational) vector space of tuples $(a_1,\dots,a_n)$ with $\sum a_i=0$. $\frh_n^*$ is parametrized by tuples $\lambda=(\lambda_1,\dots,\lambda_n)$ with the identification $\lambda=\mu$ if $\lambda_a-\mu_a$ is constant in $a$.  We can normalise the representative by requiring that $\lambda_n=0$.

The Killing form gives an identification
$\kappa:\frh_n^*\to \frh_n$ which takes $\lambda=(\lambda_1,\dots,\lambda_n)$ to $a=(a_1,\dots,a_n)$ where
$a_k=\lambda_k-|\lambda|/n$ with $|\lambda|=\sum_{b=1}^n \lambda_b$. Also recall the positive Weyl chamber
$\frh_{+,n}\subseteq  \frh_{n}$ (resp. $\frh_{+,n,\Bbb{Q}}\subseteq  \frh_{n,\Bbb{Q}}$) defined as the subset of  $(a_1,\dots,a_n)$ with $a_1\ge\dots\ge a_n$.

\subsubsection{Moduli stacks}\label{loudvoice2}
The flag variety $\op{Fl}(T)$ is the variety of complete flags on a vector space $T$. Recall that line bundles on $\op{Fl}(n)=\op{Fl}(\Bbb{C}^n)$ correspond to integral weights $\lambda=(\lambda_1,\dots,\lambda_n)\in \frh_{n,\Bbb{Z}}^*$ (i.e., all $\lambda_i\in\Bbb{Z}$). The fiber of the line bundle $\ml_{\lambda}$ corresponding to $\lambda$ over the full flag $E_{\bull}$ is
$\bigotimes_{a=1}^{n} ((E^p_a/E^p_{a-1})^{\lambda_a})^*.$ The line bundle $\ml_{\lambda}$
 is effective iff $\lambda\in \frh_{+}^*$, defined as the set of $\lambda$ with $\lambda_1\ge\dots\ge \lambda_n$.

\begin{defi}
Let  $S=\{p_1,\dots,p_s\}\subseteq \Bbb{P}^1$ be a fixed collection of $s$ distinct points on $\Bbb{P}^1$.
\begin{enumerate}
\item For a vector bundle $\mathcal{W}$ on $\Bbb{P}^1$, define $\Fl_S(\mw)=\prod_{p\in S} \Fl(\mw_p)$. If $\me\in \Fl_S(\mw)$, we use the notation $\me=(E^p_{\bull}\mid p\in S)$.
\item  Let $\op{Bun}_{\mn}(n)$ be the moduli stack of  pairs   $(\mw,\gamma)$ where $\mw$ is a vector bundle on $\Bbb{P}^1$ of rank $n$,  and $\gamma$ is an isomorphism $\gamma:\det \mw\leto{\sim}\mn$ where $\mn\in \op{Pic}(\Bbb{P}^1)$. $\op{Bun}_{\mn}(n)$ is a smooth Artin stack of dimension $1-n^2$.
\item  Let $\Par_{n,\mn,S}$ be the moduli  stack of triples $(\mw,\me,\gamma)$ where $\mw$ is a vector bundle on $\Bbb{P}^1$ of rank $n$,   $\gamma$ is an isomorphism $\gamma:\det \mw\leto{\sim}\mn$, and $\me\in\Fl_S(\mw)$. The mapping  $\Par_{n,\mn,S}\to\op{Bun}_{\mn}(n)$ is a representable morphism
with $\op{Fl}(n)^s$ as fibers. $\Par_{n,\mathcal{O},S}$ is called the moduli  stack
of quasi-parabolic $\op{SL}(n)$ bundles with parabolic structures at the points $p_1,\dots,p_s$. It is a smooth Artin stack of dimension $\dim \Fl(n)^s + 1-n^2$.
\end{enumerate}
\end{defi}
By the shift operations discussed in Section \ref{shifty}, the stacks $\Par_{n,\mn,S}$ can be identified with each other as we vary $\mn$. Although our primary interest is when $\mathcal{N}=\mathcal{O}$, other choices of $\mathcal{N}$ are needed:
 \begin{itemize}
 \item[-]  In an enumerative (Gromov-Witten) situation of counting suitable subbundles of degree $-d$ of $\mathcal{O}^{\oplus n}$, the subbundle will not have a trivial determinant if $d\neq 0$. This will be important for induction operations.
 \item[-] The shift operation is needed in the derivation of the ``level'' of our line bundles. It is also needed in controlling some enumerative complexities.
 \end{itemize}
 \begin{defi}
We review some notation for irreducible Artin stacks $\mathcal{M}$:
\begin{enumerate}
\item $\Pic(\mathcal{M})$ denotes the Picard group of line bundles on $\mathcal{M}$.
\item Set $\Pic_{\Bbb{Q}}(\mathcal{M})=\Pic(\mathcal{M})\tensor\Bbb{Q}$.
\item Let $\Pic^+(\mathcal{M})\subseteq \Pic(\mathcal{M})$ denote the monoid of line bundles with non-zero global sections.
\item Let $\Pic^+_{\Bbb{Q}}(\mathcal{M})\subseteq \Pic_{\Bbb{Q}}(\mathcal{M})$ denote the subset of elements such that some multiple has a non-zero global section. $\Pic^+_{\Bbb{Q}}(\mathcal{M})$ is also called the cone of effective $\Bbb{Q}$-divisors on the stack $\mathcal{M}$.
\end{enumerate}
\end{defi}
The Picard group of $\Par_{n,\mn,S}$ is generated by the determinant of cohomology line bundle (defined below), and the line bundles that associate to a parabolic bundle, one of the subquotients of the parabolic bundle at one of the points of $S$:
\begin{defi}\label{levell}
There is a mapping
\begin{equation}\label{canonical}
(\bigoplus_{p\in S} \Pic(\Fl(n)))\oplus \Bbb{Z}=(\frh_{\Bbb{Z}}^*)^s\oplus\Bbb{Z}\to \Pic (\Par_{n,\mn,S})
\end{equation}
which sends $(\vec{\lambda},\ell)$ where  $\vec{\lambda}=(\lambda^1,\dots,\lambda^s)$, to
\begin{equation}\label{bruckner}
\mathcal{B}(\vec{\lambda},\ell)=\mathcal{D}^{\ell}\otimes \bigotimes_{j=1}^s\ml_{\lambda^j,p_j}
\end{equation}
where
\begin{enumerate}
\item[(i)] For each $p\in S$ and integral weight $\lambda=(\lambda_1,\lambda_2,\dots, \lambda_n)$, $\ml_{\lambda,p}$ is the line bundle whose fiber over
 $(\mw,\me,\gamma)$ is
  $$\ml_{\lambda}(\mw_p,E^p_{\bull})=\bigotimes_{a=1}^{n} ((E^p_a/E^p_{a-1})^{\lambda_a})^*.$$
\item [(ii)] $\mathcal{D}=\mathcal{D}(\mw)$ is the determinant of cohomology line bundle on $\Par_{n,\mn,S}$, whose fiber over $(\mw,\me,\gamma)$  is
$$\mathcal{D}(\mw)=\det H^0(\Bbb{P}^1,\mw)^*\otimes \det H^1(\Bbb{P}^1,\mw),$$
 where $\det T$  is the top exterior power $\wedge^m T$ of a vector space $T$, $m=\rk T$.
 \end{enumerate}

 Note that the line bundle on $\Par_{n,\mn,S}$ with fiber $\det \mw_p$ over $(\mw,\me,\gamma)$, for any fixed point $p\in\Bbb{P}^1$, is trivial because it is isomorphic to the fixed line $\mn_p$. Hence we may in  any $\lambda^i$ $\vec{\lambda}$ by $\lambda^i+(1,1,\dots,1)$ without changing the isomorphism type of
     $\mathcal{B}(\vec{\lambda},\ell)$.
\end{defi}

\begin{remark}\label{defL}
We call $\ell$ the level of $\mathcal{B}(\vec{\lambda},\ell)$.
\end{remark}

\begin{proposition}\label{LaSo}
The mapping \eqref{canonical} is an isomorphism.
\end{proposition}

When $\mathcal{N}=\mathcal{O}$, $\Par_{n,\mn,S}$ can be identified with the moduli stack of parabolic $G$-bundles for $G=\op{SL}(n)$, and the result follows from \cite{LS}, and the general case is reduced in Section \ref{shifty2} to this case using the shift operations in Section \ref{shifty}.

The following result allows us to rephrase the problem of finding extremal rays of the cone over $P_n(s)$ in terms of the effective cone of $\Par_{n,\mathcal{O},S}$, see e.g., \cite[Theorem 5.2]{BKq} for a proof:
\begin{theorem}\label{blacktea}
Assume $\mn=\mathcal{O}$, and $\ell\neq 0$.
$H^0(\Par_{n,\mathcal{O},S},\mathcal{B}^m(\vec{\lambda},\ell))\neq 0$
for some $m>0$ if and only if the following conditions are satisfied:
\begin{enumerate}
\item $\ell>0$ and $\lambda^j$ are dominant integral weights  of level $\ell$, $j=1,\dots,s$, i.e.,
\begin{equation}\label{weyl}
\lambda^{j}_1\geq \lambda^{j}_2\geq \dots\geq \lambda^{j}_n\geq \lambda^{j}_1-\ell.
\end{equation}
\item The point $\frac{1}{\ell}(\kappa(\lambda^1),\dots,\kappa(\lambda^s)))$ lies in $\Delta_n^s$ by the first condition.
The second condition is that this point is in $P_n(s)$.
\end{enumerate}
\end{theorem}
Theorem \ref{blacktea} implies the following:
\begin{proposition}\label{b9}
\hspace{2em}
\begin{enumerate}
\item $\Pic^+_{\Bbb{Q}}(\Par_{n,\mathcal{O},S})\subseteq (\frh_{\Bbb{Q}}^*)^s\times\Bbb{Q}$ is the cone over $P_{n,\Bbb{Q}}(s)\times 1$ in $\frh^s_{n,\Bbb{Q}}\times \Bbb{Q}$ under the Killing form identification $\frh=\frh^*$.
\item
The extremal rays of $\Pic^+_{\Bbb{Q}}(\Par_{n,\mathcal{O},S})$ correspond to vertices of $P_{n,\Bbb{Q}}(s)$ (which coincide with those of $P_{n}(s)$).
\end{enumerate}
\end{proposition}
\begin{defi}
An element $\ml\in \Pic^+_{\Bbb{Q}}(\Par_{n,\mathcal{O},S})$ is said to lie on a regular face if the corresponding point of  $P_{n}(s)$ lies on a regular facet (see the discussion following Theorem \ref{ABW}).
\end{defi}

We will henceforth consider the problem of describing the extremal rays of $\Pic^+_{\Bbb{Q}}(\Par_{n,\mathcal{O},S})$. The description follows the methods introduced in  \cite{BHermit,BKiers,Kiers}.

\begin{remark}\label{classicall}
There is a natural open embedding of the quotient stack   $[\op{Fl}(n)^s/\op{SL}(n)]$ in $\Par_{n,\mathcal{O},S}$, with $\op{SL}(n)$ acting diagonally on $\op{Fl}(n)^s$, by considering flags on a trivial vector bundle. The complement is a divisor with associated line bundle $\mathcal{D}$ (see Definition \ref{levell}). This gives a surjection of Picard groups
$$\Pic(\Par_{n,\mathcal{O},S})\to \Pic([\op{Fl}(n)^s/\op{SL}(n)])$$
inducing a surjection on the subsets of effective line bundles (using the fact that when the level is large with fixed $\vec{\lambda}$, conformal blocks equal classical invariants).  $\Pic^+([\op{Fl}(n)^s/\op{SL}(n)])$ is the classical Littlewood Richardson cone whose extremal rays (of the corresponding $\Bbb{Q}$-cone) were determined in \cite{BHermit}.
\end{remark}
\subsubsection{F-line bundles}\label{Def_F}
  \begin{defi}\label{defF}
 An effective line bundle $\ml=\mathcal{O}(E)=\mathcal{B}(\vec{\lambda},\ell)$ on $\Par_{n,\mathcal{N},S}$ is called an F-line bundle if
 \begin{itemize}
 \item[-]$h^0(\Par_{n,\mathcal{N},S},\mathcal{L}^m)=1$ for all integers $m\geq 0$ (it is sufficient to verify this condition for $m=1$ by the quantum Fulton conjecture (see Proposition \ref{qTheorems})), and
 \item[-] $E$ is a non-empty, reduced and irreducible divisor on $\Par_{n,\mathcal{N},S}$.
 \end{itemize}
 \end{defi}
 Note that the scaling properties are reminiscent of a conjecture of Fulton (see Section \ref{FSP}, but note that Fulton's conjecture did not include any irreducibility of the zeroes of sections).

There is a bijective correspondence between  F-vertices of $P_{n}(s)$ and  F-line bundles on $\Par_{n,\mathcal{O},S}$ (see Lemma \ref{corresp}).

\subsection{The extremal rays problem}
The cone  $\Pic^+_{\Bbb{Q}}(\Par_{n,\mathcal{O},S})$ is cut inside the space of $(\vec{\lambda},\ell)$ by three kinds of inequalities: (1) the ``affine Weyl alcove''  inequalities \eqref{weyl}, (2)
$\ell\geq 0$ and (3), the eigenvalue inequalities (corresponding to inequalities \eqref{wall}):
\begin{equation}\label{wallnew}
\sum_{j=1}^s\sum_{k\in I^j} \lambda^j_k\leq d\ell +\sum_{j=1}^s \frac{r|\lambda^j|}{n}
\end{equation}
for any collection of data
$(d,r,n,I^1,\dots,I^s)$ as in Definition \ref{GW}, with $\langle\sigma_{I^1}, \sigma_{I^2},\dots,\sigma_{I^s}\rangle_d= 1$.

By Lemma  \ref{pinky} and Proposition \ref{b9}, any extremal ray of $\Pic^+_{\Bbb{Q}}(\Par_{n,\mathcal{O},S})$ lies on a face given by equality in inequality
\eqref{wallnew}  for some $(d,r,n,I^1,\dots,I^s)$ with $\langle\sigma_{I^1}, \sigma_{I^2},\dots,\sigma_{I^s}\rangle_d= 1$. Denote this face of $\Pic^+_{\Bbb{Q}}(\Par_{n,\mathcal{O},S})$ by $\mf=\mf_{d,r,I^1,\dots,I^s}$. This leads us to the problem of describing  the extremal rays of  $\mf=\mf_{d,r,I^1,\dots,I^s}$ which we consider next.
\subsubsection{A description of some extremal rays of $\mf$}\label{basicR}
Let $(a,j)$ be a pair with $a\in[n]$ and $1\leq j\leq s$ such that
\begin{enumerate}
\item $a>1$, $a\in I^j$, and $a-1\not\in I^j$, or
\item $a=1$, $1\in I^j$, $n\not\in I^j$.
\end{enumerate}
We will produce a divisor $D(a,j)\subseteq \Par_{n,\mathcal{O},S}$, such that $\mathcal{O}(D(a,j))$ is an extremal ray of $\Pic^+_{\Bbb{Q}}(\Par_{n,\mathcal{O},S})$
which lies on the face $\mf$. We will also write $\mathcal{O}(D(a,j))=\mathcal{B}(\vec{\lambda},\ell)$ (see Definition \ref{levell}) and give formulas for $\ell$ and $\vec{\lambda}$.
\begin{defi}
Define subsets $J^1,\dots,J^s$ of $[n]$ each of cardinality $r$ and an integer $d'$ as follows:
In case (1) let $J^k=I^k$ if $k\neq j$, $J^j=(I^j-\{a\}) \cup \{a-1\}$ and $d'=d$, and in case (2) let  $J^k=I^k$ if $k\neq j$, $J^j=(I^j-\{1\}) \cup \{n\}$ and $d'=d-1$.

Let $D(a,j)$ denote the set  of $(\mw,\me,\gamma)\in \Par_{n,\mathcal{O},S}$ such that there exists a
coherent subsheaf $\mv\subseteq \mw$ of  rank $r$,  degree $-d'$, and for each $p=p_i\in S$, the map
 $$\mv_p\to \mw_p/E^p_{j^i_k}$$ has a kernel of rank $\geq k$ for $k=1,\dots,r$. We will consider $D(a,j)$ as a reduced closed substack of $\Par_{n,\mathcal{O},S}$.
\end{defi}
\begin{theorem}\label{Adagio}
\hspace{2em}
\begin{enumerate}
\item  $D(a,j)$ is an irreducible divisor in  $\Par_{n,\mathcal{O},S}$.
\item $h^0(\Par_{n,\mathcal{O},S}, \mathcal{O}(mD(a,j)))=1$ for all positive integers $m$. Hence  $\mathcal{O}(D(a,j))$ is an F-line bundle on $\Par_{n,\mathcal{O},S}$.
\item $\Bbb{Q}_{\geq 0}\mathcal{O}(D(a,j))$  is an extremal ray of $\Pic^+_{\Bbb{Q}}(\Par_{n,\mathcal{O},S})$.
\item The extremal ray $\Bbb{Q}_{\geq 0}\mathcal{O}(D(a,j))$ lies on the face  $\mf=\mf_{d,r,I^1,\dots,I^s}$ of $\Pic^+_{\Bbb{Q}}(\Par_{n,\mathcal{O},S})$.
\end{enumerate}
\end{theorem}

Set $$\mathcal{O}(D(a,j))=\mathcal{B}(\vec{\lambda},\ell).$$
There are formulas (also see Section \ref{concept}) for $\ell$ and $\vec{\lambda}$, and hence for $\mathcal{B}(\vec{\lambda},\ell)$:
Write $\lambda^i=\sum_{b=1}^{n-1} c_{i}^b\omega_b$, where $\omega_b$ corresponds to $(1,\dots,1,0,\dots,0)$ with $b$ ones and $n-b$ zeros, and is the $b$th fundamental weight.
\begin{theorem}\label{brahms}
\hspace{2em}
\begin{enumerate}
\item Let $J^{s+1}=\{1,n-r+1,n-r+2,\dots,n-1\}$. Then
\begin{equation}\label{formulel}
\ell=\langle\sigma_{J^1}, \sigma_{J^2},\dots,\sigma_{J^s},\sigma_{J^{s+1}}\rangle_{d'+1}.
\end{equation}
\item $c_{i}^{b}=0$ if $b\not\in J^i$, or $b+1\in J^i$. If $b\in J^i$ and $b+1\not\in J^i$, define subsets $J'^1,\dots,J'^s$ of $[n]$
 each of cardinality $r$ as follows: $J'^k=J^k$ if $k\neq i$, and $J'^i=(J^i-\{b\})\cup \{b+1\}$. Then
$$c_i^{b}= \langle\sigma_{J'^1}, \sigma_{J'^2},\dots,\sigma_{J'^s}\rangle_{d'}.$$
\end{enumerate}
\end{theorem}
We can readily convert the formulas for $c_i^b$ and $\ell$ above into the ranks of suitable conformal blocks for $\mathfrak{sl}_r$  at level $k=n-r$ (by using Witten's theorem, see \cite[Proposition 3.4]{BHorn} and Section \ref{QCB}).
Let $\mu^i=(\lambda^i)^T$. These are Young diagrams that fit into an $\ell\times n$ box. The following corollary, which is one of the main results of this paper, describes a procedure of producing rigid irreducible local systems from quantum Schubert calculus, and includes all known methods of systematically producing unitary irreducible rigid local systems (including the ones in \cite{BH,Haraoka}). For more
discussion see Sections \ref{qSC} and \ref{qSC2}.
\begin{corollary}\label{existence}
There exists a rigid irreducible unitary local system of rank $\ell$ on $\Bbb{P}^1-S$ whose local monodromy at $p_i$ is conjugate to the diagonal matrix with eigenvalues  $\exp(2\pi\sqrt{-1}\mu_j^i/n)$, $j\in[\ell]$ and $i\in[s]$. Furthermore, the following  equality holds:
  $$\sum_{i=1}^s \bigl((\ell-\sum_{b=1}^{n-1} c^b_i)^2 \ + \ \sum_{b=1}^{n-1}(c^b_i)^2\bigr)=(s-2)\ell^2+2.$$
\end{corollary}
\begin{remark}\label{brahms2}
Since $\mathcal{O}(D(a,j))$ is on the face $\mf$ (by Theorem \ref{Adagio}), the numbers $\ell$ and $c^i_b$ satisfy a linear equality (equality in the inequality \eqref{wallnew}) in addition to the quadratic equality in the above corollary.
\end{remark}
Examples of vertices arising from Theorem \ref{brahms} are given in Section \ref{somemore}. These include a vertex of $P_8(3)$ found by Thaddeus \cite{thaddy} and an infinite series of examples found by Kiers and Orelowitz.

\subsection{Induction}\label{schumann}
The basic idea behind the induction operation is to take a  ``general'' parabolic bundle $(\mw,\me,\gamma)$ and locate the subbundle $\mv\subseteq \mw$ corresponding to the enumerative problem  defining the face $\mf$, i.e., $\langle\sigma_{I^1}, \sigma_{I^2},\dots,\sigma_{I^s}\rangle_d= 1$ (Definitions \ref{GW}  and \ref{GWgen}). Let $\mq=\mw/\mv$. Clearly  $\mv$ and $\mq$ are vector bundles of ranks $r$ and $n-r$ respectively, and we have a given isomorphism $\delta:\det\mv \tensor\det\mq\to \mathcal{O}$. We also get induced flags on $\mv$ and $\mq$ at points of $S$ i.e., elements of  $ \Fl_S(\mv)$ and $\Fl_S(\mq)$.  Let $\ParL$ (see Definition \ref{parul}) be the moduli stack parametrizing this data.  $\ParL$ can be considered ``a  degree shift"  of  $\Par_{r,\mathcal{O},S}\times\Par_{n-r,\mathcal{O},S}$ (see Property (I3) in Section \ref{noncan}, and Proposition \ref{easily}).

The above construction leads to a rational morphism
$$\Par_{n,\mathcal{O},S}\to \ParL$$
which  extends to codimension $2$, and hence allows for a pullback operation of line bundles.
Effective line bundles on $\ParL$ correspond to  the extremal  ray problem for smaller groups $\SL(r)$ and $SL(n-r)$. This results in the induction morphism described below which has image in a space $\mf^{(2)}$ ``complementary" to the basic extremal rays from Theorem \ref{Adagio}.

Let $D_1,\dots,D_q$ be the basic extremal ray generators produced as $[D(a,j)]$ in Section \ref{basicR}. The following three results complete our picture of $\mf$:
\begin{enumerate}
\item[(a)]
 There is a rational cone $\mf^{(2)}$ such that
there is a natural product structure (Proposition \ref{Br4})
\begin{equation}\label{productstructure}
\mf\leto{\sim}\prod_{i=1}^q \Bbb{Q}_{\ge 0} D_i \times \mf^{(2)}
\end{equation}
\item[(b)] There is a  surjection (called induction), with explicit formulas (see Section \ref{explicitF}),
\begin{equation}\label{inductiones}
\Pic^+_{\Bbb{Q}}(\Par_{r,\mathcal{O}(-d),S})\times \Pic^+_{\Bbb{Q}}(\Par_{n-r,\mathcal{O}(d),S})\twoheadrightarrow \mf^{(2)}.
\end{equation}
Therefore the extremal rays of $\mf$ are either one of the $\Bbb{Q}_{\ge 0} D_i$ for $1\leq i\leq q$, or extremal rays of $\mf^{(2)}$. From the surjection \eqref{inductiones}, the extremal rays of $\mf^{(2)}$ are images of some of the extremal rays of $\Pic^+_{\Bbb{Q}}(\Par_{r,\mathcal{O}(-d),S})$ (which is a shift of the cone for $\SL(r)$) and $\Pic^+_{\Bbb{Q}}(\Par_{n-r,\mathcal{O}(d),S})$ (which is a shift of the cone for $\SL(n-r)$), although not all such images need give extremal rays of $\mf^{(2)}$.
\item[(c)]
All F-line bundles on the face $\mf$ can be described as follows: They are one of the line bundles $\mathcal{O}(D_i)$ for $1\leq i\leq q$, or lie on $\mf^{(2)}$.  The F-line bundles on $\mf^{(2)}$ can be put in one-one correspondence  (under induction)  with an explicit subset of the disjoint union of F-line bundles on the Levi factors $\Par_{r,\mathcal{O},S}$ and $\Par_{n-r,\mathcal{O},S}$ (See Proposition \ref{condensed}).
\end{enumerate}

For the non F-line bundle extremal rays of  $\mf^{(2)}$, an exact parametrization as in (c) is not known. This is probably related to general patterns of representation theory (e.g., \cite[Chapter 49]{Bump}) in which induction operations  may not preserve irreducibility. Another technical issue is that we are inducing only from maximal parabolics and not all parabolics. However, the induction procedure still produces a set of generating rays of the face, and this has been used by J. Kiers (in a related context) to prove  the saturation conjecture for $\op{Spin}(12)$  \cite{Kiers1}. In this case, the linear programming problem starting from the defining inequalities was prohibitively slower to solve than the linear programming problem starting from the generating set (basic rays given by degeneracy loci and what is obtained from
induction).

\subsection{Strange duality}\label{etrange}
The numbers $h^0(\Par_{n,\mathcal{O},S},\mathcal{B}(\vec{\lambda},\ell))$ are structure coefficients in the Verlinde algebra (also called the fusion ring)
for the Lie algebra $\mathfrak{sl}_n$ at level $\ell$.  We will use strange duality, in its numerical form in genus zero,  which asserts that these Verlinde structure coefficients for the Lie algebra $\mathfrak{sl}_n$ at level $\ell$ are equal to suitable structure coefficients for the Lie algebra $\mathfrak{sl}_{\ell}$ at level $n$. See Section \ref{understanding} for a discussion of the basic  geometric ideas underlying this duality.

The data $\vec{\lambda}$ goes over to the transposed data $\vec{\lambda^T}$ with a central  twist (assume here that $\vec{\lambda}$ consists of Young diagrams that fit into an $n\times\ell$ box).
\begin{itemize}
\item[-] Strange duality takes $(\vec{\lambda},\ell)$ on $\Par_{n,\mathcal{O},S}$ to $(\vec{\lambda^T},n)$ on $\Par_{\ell,\mathcal{O}(-\tilde{d}),S}$ with $-\tilde{d}= \frac{1}{n}(n\ell-\sum |\lambda^i|)$ (assumed to be an integer). The matter of central twisting tells us how to pass to $\tilde{d}=0$, see Section \ref{shifty}.
\end{itemize}

This picture is very general, and applies to many other groups: in the physics picture it accompanies conformal embeddings of Lie algebras (see  \cite{BJDG} and the references therein). In the case of special linear groups there is a simple and vivid way of explicating this numerical relation  using the result of  Witten \cite{Witten}  (also Section \ref{QCB} below) that  the fusion structure coefficients for the Lie algebra $\mathfrak{sl}_n$ at level $\ell$ are equal to, up to some cyclic twists, the  structure constants in the quantum cohomology of the space $\Gr(n,n+\ell)$. Since (by ``Grassmann duality'') $\Gr(n,n+\ell)=\Gr(\ell,n+\ell)$, the strange duality follows.

Theorem \ref{egregium} and Corollary \ref{egregiumprime} are proved using strange duality in Section \ref{SD}. Here we use the fact that the non-zeroness of ranks of line bundles on parabolic moduli stacks control existence of unitary local systems with prescribed monodromies (Propositions \ref{qTheorems}, \ref{kappa}, and \ref{sansD}).
 \subsection{Rigid local systems and KZ equations}
 The Knizhnik-Zamalodchikov equations from conformal field theory are  differential equations (i.e., vector bundles with flat connections)
 on the configuration spaces of points in $\Bbb{A}^1$ corresponding to a Lie algebra and certain data of representations, and a level.  These differential equations are special cases of the Wess-Zumino-Witten (WZW) connection on conformal blocks in any genus \cite{tuy} (which includes Hitchin's connection on spaces of non-abelian theta functions as a special case \cite{laszlo}). The WZW formalism allows for a study of the differential equations on the boundary of moduli spaces (i.e., when the marked points coalesce, and the connection is logarithmic over the boundary).

 There is a geometric realization \cite{Ram,BKZ,BM} of the monodromy of KZ equations in genus zero  on conformal blocks, in terms of cohomology of unramified covers of configuration spaces \cite{SV}, providing conformal blocks with a (Hodge theoretic) unitary structure, preserved by the KZ connection (possibly reducible, and not necessarily rigid, also see Remark \ref{hodge}). In Section \ref{chocolate}, we give  a construction that shows that all unitary irreducible rigid local systems with local monodromies  of finite orders come from KZ equations. This can be interpreted as a kind of ``universality'' statement regarding KZ equations.

\subsubsection{Acknowledgements}
I thank N. Fakhruddin, J. Kiers, S. Kumar, M. Nori, S. Mukhopadhyay and A. Wilson for useful discussions and comments/corrections. I thank N. Katz for asking the question answered in Theorem \ref{KatzQ}, and for his comments on an earlier version of this paper. I thank J. Kiers and G. Orelowitz for showing me  the infinite set of examples given here in Section \ref{KO}.
\section{An overview of the methods}\label{uber}
Using  some notation and definitions from the introduction, we outline the main constructions and the ideas that go into them.

The first input is the Mehta-Seshadri theorem \cite{MS} which recasts existence of unitary local systems on $\Bbb{P}^1-S$, in terms of non-emptiness of semi-stability loci, a condition involving existence of invariant sections of line bundes (in a GIT type situation). This is the same as existence of  non-zero sections of suitable tensor powers of a line bundle on a moduli stack of parabolic bundles (Theorem \ref{blacktea} is formulated for special unitary local systems, but the unitary case can be reduced to this). This condition is of the form
\begin{itemize}
\item
$H^0(\Par_{n,\mathcal{O},S},\mathcal{B}^m(\vec{\lambda},\ell))\neq 0$ for some $m$.
\end{itemize}

The next step is quantum saturation \cite{BHorn}, see  Proposition \ref{qTheorems} (1), which under  suitable conditions (of triviality of actions of isotropy groups), says that the above condition is equivalent to the non-vanishing for $m=1$.
We also have control over the rigidity of the moduli of unitary representations by quantum Fulton Proposition \ref{qTheorems}(2):  The moduli space is a point if and only if the above vector space is one dimensional for $m=1$ (note that the moduli space is the Proj of the section ring).

At this point, we have moved on from $P_n(s)$ the compact polyhedron controlling the unitary multiplicative eigenvalue problem to the cone over it, the effective cone of $\Par_{n,\mathcal{O},S}$ (see Theorem \ref{blacktea} and Proposition \ref{b9}). We now observe that some line bundles on $\Par_{n,\mathcal{O},S}$ satisfying geometrically defined properties give (some) extremal rays of the effective cones. These are the  F-line bundles defined in Definition \ref{defF}, see Lemma \ref{elgar}. The defining property is that they have a unique up to scalars non-zero section (even their tensor powers, but this follows from quantum Fulton), and further the zero scheme of this section is reduced and irreducible. The corresponding vertices of $P_n(s)$ are called F-vertices. Their defining property beyond being vertices of $P_n(s)$ is that the corresponding matrix equation has a unique solution up to  conjugation (See Definition \ref{defFprime}, and Lemma \ref{corresp}). Every line bundle $\mathcal{B}(\vec{\lambda},\ell)$ on $\Par_{n,\mathcal{O},S}$ has a level $\ell$ (Remark \ref{defL}).

The next input is strange duality. This says at the numerical level that $H^0(\Par_{n,\mathcal{O},S},\mathcal{B}(\vec{\lambda},\ell))$ has the same rank as a similar space of sections $H^0(\Par_{\ell,\widetilde{\mathcal{N}},S}\mathcal{B}(\vec{\lambda}^T,n))$ with $n$ and $\ell$ interchanged, and Young diagrams $\lambda_i$ replaced by their transposes, but the vector bundles (now of rank $\ell$)  now no longer have trivial determinants, but of a degree determined by the initial data, as in equation \eqref{GepnerPP}. We return to the history of \eqref{GepnerPP}, and the ideas behind understanding it geometrically in Section \ref{understanding}.

Once we have this, we have an equivalence of existence of a rank $\ell$ unitary local system with local monodromies $n$th roots of unity with that of a rank $n$ special unitary local system with corresponding local monodromies. The notions of F-line bundles and F-vertices are on the rank $n$-side. The notion that we consider on the rank $\ell$-side is Katz's notion of rigidity (the  representation should be irreducible, and the moduli space a point, also see Remark \ref{stretch2}).

The equivalence of the earlier paragraph leads to the question on what it does to rigid local systems on the rank $\ell$-side. This equivalence is non-linear since addition of Young diagrams as representations of $SL(n)$ is a row by row addition, and the transpose operation makes it a column by column operation. The answer is that rigids correspond to F-vertices, Theorem \ref{egregium}. 

We also associate to a rigid rank $\ell$ unitary representation with local monodromies $n$th roots of unity, a divisor on $\Par_{n,\mathcal{O},S }$ with enumerative significance (Lemma \ref{divisorE}, Proposition \ref{propP}), which puts non-trivial conditions (property (P) from the introduction) on the local monodromy of the rank $\ell$ unitary representation. One of the inputs here is a computation of divisor classes (see Section \ref{concept}).
The rigid unitary local system can then be constructed explicitly as a vector bundle of conformal blocks  equipped with the  Knizhnik-Zamalodchikov connection (Theorem \ref{KZe}).

We are therefore led to the question of determining vertices (and F-vertices)  of $P_n(s)$, which is the same as the question of determining the  extremal rays (and F-line bundles) of the cone $\Pic^+_{\Bbb{Q}}(\Par_{n,\mathcal{O},S})$ which is given by a system of inequalities parametrized by Gromov-Witten invariants  (Theorem \ref{ABW}, also see Proposition \ref{b9}). We determine the extremal rays on any regular face (see the discussion following Theorem \ref{ABW}) as follows:
\begin{enumerate}
\item[(a)] Some of extremal rays  correspond to explicit degeneracy loci  on $\Par_{n,\mathcal{O},S}$ (Theorems  \ref{Adagio} and \ref{brahms}). These give F-line bundles with numerical data given by Gromov-Witten invariants.
\item[(b)] The remaining extremal rays  on the face arise from a process of induction arising from a rational map
$\Par_{n,\mathcal{O},S}\to \ParL$ as explained in the introduction (Section \ref{schumann}). This process leads to a bijective parametrization of F-line bundles on the face  ( see Proposition \ref{condensed}).
\end{enumerate}
 The basic divisors in (a) correspond to the codimension 1 components in the na\"ive base locus of the rational map $\Par_{n,\mathcal{O},S}\to \ParL$ (see Remark \ref{explainmore}).
  \subsection{The strange duality construction}\label{understanding}
The equality  \eqref{GepnerPP} is the numerical consequence of a canonical  duality of vector spaces:
\begin{equation}\label{canonD}
H^0(\Par_{n,\mathcal{O},S},\ml)\leto{\sim} H^0(\Par_{\ell,\widetilde{\mathcal{N}},S},\ml')^*
\end{equation}
 This is the genus $0$ case of the strange duality conjecture for curves (\cite{Betrange,MO,Ou} and the surveys  \cite{paulyB, popa}). The strange duality conjecture has origins in mathematical physics (see e.g., \cite{NT}) where it is also called level-rank duality (also see \cite{BJDG}).

There is a natural divisor $D$ (called a theta divisor)  in the product $\Par_{n,\mathcal{O},S}\times \Par_{\ell,\widetilde{\mathcal{N}},S}$ whose associated line bundle is $\ml\boxtimes\ml'$. The divisor is a natural degeneracy locus: A point $(\mw,\me,\gamma)\times (\mt,\mf,\delta)$ is in $D$ if and only if there is a non-zero map $\phi:\mw\to \mt^*$ (or $\phi\in H^0(\Bbb{P}^1,\mt^*\tensor \mw^*)$) such that  $\phi_{p_i}(E^{p_i}_a)\subseteq G^{p_i}(\ell-\lambda^i_a)$ for $a=1,\dots,n$ and $i\in [s]$. Here $\mg\in \Fl_S(\mt^*)$ is induced from the identification  $\Fl_S(\mt^*)= \Fl_S(\mt)$ and the given $\mf\in\Fl_S(\mt)$. The expected dimension of such $\phi$ is zero in the settings considered in strange duality, so that the locus where a non-zero section exists is a divisor (see Section \ref{RamDiv} for one such setting).

The strange duality (or level-rank duality) is the assertion that $D$ results in an isomorphism
\eqref{canonD}. If we fix a point $x\in\Par_{\ell,\widetilde{\mathcal{N}},S}$, we get a corresponding section of $H^0(\Par_{n,\mathcal{O},S},\ml)$. Given the numerical equality of ranks, the strange duality is the assertion that  these sections as $x$ varies span $H^0(\Par_{n,\mathcal{O},S},\ml)$.

The theta divisor construction has been used frequently in the moduli theory of vector bundles (see e.g., \cite{Faltings}). In the non-parabolic (e.g., in higher genus) settings one is in situations where the Euler characteristic $\chi(\mt^*\tensor \mw^*)=0$, so that the locus where $H^0$ is non-zero is a divisor.

We  use only the numerical consequence of this duality. However a full picture of this duality may help in understanding the origin of the numerical equality (which is the first step in the proof of the vector space duality). The algebro-geometric proofs of strange duality proceed first by showing that the representation
theoretic spaces $H^0(\Par_{n,\mathcal{O},S},\ml)$ have ranks that are equal to enumerative numbers for the Grassmannian $\Gr(n,n+\ell)$ (because both are governed by ring structures,  the verification of this
comes down to very few cases, once the parameters are matched). Each point in the enumerative count is made to  give a section of $H^0(\Par_{n,\mathcal{O},S},\ml)$. The transversality in the enumerative count implies that these sections are linearly independent. Together with the numerical equality of the enumerative count and the rank of  $H^0(\Par_{n,\mathcal{O},S},\ml)$, we see that these sections span, and this gives the duality. A particularly simple case of this construction can be seen in the comparison of the structure constants in the representation rings of special linear groups and the intersection theory of Grassmannians \cite{BIMRN} (also see \cite[Section 6.2]{FBull}, and Remark \ref{classicall}).

\subsection{Enumerative computations}\label{concept}
In Section \ref{enume} we have a computation of cycle classes of some natural loci in $\Par_{n,\mathcal{N},S}$.  Consider  the space formed by  the closure of $(\mw,\me,\gamma)$ such that the rank $n$ vector bundle $\mw$ has a subbundle $\mv$ of rank $r$ whose fibers in $\Gr(r,\mw_p)$ lie in given Schubert varieties $\Omega_{J^i}(\me_p)$ with respect to the given flags on $\mw_p$, with $p\in S$. We also assume that the expected dimension of such $\mv$ is $-1$ having fixed $(\mw,\me,\gamma)$. The set of  such $(\mw,\me,\gamma)$ gives  a natural divisor in $\Par_{n,\mathcal{N},S}$ (possibly zero).

In Proposition \ref{enumerative} we compute this divisor class, i.e., the corresponding line bundle $\mathcal{B}(\vec{\lambda},\ell)$. The coefficients that enter $\lambda_i$ and $\ell$ are obtained by weakening the number of conditions by one, so that the expected dimension of such $\mv$ is now zero (and we count over generic $(\mw,\me,\gamma)$). The weakening of conditions can change the degree of the underlying bundle since we allow for subbundles becoming subsheaves. The set of weakenings corresponds to enlargements of one of the Schubert varieties $\Omega_{J^i}(E_{\bull})$  for some $i$ by one dimension. Such enlargements correspond to $\Omega_{I}(E_{\bull})$ with $I$ obtained by replacing some $a\in J^i$ by $a+1$ assuming $a+1\notin J^i$ and $a\neq n$ (if $a=n$ we need a degree shift to handle this). The complexity of this last condition translates to property (P) in relation to the natural divisor on $\Par_{n,\mathcal{O},S}$ associated to a unitary rank $\ell$ rigid local system (as earlier in this section, Lemma \ref{divisorE} and Proposition \ref{propP}).

\section{Quot schemes and some of their basic properties}\label{quot1}
We use notation and definitions  from the introduction.
\begin{defi}Suppose $\mn$ is a line bundle on $\Bbb{P}^1$. Let $\op{Quot}(d,r,\mn,n)\to \op{Bun}_{\mn}(n)$ denote the representable relative quot scheme of quotients, whose fiber over $\mw\in \op{Bun}_{\mn}(n)$ is the quot scheme of flat quotients $\mw\to\mq$ where $\mq$ is a coherent sheaf of rank $n-r$ and degree $d+\deg(\mn)$.
\end{defi}
See \cites{Laumon,CF} for the proof of the following:
\begin{proposition}\label{laumon}
$\Quot(d,r,\mn,n)$ is a smooth and irreducible Artin stack of dimension $r(n-r) +\deg(\mn)r +dn +1-n^2.$
\end{proposition}

\begin{defi}\label{mindful}
Let $\vec{J}=(J^1,\dots,J^s)$ be an $s$-tuple of subsets  of $[n]$ each of cardinality $r$.
\begin{enumerate}
\item Let $\Omega=\Omega(d,r,\mn,n,\vec{J})$ be the stack parametrizing
tuples $(\mv,\mw,\me,\gamma)$ where $\mw\to \mw/\mv$ gives a point of $\Quot(d,r,\mn,n)$, so that $\mv$ is a rank $r$ bundle of degree $-d$, and for each $p=p_i\in S$, the map
 $$\mv_p\to \mw_p/E^p_{j^i_k}$$
 has kernel of rank $\geq k$ for $k=1,\dots,r$.
\item Let $\Omega^0=\Omega^0(d,r,\mn,n,\vec{J})$ be the substack of $\Omega$ parametrizing tuples $(\mv,\mw,\me,\gamma)$ with $\mv$ a subbundle of $\mw$, and $\mv_p\in \Omega^0_{J^i}(\me_p)\subseteq \Gr(r,\mw_p)$ for all $p=p_i\in S$.
\end{enumerate}
\end{defi}
\begin{lemma}\label{repres}
The map $\Omega(d,r,\mn,n,\vec{J})\to \Par_{n,\mn,S}$, which sends $(\mv,\mw,\me,\gamma)$ to $(\mw,\me,\gamma)$, is representable and proper.
\end{lemma}
\begin{proof}
Let $\mathcal{Q}'$ be the fiber product of  $\Quot(d,r,\mn,n)$ and  $\Par_{n,\mn,S}$ over $\Bun_{\mathcal{N}}(n)$. The map  $\Omega(d,r,\mn,n,\vec{J})\to \mathcal{Q}'$ is representable and proper since the fibers are closed subschemes of products of flag varieties. Since  $\mathcal{Q}'\to \Par_{n,\mn,S}$ is a base change of $\Quot(d,r,\mn,n)\to \Bun_{\mathcal{N}}(n)$, it is proper and representable. The map $\Omega(d,r,\mn,n,\vec{J})\to \Par_{n,\mn,S}$ is a composite of these two maps and is hence  proper and representable.

\end{proof}

The following is the main result of this section. Statements of this form were first proved in \cite{bertie}:
 \begin{proposition}\label{maya}
  $\Omega=\Omega(d,r,\mn,n,\vec{J})$ is irreducible (possibly singular) of dimension
 \begin{equation}\label{codeyy}
\dim \Quot(d,r,\mn,n) + \dim \Fl(n)^s -\sum_{i=1}^s |\sigma_{J^i}|
\end{equation}
Furthermore, $\Omega^0(d,r,\mn,n,\vec{J})$ is dense in $\Omega$.
 \end{proposition}
\subsection{Proof of Proposition \ref{maya}}
 By definition $\Omega$ is a subset of  tuples $(\mv,\mw,\me,\gamma)$ cut out by certain rank conditions on the maps $\mv_p\to \mw_p$ for $p=p_i\in S$. These rank conditions are given locally by $ |\sigma_{J^i}|$ equations for each $i$ (see the universal local case in the proof of \cite[Theorem 14.3]{FInt}), and hence the dimension of each irreducible component of $\Omega$ is at least as much as \eqref{codeyy}.

 It is also clear that $\Omega^0(d,r,\mn,n,\vec{J})$ is a smooth fiber bundle over an open subset of $\Quot(d,r,\mn,n)$ with fibres of dimension  $\dim \Fl(n)^s -\sum_{i=1}^s |\sigma_{J^i}|$ (Any fiber can be built by first considering induced flags on the  fibers of the subbundle and then extending these flags to the fibers of the universal bundle). Hence $\Omega^0$ has exactly the dimension \eqref{codeyy}.

 To conclude the proof we need to show that any irreducible component of $\Omega-\Omega^0$ has dimension less than \eqref{codeyy}.  We begin with  special cases in Sections \ref{bruck01}, \ref{bruck02} and \ref{bruck03}. The general case, which may be considered a linear superposition of these special cases,  is then proved in Section \ref{bruck04}. We also pay attention to  the divisorial components of  $\Omega-\Omega^0$ since they will play a special role in some computations in Section \ref{compact}.
 See \cite[Section 12.1]{BGM} for a similar computation.

\subsubsection{Subbundle property failing at points not in $S$}\label{bruck01}
Consider first the following simple case: Let $q\not\in S$ and $\Omega_q(\epsilon)\subset \Omega$ the substack formed by
$(\mv,\mw,\me,\gamma)$ such that the following conditions are satisfied:
\begin{enumerate}
\item[(a)] $\mv\subseteq \mw$ is a subbundle on $\Bbb{P}^1-\{q\}$.
\item[(b)]   $\mv_p\in \Omega^0_{J^i}(\me_p)\subseteq \Gr(r,\mw_p)$ for all $p=p_i\in S$.
\item[(c)] The kernel of $\mv_q\to\mw_p$ has  dimension $\epsilon$.
\end{enumerate}
Let $\Quot_{\epsilon}(q)\subseteq  \Quot(d,r,\mn,n) $ be the substack formed by $\mv\subseteq \mw$ such that conditions (a) and (c) hold. Then the dimension of  $\Omega_q(\epsilon)$ is $\leq$ the quantity \eqref{codeyy} plus
\begin{equation}\label{simplecase}
\dim \Quot_{\epsilon}(q)-\dim  \Quot(d,r,\mn,n)
\end{equation}

Consider a point $(\mv,\mw)$ in  $\Quot_{\epsilon}(q)$, and let $K$ be the kernel in (c). Let $\widetilde{\mv}$ be the subsheaf of $\mw$ which coincides with $\mv$ outside of $q$ and whose sections in a neighborhood of $q$ are meromorphic sections $s$ of $\mv$ such that $ts$ is a holomorphic section of $\mv$ with fiber at $q$ in $K$. Here $t$ is a uniformizing parameter at $q$. $\mv\subset\widetilde{\mv}\subseteq\mw$ and the image  $Q$ of $\mv_q\to \widetilde{\mv}_q$ has dimension $r-\epsilon$. The degree of $\widetilde{\mv}$ is $-d'$ where $d'=d-\epsilon$. The pairing $(\mv,\mw,K)$ to $(\widetilde{\mv},\mw,Q)$ is a bijection between triples with $K$ a subspace of the kernel of $\mv_q\to \mw_q$, and triples $(\widetilde{\mv},\mw,Q)$ with $Q$ a subspace of $\widetilde{\mv}_q$. The quantity
\eqref{simplecase} is therefore
$$\dim \Gr(r-\epsilon,r)+ (d'-d)n=\epsilon (r-n-\epsilon) =-(n+\epsilon-r)\epsilon\leq -2.$$
The point $q$ moves over a curve of dimension one. Therefore the following is a definite possibility:
\begin{itemize}
\item On a codimension one substack of $\Omega$, the quotient $\mw/\mv$ fails to be locally free. The only case where  this could happen is
$r=n-1$ and $\epsilon=1$.
\end{itemize}
\subsubsection{Subbundle property fails at a point of $S$}\label{bruck02}
Consider first the following simple case: Let $p=p_1\in S$ and $\Omega_{p,L}(\delta)\subset \Omega$ the substack formed by
$(\mv,\mw,\me,\gamma)$ such that the following conditions hold:
\begin{enumerate}
\item[(a)] $\mv\subseteq \mw$ is a subbundle on $\Bbb{P}^1-\{p\}$.
\item[(b)]   $\mv_{p'}\in \Omega^0_{J^i}(\me_{p'})\subseteq \Gr(r,\mw_{p'})$ for all $p'\neq p\in S$.
\item[(c)] $\mv_q\to\mw_q$ has a kernel of dimension $\delta$.
\item[(d)] The image $Q$ of $\mv_p\to\mw_p$ is an $r-\delta$ dimensional subspace of $\mw_p$ which lies in the Schubert cell  corresponding to $L=\{l_1,\dots,l_{r-\delta}\}$.
\end{enumerate}

Computing as before, the dimension of  $\Omega_{p,L}(\delta)$ is $\leq$ the quantity \eqref{codeyy} plus
\begin{equation}\label{simplecase1}
\dim \Gr(r-\delta,r)+ (d'-d)n + |\sigma_L|-|\sigma_{J^1}|
\end{equation}
We also use $l_a\leq j^1_{a+\delta}$ for $a\leq r-\delta$ (this follows from the rank conditions).  This gives  $|\sigma_L|-|\sigma_{J^1}|\leq -\sum_{a=1}^{\delta} (n-r+a-j^1_a)$. Therefore the quantity \eqref{simplecase1} is $\leq$
$\sum_{a=1}^{\delta} (a-j^1_a-\delta)$. This number is $\leq -2$ except for when $\delta=1$ and $j^1_1=1$ when it is equal to $-1$.

\subsubsection{Schubert condition strengthens at a point of $S$}\label{bruck03}
We need to now consider the cases in which $\mv$ is a subbundle of $\mw$, but the Schubert condition at $p=p_1$ strengthens to a Schubert cell parametrized by a subset $J$ of $[n]$, $|J|=r$. Dimension counts still apply, and we see that the codimension drops by $1$ whenever $J$ is of the form $J=(I^1-\{a\})\cup \{a-1\}$ where $a>1$, $a\in I^1$, $a-1\not\in I^1$.
\subsubsection{Conclusion of the proof of Proposition \ref{maya}}\label{bruck04}
It remains to  show that any irreducible component  of $\Omega-\Omega^0$ has dimension less than the quantity \eqref{codeyy}.  Consider the generic point $(\mv,\mw,\me,\gamma)$ of an irreducible component $Y$.
\begin{enumerate}
\item Let $Q=\{q_1,\dots,q_m\}$ be the set of points (possibly empty) not in $S$ such that for $q\in Q$, $\mv_q\to \mw_q$ has kernel of rank $\epsilon(q)>0$.
\item Let $\delta: S\to \Bbb{Z}_{\geq 0}$ be such that the maps $\mv_p\to\mw_p$ have kernels $K(p)$ of rank $\delta(p)$, for $p\in S$. Assume further that the images $Q(p)$ of $\mv_p\to\mw_p$ are  $r-\delta(p)$ dimensional subspaces of $\mw_q$  in the Schubert cells corresponding to $L^p=\{l^p_1,\dots,l^p_{r-\delta(p)}\}$.
\end{enumerate}
Computing as before, the dimension of  $Y$ is $\leq$ the quantity \eqref{codeyy} plus
$$\sum_{q\in Q} \bigl(1-(n+\epsilon(q)-r)\epsilon(q)\bigr) + \sum_{p\in S} \sum_{a=1}^{\delta(p)} (a-j^p_a-\delta(p)).$$
Each of the $q$ contributions is $\leq -1$ by the computation in Section \ref{bruck01}. If $\delta(p)>0$, then the $p$ term is $\leq -1$ by the computation in Section \ref{bruck02}. Note that if there no $q$ terms and $\delta(p)=0$ for all $p$, then since $Y\cap\Omega^0=\emptyset$, we need $L^p\neq J^p$ for some $p$, and this makes the dimension of $Y$ strictly less than  the quantity \eqref{codeyy} as in Section \ref{bruck03}. This corresponds to one of the inequalities $l^p_a\leq j^p_{a+\delta(p)}$ in Section \ref{bruck02} becoming strict.

Therefore the dimension of $Y$ is strictly less than the quantity \eqref{codeyy}, and this completes the proof of Proposition \ref{maya}.
\subsection{Cycle classes}
Propositions \ref{laumon}, \ref{maya} and Lemma \ref{repres} give the following
\begin{corollary}
The representable morphism $\Omega(d,r,\mn,n,\vec{J})\to \Par_{n,\mn,S}$ has relative dimension
$\sum_{i=1}^s |\sigma_{J^i}| -(dn +\deg(\mn)r +r(n-r)).$
\end{corollary}
For proper morphisms of irreducible schemes $Y\to X$ of relative dimension $d$, one can form the push forward of the fundamental class $[Y]$  of $Y$, giving  a codimension $d$ cycle on $X$ \cite[Section 1.4]{FInt}. The push forward operation commutes with flat base changes by \cite[Proposition 1.7]{FInt}, and is therefore applicable in the setting of representable morphisms of irreducible Artin stacks. Applying it to the morphism $\Omega(d,r,\mn,n,\vec{J})\to \Par_{n,\mn,S}$, we obtain a cycle class on $\Par_{n,\mn,S}$:
\begin{defi}\label{cyclist}
The pushforward of the cycle $\Omega(d,r,\mn,n,\vec{J})$ to $\Par_{n,\mn,S}$ gives an effective cycle $C(d,r,\mn,n,\vec{J})$ (possibly zero) of codimension
 $\sum_{i=1}^s |\sigma_{J^i}| -(dn +\deg(\mn)r +r(n-r)).$
 \end{defi}
\section{Shift Operations}\label{shifty}
Let $p=p_i\in S$. We define a shift operation $\operatorname{Sh}_p: \Par_{n,\mn,S}\to \Par_{n,\mn(p),S}$,  as follows:
$$\operatorname{Sh}_p(\mw,\me,\gamma)=(\mw',\me',\gamma')$$
where
\begin{enumerate}
\item $\mw'$ coincides with $\mw$ away from $p$, and is defined in a neighbourhood of $p$ as follows: Let $t$ be a uniformizing parameter at $p$, $U$ an open subset containing $p$, and assume $t$ does not vanish on $U-\{p\}$. Sections $s$ of $\mw'$ on $U$ are meromorphic sections of $\mw$ of $U$, regular on $U-\{p\}$ such that $ts$ is a regular section of $\mw$
whose fiber at $p$ is in $E^p_1$. It is easy to see that $\mw'$ is independent of the choice of the uniformizing parameter $t$ at $p$. Clearly $\mw\subseteq\mw'$.
\item
The natural map $\det(\mw)\to \det(\mw')$ has a zero of order 1 at $p$, and hence $\det(\mw')$ is isomorphic to $\mn(p)$. Let $\gamma'$ denote this isomorphism.
\item Since $\mw$ and $\mw'$ are isomorphic outside of $p$, we can take $E'^{q}_{\bull}=E^q_{\bull}$ for $q\in S-\{p\}$. To complete the definition of $\me'$ we need to define $E'^p_{\bull}$.  Now $\mw_p\to \mw'_p$ has kernel $E^p_1$, so we define $E'^p_a$ to be the image of $E^p_{a+1}$ for $a<n$ and equal to  $\mw'_p$ for $a=n$.
\end{enumerate}

The exact sequence
$$0\to \mw \to \mw'\to \mw'_p/E'^p_{n-1}\to 0$$
allows us to define the inverse $\operatorname{ISh}_p$ to $\operatorname{Sh}_p$.
\subsection{The effect of the shift operation on line bundles}
Writing
$\operatorname{Sh}_p(\mw,\me,\gamma)=(\mw',\me',\gamma')$
we have an exact sequence
\begin{equation}\label{roast1}
0\to \mw'\to \mw(p)\leto{t} E^p_n/E^p_1\to 0
\end{equation}
which we make independent of the choice of $t$ by twisting by the canonical bundle:
\begin{equation}\label{roast}
0\to \mw'\to \mw(p)\to(E^p_n/E^p_1)\tensor K_p^*\to 0
\end{equation}
\begin{lemma}\label{phase}
$E^p_a/E^p_{a-1}=E'^p_{a-1}/E'^p_{a-2}$ if $2\leq a\leq n$, and $E^p_1/E^p_0=(E'^p_n/E'^p_{n-1})\tensor K_p$.
\end{lemma}

\begin{proposition}\label{easily}
Let $p=p_i\in S$.
\begin{enumerate}
\item[(a)] $\operatorname{Sh}_p^*(\mathcal{B}(\vec{\lambda},\ell))= \mathcal{B}(\vec{\mu},\ell)$
where $\mu^j=\lambda^j$ for $j\neq i$ and $\mu^i_a= \lambda^i_{a-1}$ if $2\leq a\leq n$, and $\mu^i_1=\lambda^i_{n}+\ell$.
\item[(b)] $\operatorname{ISh}_p^*(\mathcal{B}(\vec{\lambda},\ell))= \mathcal{B}(\vec{\nu},\ell)$
where $\nu^j=\lambda^j$ for $j\neq i$ and $\nu^i_a= \lambda^i_{a+1}$ if $1\leq a< n$, and $\nu^i_{n}=\lambda^i_{1}-\ell$.

\item[(c)]The morphism $\operatorname{Sh}_p: \Par_{n,\mn,S}\to \Par_{n,\mn(p),S}$ induces bijections $\Pic^+(\Par_{n,\mn(p),S})\leto{\sim}\Pic^+(\Par_{n,\mn,S})$ and
$\Pic^+_{\Bbb{Q}}(\Par_{n,\mn(p),S})\leto{\sim}\Pic^+_{\Bbb{Q}}(\Par_{n,\mn,S})$.
\item[(d)] $\ml\in  \Par_{n,\mn(p),S}$ is an F-line bundle if and only if $\operatorname{Sh}_p^*\ml$ is an F-line bundle on  $\Par_{n,\mn,S}$.
\end{enumerate}
\end{proposition}
\begin{proof}
$\operatorname{Sh}_p$ and $\operatorname{ISh}_p$ are isomorphisms (they are inverses of each other). This proves (c). Lemma \ref{phase}, and \eqref{roast1}  (and Definition  \ref{levell}) imply parts (a) and (b). \end{proof}

\begin{defi}\label{grade} Define maps: $\operatorname{Gr}:\Pic (\Par_{n,\mn,S})\to \Bbb{Z}/n\Bbb{Z}$ by
$(\vec{\lambda},\ell)\mapsto \sum_{i=1}^s|\lambda_i|+ \ell \deg\mn \in\Bbb{Z}/n\Bbb{Z}$. Note that shift operations preserve this grading.  It is also easy to verify that the grade of an effective line bundle is zero (reduce to the case $\deg\mn=0$ using shift operations).
\end{defi}

\begin{remark}\label{able}
Recall the Killing form isomorphism $\frh_n^*\to \frh_n$ from Section \ref{loudvoice}. In case (a) of Proposition \ref{easily}, we note that $\kappa(\frac{{\mu^j}}{\ell}),\kappa(\frac{\lambda^j}{\ell})\in \Delta_n$ give the same conjugacy class in $\op{SU}(n)$ if $j\neq i$ and $\kappa(\frac{{\mu^i}}{\ell})$ equals  $\zeta_n^{-1}\cdot \kappa(\frac{{\lambda^i}}{\ell})$ (see Definition  \ref{symmetry}). In case (b),
we have  $\kappa(\frac{{\nu^j}}{\ell})=\kappa(\frac{{\lambda^j}}{\ell})$ if $j\neq i$ and $\kappa(\frac{{\nu^i}}{\ell})=\zeta_n\cdot\kappa(\frac{{\lambda^i}}{\ell})$.
\end{remark}
\begin{lemma}\label{threepointone}
Tensoring $\mw$ with $\mathcal{O}(p)$ sets up an isomorphism $\Par_{n,\mn,S}\to \Par_{n,\mn(np),S}$. This isomorphism coincides with an $n$-fold composition of
$\operatorname{Sh}_p$. This isomorphism carries the line bundle $\mathcal{B}(\vec{\lambda},\ell)$ on one stack
to a line bundle with the same parameters $\mathcal{B}(\vec{\lambda},\ell)$ on the other.
\end{lemma}
\subsubsection{Proof of Proposition \ref{LaSo}}\label{shifty2}
Proposition \ref{LaSo}, which describes the Picard group of $\Par_{n,\mn,S}$ follows from (a), (b), (c) of Proposition \ref{easily} and the special case $\mn=\mathcal{O}$ proved in \cite{LS}.
\subsection{The effect of shift operations on cycles}
Let $p=p_i\in S$. To study the effect of the shift operations on the stacks $\Omega^0(d,r,\mn,n,\vec{J})$ defined in Definition  \ref{mindful}, we begin with a basic commutative square:
\begin{equation}\label{basicD}
\xymatrix{
 \Omega^0(d,r,\mn,n,\vec{J})\ar[d]\ar[r]^{\operatorname{Sh}_p}  & \Omega^0(d',r,\mn(p),n,\vec{K})\ar[d]\\
\Par_{n,\mn,S}\ar[r]^{\operatorname{Sh}_p}  & \Par_{n,\mn(p),S}
}
\end{equation}
where the terms $d'$ and $\vec{K}$ are determined as follows,
 \begin{enumerate}
 \item $d'=d$ if $1\notin J^i$, and $d'=d-1$ if $1\in J^i$
 \item $K^k=J^k$ for $k\neq i$.
 \item $K^i =\{j^i_1-1,j^i_2-1,\dots, j^i_r-1\}$ if $1\notin J^i$, and $K^i =\{j^i_2-1,\dots, j^i_r-1,n\}$ if $1\in J^i$.
 \end{enumerate}

To see that we note that there is a bijective correspondence between subbundles of bundles $\mw$ and $\mw'$ when we have a relation $\Sh_p(\mw,\me,\gamma)=(\mw',\me',\gamma')$. Keeping track of the Schubert cell data gives Diagram \eqref{basicD}.

Recall the cycles $C(d,r,\mn,n,\vec{J})$ on $\Par_{n,\mn,S}$ defined as pushforwards of $\Omega(d,r,\mn,n,\vec{J})$ (Definition  \ref{cyclist}).
 \begin{lemma}\label{newS} Let $p=p_i\in S$. Then
 $$\operatorname{Sh}_{p}^*C(d',r,\mn(p),n,\vec{K})=C(d,r,\mn,n,\vec{J}).$$
 \end{lemma}
\begin{proof}
 The operation $\op{Sh}_p$  lifts to a map  $\Omega^0(d,r,\mn,n,\vec{J})\to \Omega^0(d',r,\mn(p),n,\vec{K})$ (as in \eqref{basicD}), but  may not lift to a map of the closures $\Omega(d,r,\mn,n,\vec{J})\to \Omega(d',r,\mn(p),n,\vec{K})$. Since the multiplicities of the divisor
$C(d',r,\mn(p),n,\vec{K})$ can be calculated at generic points, we obtain the stated result from Diagram \eqref{basicD} and Proposition \ref{maya}.
\end{proof}
 \subsection{Generalized Gromov-Witten numbers}\label{stuffed}
 Definition \ref{GW} is an enumerative count of subbundles on $\mathcal{O}^{\oplus n}$, which is the ``general" rank $n$ vector bundle with trivial determinant. We can replace the role of $\mathcal{O}^{\oplus n}$ by a general rank $n$ vector bundle
with arbitrary determinant $\mathcal{N}$ of degree $-D$, \cite[Definition 2.4]{BHorn}:
\begin{defi}\label{GWgen}
Let $I^1,\dots I^s$ be subsets of $[n]$ each of cardinality $r$ and $d\in\Bbb{Z}$. Let $\mn$ be a line bundle on $\Bbb{P}^1$
of degree $-D$.  Let $\mw$ be a general rank $n$ bundle of degree $-D$ (hence $\det\mw=\mn$).  Let $\me\in\Fl_S(\mw)$ be a general point.

The Gromov-Witten number $\langle\sigma_{I^1}, \sigma_{I^2},\dots,\sigma_{I^s}\rangle_{d,D}$
is defined to be the number of subbundles (count as zero if infinite) $\mv\subseteq \mw$ of degree $-d$ and rank $r$  such that for $p\in S$,
$\mv_p\in\Omega_{I^j}(E^p_{\bullet})\subseteq \Gr(r,\mw_p)$. It follows from the dimension formula for quot schemes that
 if $\langle\sigma_{I^1}, \sigma_{I^2},\dots,\sigma_{I^s}\rangle_{d,D}$ is non-zero, then
 \begin{equation}\label{numerical}
 \sum_{j=1}^s |\sigma_{I^j}| =r(n-r)+dn-Dr.
 \end{equation}

\end{defi}
These new Gromov-Witten numbers reduce to the usual Gromov-Witten numbers using shift operations by the following:
Assume that the  codimension of the cycle $C(d,r,\mn,n,\vec{J})$ in $\Par_{n,\mn,S}$ is $0$, then in the setting of Lemma \ref{newS}, 
$$\Omega(d,r,\mn,n,\vec{J},S)\to \Par_{n,\mn,S}, \ \ \Omega(d',r,\mn(p),n,\vec{K},S)\to \Par_{n,\mn(p),S},$$
 have the same degree. Let $D=-\deg(\mn)$. We then have by Lemma \ref{newS},  
\begin{equation}\label{formula1}
\langle\sigma_{J^1}, \sigma_{J^2},\dots,\sigma_{J^s}\rangle_{d,D}= \langle\sigma_{K^1}, \sigma_{K^2},\dots,\sigma_{K^s}\rangle_{d',D-1}.
\end{equation}
Also note the formula
\begin{equation}\label{formula2}
\langle\sigma_{J^1}, \sigma_{J^2},\dots,\sigma_{J^s}\rangle_{d+r,D+n}=\langle\sigma_{J^1}, \sigma_{J^2},\dots,\sigma_{J^s}\rangle_{d,D}
\end{equation}
 which is obtained by an $n$-fold application of the shift operation at any $p_i$.

\subsection{Quantum cohomology of Grassmannians and ranks of conformal blocks}\label{QCB}
There is an expression for  $\langle\sigma_{I^1}, \sigma_{I^2},\dots,\sigma_{I^s}\rangle_{d,D}$ (with notation from Definition \ref{GWgen}) as equal to  a suitable  $h^0(\Par_{r,\mathcal{O}(-d),S},\ml)$ \cites{Gepner,Witten} (and Section \ref{QCB} below). The formula in general (even when $D=0$) has some shifting. The most succinct way of stating it is to reduce to $d=0$ (using \eqref{formula1}) and then using the following: Assume $d=0$ in Definition \ref{GWgen} and hence $\sum_{j=1}^s |\sigma_{I^j}|$ is divisible by $r$.
Let $\lambda^1,\dots,\lambda^s$ be highest weights for $\op{SL}(r)$ given by $\lambda^i=\lambda(I^j)$ following Definition \ref{correspondence} below. Then  (see \cite[Proposition 3.4]{BHorn}),
\begin{lemma}\label{newlabel2}
Suppose $d=0$ in Definition \ref{GWgen}. Then
$$\langle\sigma_{I^1}, \sigma_{I^2},\dots,\sigma_{I^s}\rangle_{d,D}= h^0(\Par_{r,\mathcal{O},S},\ml)$$ where $\ml=\mathcal{B}(\vec{\lambda},n-r)$.
\end{lemma}
Note that the $h^0$ above is also the rank of a vector space of conformal blocks for the Lie algebra $\mathfrak{sl}_r$ and the representations $\vec{\lambda}$ placed at the points of $S$ at level $n-r$.
\begin{defi}\label{correspondence}
Given a subset $I=\{i_1<i_2<\dots<i_r\}\subseteq [n]=\{1,\dots,n\}$, $|I|=r$, let $\lambda(I)=(\lambda_1,\dots,\lambda_r)$ where
$\lambda_a=n-r+a-i_a,\ a=1,\dots,r.$
\end{defi}
If $d\neq 0$, there is a similar equality. Let $\lambda^j=\lambda(I^j)$, and $\ml=\mathcal{B}(\vec{\lambda},n-r)$ on $\Par_{r,\mathcal{O}(-d),S}$. The assumption \eqref{numerical} shows that the grade of $\ml$ is zero (see Definition  \ref{grade}) and using the shift operations and Lemma \ref{newlabel2} we obtain
\begin{equation}\label{newlabel3}
\langle\sigma_{I^1}, \sigma_{I^2},\dots,\sigma_{I^s}\rangle_{d,D}= h^0(\Par_{r,\mathcal{O}(-d),S},\ml).
\end{equation}
One can use shift operations to reduce the $h^0$ to the rank of a suitable vector space of conformal blocks.

\subsection{Quantum saturation and Fulton conjectures}
We recall quantum saturation \cite{BHorn} and a quantum generalization of a conjecture of Fulton \cite[Remark 8.5]{BKq}. These (now theorems) generalize statements in the classical invariant theory of the special linear group \cite{KT,KTW}.

\begin{proposition}\label{qTheorems}
Suppose $\mathcal{B}(\vec{\lambda},\ell)$ is a line bundle on  $\Par_{n,\mn,S}$ of grade zero, where $\lambda^i, i=1,\dots, s$ are dominant integral weights for $\op{SL}(n)$ and $\ell>0 $. Then
\begin{enumerate}
\item For any positive integer $N$, $H^0(\Par_{n,\mn,S},\ml^N)\neq 0$ if and only if
$H^0(\Par_{n,\mn,S},\ml)\neq 0$. Therefore $H^0(\Par_{n,\mn,S},\ml)\neq 0$ if and only if
there is a semi-stable point in $\Par_{n,\mn,S}$ for the polarization given by $\ml$.
\item For any positive integer $N$, $h^0(\Par_{n,\mn,S},\ml^N)=1$ if and only if
$h^0(\Par_{n,\mn,S},\ml)=1$.
\end{enumerate}
\end{proposition}
\begin{proof}
The shift operations reduce (1) to the case $\mn=\mathcal{O}$, this case of (1) follows from
quantum saturation \cite[Theorem 1.8]{BHorn}.

For (2), we can reduce to the case $\mn=\mathcal{O}$, which follows from  \cite[Remark 8.5]{BKq}. Proposition \ref{fulC} below also contains a proof of the quantum Fulton conjecture. See \cite{CS} for a recent discussion of the quantum saturation and Fulton conjectures.
\end{proof}

\begin{proposition}\label{kappa}
Suppose $\mathcal{B}(\vec{\lambda},\ell)$ is a line bundle on  $\Par_{n,\mn,S}$ of grade zero where $\lambda^i, i\in[s]$ are dominant integral weights for $\op{SL}(n)$ and $\ell> 0$. Let $D=-\deg(\mn)$. Let $p=p_i\in S$. Define a point $\vec{a}= \Delta_n^s$ as follows: $a^j=\kappa(\frac{\vec{\lambda_j}}{\ell})$  if $j\neq i$ and $a^i=\zeta^D_n\cdot \kappa(\frac{\vec{\lambda_i}}{\ell})$ (see Definition  \ref{symmetry}). Then,
\begin{enumerate}
\item  $\mathcal{B}(\vec{\lambda},\ell)$ is effective if and only $\vec{a}\in P_n(s)$, i.e., there exist matrices $A_i\in \op{SU}(n)$ conjugate to $a^i$, $i=1,\dots,s$ with product $A_1 A_2\cdots A_s=I_n$.
\item  $h^0(\Par_{n,\mn,S},\ml)=1$ if and only if there is a unique $A_i\in \op{SU}(n)$ conjugate to $a^i$, $i=1,\dots,s$  with product $A_1A_2\cdots A_s=I_n$. Here uniqueness means that if $(A'_1,\dots,A'_s)$ is another such collection, there exists a $C\in \op{SU}(n)$ with $A_i'=CA_iC^{-1}$ for $i=1,\dots,s$.
\end{enumerate}
\end{proposition}
\begin{proof}
Suppose $-D=\deg(\mn)\leq 0$. Then $\operatorname{ISh}_p^D: \Par_{n,\mn(D),S}\to \Par_{n,\mn(D),S}$. Let  $\mathcal{B}(\vec{\nu},\ell)\in \Pic^+(\Par_{n,\mn(D),S})$ be the pullback of $\mathcal{B}(\vec{\lambda},\ell)$.  Since $\mn(D)$ is trivial and of grade zero, (1) follows from Proposition \ref{blacktea}, quantum saturation \cite{BHorn}(see Proposition \ref{qTheorems}(1)), and Remark \ref{able}. The case $D<0$ is handled similarly using  $\operatorname{Sh}_p$ instead of  $\operatorname{ISh}_p$.

For (2), we similarly reduce to $D=0$. In this case if $h^0(\Par_{n,\mn,S},\ml)=1$ then by quantum Fulton (Theorem \ref{qTheorems}(2)), $h^0(\Par_{n,\mn,S},\ml^N)=1$ for all $N$ and hence the Mehta-Seshadri moduli space is a point. This shows the uniqueness of the $A_i$. To go the other way, note that  $h^0(\Par_{n,\mn,S},\ml)$ is the space of global sections of  a line bundle over a point and is hence one dimensional.
\end{proof}

The $\kappa$ operation introduces some unnecessary denominators. In fact working in the unitary group $U(n)$, there is an equivalent formulation where the local monodromies are $\ell$th roots of unity. Here $\ell$ is the level of the corresponding line bundle:

\begin{corollary}\label{sansD}
Suppose  $\ml=\mathcal{B}(\vec{\lambda},\ell)$ is a line bundle of grade zero on  $\Par_{n,\mn,S}$ where $\lambda^i, i\in[s]$ are dominant integral weights for $\op{SL}(n)$ and $\ell>0$. Write  $\sum_{i=1}^s|\lambda_i|+ \ell \deg\mn=nu$, $u\in \Bbb{Z}$, Then, $\mathcal{B}(\vec{\lambda},\ell)$
is effective on   $\Par_{n,\mn,S}$  if and only if there exist matrices $A_1,\dots,A_s,C\in \op{U}(n)$ with $A_1\cdots A_s=C$, where
\begin{enumerate}
\item $A_i$ is conjugate to the diagonal matrix with entries $\exp(2\pi\sqrt{-1}\lambda^i_j/\ell)$ for $j=1,\dots,n$.
\item $C$ is the central matrix $c I_n$ with $c=\exp(2\pi\sqrt{-1}u/\ell)$.
\end{enumerate}
Moreover,  $h^0(\Par_{n,\mn,S},\ml)=1$ if and only if the corresponding representation variety of unitary local systems  on $\Bbb{P}^1-(S\cup\{p_{s+1}\})$, $p_{s+1}\not\in S$ (with local monodromies $A_1,\dots,A_s, C^{-1}$) is a point (this implies that there is a unique (possibly block reducible) solution to the matrix equation above in $U(n)$ up to conjugation).
\end{corollary}
\begin{proof}
By Proposition \ref{kappa}, $\mathcal{B}(\vec{\lambda},\ell)$
is effective on   $\Par_{n,\mathcal{O},S}$  if and only if there exist matrices $B_1,\dots,B_s\in \op{SU}(n)$ with $B_1\cdots B_s=\zeta_n^{\deg \mn}I_n$, where $A_i$ is conjugate to the diagonal matrix with entries $\exp(2\pi\sqrt{-1}\bigl(\lambda^i_j/{\ell}-\sum |\lambda^i|/(n\ell)\bigr))$ but
$\sum |\lambda^i|/(n\ell)=u/\ell - \deg\mn/n$. This proves the desired statement.

\end{proof}

\subsection{F-vertices and F-line bundles}
\begin{lemma}\label{elgar}
Suppose $\ml=\mathcal{O}(E)$  is an F-line bundle on $\Par_{n,\mathcal{O},S}$. Then
\begin{enumerate}
\item $\ml$ gives an extremal ray $\Bbb{Q}_{\geq 0}\mathcal{O}(E)$   of $\Pic^+_{\Bbb{Q}}(\Par_{n,\mathcal{O},S})$, which  lies on some regular face of  $\Pic^+_{\Bbb{Q}}(\Par_{n,\mathcal{O},S})$.
\item $\ml$ gives a Hilbert basis element of the additive semigroup $\Pic^+(\Par_{n,\mathcal{O},S})$.
\end{enumerate}
 \end{lemma}
\begin{proof}
The extremal ray property in (1)  follows from  the method of  \cite[Lemma 2.1]{BHermit}: Suppose
$\ml^m$ is a tensor of two effective line bundles for some positive integer $m$. So we may write
$\ml^m=\mathcal{O}(D_1+D_2)$ where $D_1$ and $D_2$ are effective divisors on the smooth stack $\Par_{n,\mathcal{O},S}$. Since $\ml^m=\mathcal{O}(mE)$, and $\ml^m$ has only one global section up to scalars, we should have $mE=D_1+D_2$ as effective divisors. Since $E$ is irreducible, we see that $D_1$ and $D_2$ are multiples of $E$, and hence the extremal ray property follows. By Proposition \ref{b9}(2), and Lemma \ref{pinky}, $\ml$ lies on some regular face of
$\Pic^+_{\Bbb{Q}}(\Par_{n,\mathcal{O},S})$. Since $\ml$ is an F-line bundle, it is indivisible as an effective line bundle, and so (2) follows from (1).

\end{proof}

\begin{lemma}\label{corresp}
There is a bijective correspondence between F-line bundles on $\Par_{n,\mathcal{O},S}$, and F-vertices in $P_n(s)$. The correspondence takes $\mathcal{B}(\vec{\lambda},\ell)$ to $\vec{a}=(\kappa(\lambda^1/\ell),\dots,\kappa(\lambda^s/\ell))$. For the reverse correspondence a vertex $\vec{a}=(a^1,\dots,a^s)$ corresponds to $\mathcal{B}(\vec{\lambda},\ell)$ so that
\begin{enumerate}
\item  $(\kappa(\lambda^1/\ell),\dots,\kappa(\lambda^s/\ell))=\vec{a}$.
\item $\mathcal{B}(\vec{\lambda},\ell)$ is of grade zero (i.e., $\sum |\lambda^i|$ divisible by $n$).
\item The $\ell>0$ above is the smallest possible.
\end{enumerate}
(The condition on $\mathcal{B}(\vec{\lambda},\ell)$ is that it is indivisible as a grade zero line bundle.)
\end{lemma}
\begin{proof}
Let $\ml=(\vec{\lambda},\ell)$ and $\vec{a}$ satisfy the relation $\vec{a}=(\kappa(\lambda^1/\ell),\dots,\kappa(\lambda^s/\ell))$. Then the representation variety of unitary representations with local monodromies given by $\vec{a}$ is identified with the Proj of the graded ring $\oplus_{N=0}^{\infty} H^0(\Par_{n,\mathcal{O},S}, \mathcal{B}(\vec{\lambda},\ell)^N)$. If $\mathcal{B}(\vec{\lambda},\ell)$
is an F-line bundle, then this Proj is a point.  By Lemma \ref{elgar} (1), $\ml$ gives an extremal ray  of $\Pic^+_{\Bbb{Q}}(\Par_{n,\mathcal{O},S})$. This shows the forward direction of the correspondence (also use Proposition \ref{b9}).

For the reverse direction, note that the $\mathcal{B}(\vec{\lambda},\ell)$ described in the statement is effective by quantum saturation (Proposition \ref{qTheorems}). Since (by construction) it is not a multiple of an effective line bundle, the reverse direction follows from Proposition \ref{b9}.
\end{proof}
\begin{remark}\label{inpractice}
Note that ``in practice'', to find an F-line bundle corresponding to a vertex $\vec{a}$ we write $a^j=b^j/\tilde{\ell}$ for a common denominator $\tilde{\ell}$, so that $b^j\in \Bbb{Z}^n$ for all $j\in[s]$. Now let $\nu^j=b^j-b^j_n(1,\dots,1)$.
The line bundle $\mathcal{B}(\vec{\nu},\tilde{\ell})$ has grade zero and $(\kappa(\nu^1/\tilde{\ell}),\dots,\kappa(\nu^s/\tilde{\ell}))=\vec{a}$. But it may not be indivisible by this property. We then divide $(\vec{\nu},\tilde{\ell})$ by a positive integer $k$ so that it is indivisible of grade zero. The desired pair $(\lambda,\ell)$ is then $\frac{1}{k}(\vec{\nu},\tilde{\ell})$.
\end{remark}

\begin{lemma}\label{divisorE}
Let $\ml=\mathcal{B}(\vec{\lambda},\ell)$ be an F-line bundle on $\Par_{n,\mathcal{O},S}$. Write it as $\mathcal{O}(E)$ for a reduced irreducible divisor  $E\subseteq \Par_{n,\mathcal{O},S}$.
Then $E$ is equal to an effective cycle $C(d,r,\mathcal{O},n,\vec{J})$ (Definition  \ref{cyclist})  of codimension one i.e., $dn  +r(n-r)-\sum_{i=1}^s |\sigma_{J^i}|=1$, with the further property that
$\sum_{i=1}^s \sum_{k\in J^i} \lambda^i_k> d\ell +\sum_{i=1}^s \frac{r|\lambda^i|}{n}$.
\end{lemma}
\begin{proof}
Any (generic) point $(\mw,\me,\gamma)$ of $E$ is unstable for $\ml$. Here we have used the quantum Fulton conjecture (Proposition \ref{qTheorems} (2)) that $h^0(\Par_{n,{\mathcal{O}},S},\ml^N)=1$ for all $N>1$, and hence all sections of $H^0(\Par_{n,{\mathcal{O}},S},\ml^N)$ vanish on $E$.

Let the Harder-Narasimhan maximal contradictor of semistability at a (generic) point $(\mw,\me,\gamma)$ of $E$ be a subbundle $\mv$ of the universal bundle $\mw$ on  $\Par_{n,{\mathcal{O}},S}$ of rank $r$, and suppose it meets the flags of the universal bundle in points corresponding to the Schubert cells parametrized by subsets $J^i\subseteq[n]$, $|J^i|=r$. That is, assume  $\mv_p\in \Omega^0_{J^i}(\me_p)\subseteq \Gr(r,\mw_p)$ for all $p=p_i\in S$.  Let $d=-\deg(\mv)$.

It is now easy to see that $E=C(d,r,\mathcal{O},n,\vec{J})$, and the  codimension of this class is one, i.e., $dn+r(n-r)-\sum_{i=1}^s |\sigma_{J^i}|=1$. This is because the stack $\Omega(d,r,\mathcal{O},n,\vec{J})$ is irreducible and its image under $\pi:\Omega(d,r,\mathcal{O},n,\vec{J})\to
\Par_{n,{\mathcal{O}},S}$ is destabilizing for $\ml$, and hence $\pi$ maps $\Omega(d,r,\mathcal{O},n,\vec{J})$ into $E$. The uniqueness of Harder-Narasimhan maximal contradictors of semistability then shows that $\pi:\Omega(d,r,\mathcal{O},n,\vec{J})\to E$ is generically one-to-one. Therefore $E$ is the cycle theoretic pushforward of $\Omega(d,r,\mathcal{O},n,\vec{J})$. The last inequality is the violated semistability inequality.
\end{proof}
The following gives a more practical approach towards listing all F-line bundles (hence F-vertices of $P_n(s)$ by Lemma \ref{corresp}) on $\Par_{n,\mathcal{O},S}$ in that we do not have to solve systems of linear inequalities.
\begin{proposition}\label{practical}
Consider a line bundle  $\ml=\mathcal{B}(\vec{\lambda},\ell)$ which is indivisible as a grade zero line bundle on $\Par_{n,\mathcal{O},S}$. Then $\ml$  is an F-line bundle  iff there is a cycle $E=C(d,r,\mathcal{O},n,\vec{J})$ (Definition  \ref{cyclist})  of codimension one i.e., $dn  +r(n-r)-\sum_{i=1}^s |\sigma_{J^i}|=1$ so that the following conditions hold:
\begin{enumerate}
\item[(i)] $\ml=\mathcal{O}(E)$. Note that we have formulas for $\mathcal{O}(E)$ from Proposition  \ref{enumerative}.
\item[(ii)] $\sum_{i=1}^s \sum_{k\in J^i} \lambda^i_k> d\ell +\sum_{i=1}^s \frac{r|\lambda^i|}{n}.$
\end{enumerate}
Therefore one can list all F-line bundles on $\Par_{n,\mathcal{O},S}$ by listing all cycles $E=C(d,r,\mathcal{O},n,\vec{J})$, computing $\mathcal{O}(E)=\mathcal{B}(\vec{\lambda},\ell)$ from Proposition  \ref{enumerative}, checking if $\mathcal{O}(E)$ satisfies the strict inequality in (ii), and is indivisible as a grade zero line bundle.
\end{proposition}
\begin{proof}
If $\ml=\mathcal{B}(\vec{\lambda},\ell)$ is an F-line bundle then the cycle $C(d,r,\mathcal{O},n,\vec{J})$ exists by the proof of Proposition  \ref{divisorE}. For the reverse implication note that $\mathcal{O}(E)$ is non-semistable on the support of the cycle $C(d,r,\mathcal{O},n,\vec{J})$ (the support is irreducible because it is the image of the irreducible $\Omega(d,r,\mathcal{O},n,\vec{J})$) by (ii), hence any global section of $\ml$ vanishes on this support. The desired assertion follows since  $E$ is not a multiple of an effective divisor by the indivisibility property.
\end{proof}

\section{Strange duality}\label{SD}
\subsection{Proof of Theorem \ref{egregium}}
 We show how to  associate a level $\ell$ representation $\lambda$ of $\op{SL}(n)$ to  a conjugacy class in $\op{GL}(\ell,\Bbb{C})$ of order dividing $n$:
\begin{defi} \label{reprem2}
Suppose $\bar{A}\in \op{GL}(\ell,\Bbb{C})$ such that  $A^n=I_{\ell}$.   We can find an $\ell\times n$ Young diagram $\mu$, with $n> \mu_1\geq \mu_2\geq \dots \geq \mu_{\ell}\geq 0$, such that $\bar{A}$ is conjugate to the diagonal matrix with entries $\exp(2\pi\sqrt{-1}\mu_i/n)$, for $i=1,\dots,\ell$.

  Let $\lambda=\lambda(\bar{A})=\mu^T$, the transpose Young diagram of $\mu$, which fits in a box of size $n\times \ell$. Note that $\lambda$ is dominant integral weight of $\op{SL}(n)$ of level $\ell$, i.e., $\lambda_1-\lambda_n \leq \ell$, and also  $\lambda_n=0$, which is  normalisation condition on $SL(n)$ representations (see Sections \ref{loudvoice} and \ref{loudvoice2}).
\end{defi}

\begin{lemma}\label{reprem}Let $\ell$ and $n$ be fixed. The correspondence $\bar{A}\to \lambda(\bar{A})$ in Definition \ref{reprem2} gives a bijective correspondence between the  set of semisimple conjugacy classes  $\bar{A}\in\op{GL}(\ell,\Bbb{C})$ with $\bar{A}^n=I_{\ell}$,  and the set of  level $\ell$ representations of $\op{SL}(n)$.
\end{lemma}
\begin{proof}
The reverse correspondence is the following: A  level $\ell$ representation of $\op{SL}(n)$
gives a Young diagram  $\lambda$ that fits in an $n\times \ell$ box with $\lambda_n=0$ (the  normalisation condition).   Let $\mu=\lambda^T$.  Clearly $\mu_1<n$ since $\lambda_n=0$. The conjugacy class corresponding to $\lambda$ is then the diagonal matrix with entries $\exp(2\pi\sqrt{-1}\mu_i/n)$, for $i=1,\dots,\ell$.
\end{proof}

Let $\ell$ and $n$ be fixed. Applying Lemma \ref{reprem}, coordinate by coordinate, we find that the following sets are in bijective correspondence:
\begin{enumerate}
\item[(A)]The set of $s$-tuples of semisimple  conjugacy classes $\mathcal{A}=(\bar{A}_1,\dots,\bar{A}_s)$ in $\op{GL}(\ell,\Bbb{C})$, with $\bar{A}_i^n=1$, such that  $\prod_{i=1}^s \det \bar{A}_i =1$.
\item[(B)] The set of grade zero bundles $\mathcal{B}(\vec{\lambda},\ell)$  of level $\ell$ on $\Par_{n,\mathcal{O},S}$, i.e., $\vec{\lambda}$ is a tuple of  dominant integral weights  for $\op{SL}(n)$ of level $\ell$, and  $n$ divides $\sum_{i=1}^s|\lambda^i|$.
\end{enumerate}
The correspondence between (A) and (B) is 
\begin{equation}\label{corr2}
\mathcal{A}\mapsto (\lambda^1,\dots,\lambda^s),\ \ \lambda^i=\lambda(\bar{A}_i), \ i=1,\dots,s
\end{equation}
(also see Remark \ref{newRem} for the compatibility with the point $v(\mathcal{A})$ from Definition \ref{finiteness}). 

The condition on  the grade in  (B) is that $n$ divides $\sum_{i=1}^s|\lambda^i|$. Since $|\mu|=|\mu^T|$ for a Young diagram $\mu$, this is the same condition, under the correspondence, as the determinant condition in (A).
\begin{remark}\label{newRem}
In the above correspondence \eqref{corr2}, the point $v(\mathcal{A})\in \Delta_n^s$ from Definition \ref{finiteness} is
$(\kappa(\lambda^1/\ell),\dots,\kappa(\lambda^s/\ell))$ (see the Killing form isomorphism $\kappa$ from Section \ref{loudvoice}, and Theorem \ref{blacktea}).

\end{remark}

The following numerical strange duality \eqref{GepnerPP} is a direct consequence of results of Witten and Gepner \cite{Gepner,Witten}. The full strange duality asserts that the vector spaces in \eqref{GepnerPP} are in fact  canonically dual to each other (see Section \ref{understanding}), but we will not need this fact.
\begin{proposition}\label{BWag}
Suppose $\ml=\mathcal{B}(\vec{\lambda},\ell)$ is grade zero line bundle on $\Par_{n,\mathcal{O},S}$ with $\ell>0$, $\vec{\lambda}$ an $s$-tuple of dominant integral weights, which are normalised, i.e., $\lambda^i_n=0, i\in[s]$. Write $\sum_{i=1}^s|\lambda^i|=nu$. Let  $\widetilde{\mn}$ be a line bundle of degree $-u$ on $\Bbb{P}^1$. Let $\ml'$ be the grade zero line bundle $\mathcal{B}(\vec{\lambda}^T,n)$ on $\Par_{\ell,\widetilde{\mathcal{N}},S}$. Then,
\begin{equation}\label{GepnerPP}
h^0(\Par_{n,\mathcal{O},S},\ml)=h^0(\Par_{\ell,\widetilde{\mathcal{N}},S},\ml')
\end{equation}
\end{proposition}
\begin{proof}
The rank of $H^0(\Par_{n,\mathcal{O},S},\ml)$ equals a generalized Gromov-Witten number $\langle\sigma_{I^1}, \sigma_{I^2},\dots,\sigma_{I^s}\rangle_{0,D}$ with $D=\ell-u$, $I^1,\dots,I^s$ are subsets of $[n+\ell]$ each of cardinality $n$ corresponding to the partitions $\vec{\lambda}$ (see Definition  \ref{correspondence} and Lemma \ref{newlabel2}).

Now interpret the GW count as a count of subbundles of rank $\ell$ of degree $D$ of a generic bundle of degree $D$ by forming the (ordinary) dual of the quotient $(\mw/\mv)^*\subseteq \mw^*$ above. Using \eqref{newlabel3}, we  write the numbers above as $h^0(\Par_{\ell,\mathcal{N}',S},\mathcal{B}(\vec{\mu},n))$, with $\deg(\mathcal{N}')=D=\ell-u$. Using Lemma \ref{threepointone}, these numbers are equal to $h^0(\Par_{\ell,\widetilde{\mathcal{N}},S},\ml')$, which finishes the proof.


\end{proof}

\begin{proposition}\label{schubertt}
The following are equivalent in the correspondence  between (A) and (B):
\begin{enumerate}
\item There is a unitary representation of $\pi_1$ with  local monodromy data $\mathcal{A}$.
\item The line bundle $\mathcal{B}(\vec{\lambda},\ell)$ on $\Par_{n,{\mathcal{O}},S}$ is effective.
\item $v(\ma)$ is a point of $P_n(s)$ (see Definition \ref{finiteness}).
\end{enumerate}

\end{proposition}
\begin{proof}
As in Prop \ref{BWag}, write $\sum|\mu^i|=nu$. Let $\widetilde{\mathcal{N}}$ be a line bundle of degree $-u$ on $\Bbb{P}^1$.
There is a unitary representation with  local monodromy data $\mathcal{A}$ if and only if $\ml'=\mathcal{B}(\vec{\mu},n)$ is effective on the moduli stack  $\Par_{\ell,\widetilde{\mathcal{N}},S}$ (use Corollary \ref{sansD} with $n$ and $\ell$ interchanged): We have an equation
$\sum|\mu^i| +n\deg \widetilde{\mathcal{N}} =0\cdot\ell$.

By numerical strange duality, i.e., equality \eqref{GepnerPP} in Proposition \ref{BWag},  $\ml'$ is effective if and only if $\ml$ on  $\Par_{n,{\mathcal{O}},S}$ is effective. Since the grade of $\ml$ is zero, it is effective if and only if $(\kappa(\lambda^1/\ell),\dots,\kappa(\lambda^s/\ell))$ is a point of $P_n(s)$ (use Proposition \ref{kappa}).
It is easy to check that $v(\ma)=(\kappa(\lambda^1/\ell),\dots,\kappa(\lambda^s/\ell))$, which finishes the proof.
\end{proof}

\begin{defi}\label{Shosta}
In the setting of Prop \ref{schubertt}, we will call the moduli of unitary representation of $\pi_1$ with local monodromy data $\mathcal{A}$ (possibly empty)  ``strange dual'' to  the line bundle  $\mathcal{B}(\vec{\lambda},\ell)$ (possibly non-effective), i.e., the numerical data (A) and (B) are said to be in strange duality (also see Remark \ref{added}).

\end{defi}
\begin{defi}\label{amirK}
We will also assign parabolic weights if the $\lambda_i$ are dominant and $\ell>0$: The parabolic weights of $\ml=\mathcal{B}(\vec{\lambda},{\ell})$ are $\frac{{\lambda^i}}{\ell}$ at $p_i$. The parabolic degree of a point of $\Par_{n,\mathcal{O},S}$ corresponding to the line bundle  $\mathcal{B}(\vec{\lambda},{\ell})$ is $\frac{1}{n}(\deg{\mathcal{O}}+\frac{1}{\ell}\sum_{i=1}^s|\lambda^i|)=\ell-1-\deg\widetilde{\mn}$. The parabolic degree of $\mathcal{B}(\vec{\lambda}^T,n)$ on $\Par_{\ell,\widetilde{\mathcal{N}},S}$ in Proposition \ref{BWag} is then $\frac{1}{\ell}(\deg{\widetilde{\mn}}+\frac{1}{n}\sum_{i=1}^s|(\lambda^i)^T|) =n-n=0$.
\end{defi}

\begin{theorem}\label{christmas} 
Under the correspondence between the sets (A) and (B) above, suppose that $\mathcal{B}(\vec{\lambda},\ell)$ on $\Par_{n,{\mathcal{O}},S}$ is effective (equivalently, by Proposition \ref{schubertt}, there is a unitary representation of $\pi_1$ with  local monodromy data $\mathcal{A}$). Then
\begin{enumerate}
\item[(i)] $\mathcal{B}(\vec{\lambda},\ell)$ is a Hilbert basis element of  the semigroup $\Pic^+(\Par_{n,\mathcal{O},S})$ if and only if 
any unitary local system on $\Bbb{P}^1-S$ with local monodromy data $\mathcal{A}$ is  irreducible.
\item[(ii)] $\mathcal{B}(\vec{\lambda},\ell)$ is an F-line bundle on   $\Par_{n,\mathcal{O},S}$ (and hence gives rise to a F-vertex of $P_n(s)$ by Lemma \ref{schubertt}) if and only  there is a unitary, irreducible rigid local system on  $\Bbb{P}^1-S$ with local monodromy data $\mathcal{A}$. 
\end{enumerate}
\end{theorem}
\begin{remark}\label{added}
Under the above correspondence,
F-line bundles on $\Par_{n,\mathcal{O},S}$ (and the corresponding F-vertices) on the (B) side are therefore in strange duality with the unitary rigid local systems on (A) side.
\end{remark}
\subsection{Proof of Theorem \ref{christmas}}
We first prove the  contrapositive of  the reverse direction in (i): Suppose,
\begin{equation}\label{decompo}
\mathcal{B}(\vec{\lambda},\ell)=\mathcal{B}(\vec{\lambda'},\ell')\otimes \mathcal{B}(\vec{\lambda''},\ell'')
\end{equation}
with $\vec{\lambda}=\vec{\lambda'}+\vec{\lambda''}$ and $\ell=\ell'+\ell''$, and $\vec{\lambda},\vec{\lambda'}$ and $\vec{\lambda''}$ are vectors of  normalised weights, i.e., their $n$th rows are empty.   Assume the two factors $\mathcal{B}(\vec{\lambda'},\ell')$ and $\mathcal{B}(\vec{\lambda''},\ell'')$ both have non-zero global sections. They then are necessarily of grade zero. The correspondence (A) to (B) for these factors give two (families) of unitary local systems $\mt'$ and $\mt''$ of ranks $\ell'$ and $\ell''$ respectively. We claim that $\mt'\oplus\mt''$ is a local system of rank $\ell$, with local monodromy $\mathcal{A}$: Write $\lambda'^i=\sum_{b=1}^{n-1} c'^b_i\omega_b$ and $\lambda'^i=\sum_{b=1}^{n-1} c''^b_i\omega_b$. Then $\bar{A}_i$ has eigenvalues $\exp(2\pi \sqrt{-1}b/n)$ repeated $c^b_i=c'^b_i+c''^b_i$ times, as desired.

For the other direction of (i), assume that a point in the moduli space corresponding to the strange dual of $\ml=\mathcal{B}(\vec{\lambda},\ell)$ (see Definition \ref{Shosta})  breaks up as a direct sum $\mathcal{T}=\mathcal{T}'\oplus \mathcal{T}''$ of semistable bundles of ranks $\ell'$ and $\ell''$ respectively. Under the correspondence of (A) and (B), suppose $\mt'$ and $\mt''$ correspond to $\mathcal{B}(\vec{\lambda'}, \ell')$  and $\mathcal{B}(\vec{\lambda''}, \ell'')$ respectively. It is now easy to see that
\eqref{decompo} holds on $\Par_{n,\mathcal{O},S}$, and the two factors are effective by Proposition  \ref{schubertt}.

If $\ml=\mathcal{B}(\vec{\lambda},\ell)$ is an F-line bundle on   $\Par_{n,\mathcal{O},S}$, then $\ml'=\mathcal{B}(\vec{\mu},n)$ on the moduli stack  $\Par_{\ell,\widetilde{\mathcal{N}},S}$ has a one dimensional space of sections (see equation \ref{GepnerPP}), making the corresponding moduli space, of representations of the fundamental group with local monodromy data $\ma$, a point, i.e., the corresponding representation is rigid as a unitary representation. Since it is also irreducible by (i), it is an irreducible rigid local system (see Remark \ref{RigidRem}).

If the strange dual gives a rigid unitary representation, then it obviously does not break up as a direct sum of semistable bundles anywhere in its moduli space, and the line bundle has exactly one section (up to scalars). It then follows from (i) that its  strange dual (i.e., $\mathcal{B}(\vec{\lambda},\ell)$) is a Hilbert basis element which has exactly one section up to scalars (use \eqref{GepnerPP}). Writing $\mathcal{B}(\vec{\lambda},\ell)=\mathcal{O}(E)$ we see that $E$ is irreducible and not a multiple since it is a Hilbert basis element. Therefore, $\mathcal{B}(\vec{\lambda},\ell)$ is an F-line bundle, as desired.
\subsubsection{Proof of Theorem \ref{egregium}}
The association of F-line bundles on $\Par_{n,\mathcal{O},S}$ and F-vertices of $P_n(s)$ is a bijection by Proposition \ref{corresp}. Note that this does not use 
strange duality, and follows from the Mehta-Seshadri theorem, geometric invariant theory, and the saturation conjecture (see Section \ref{uber}). 

Theorem \ref{egregium}  therefore reduces to the statement  that $F$-line bundles on $\Par_{n,\mathcal{O},S}$ of level $\ell$ are in one-one bijection with the set of local monodromy data $\ma$ (as in (A)) which correspond to unitary, irreducible, rigid local systems. Now  by Proposition \ref{schubertt}, effective level $\ell$  line bundles on $\Par_{n,\mathcal{O},S}$ are in one-one correspondence with local monodromy data $\ma$ (as in (A)) which correspond to  rank $\ell$ unitary local systems.

Theorem \ref{egregium} now follows from Theorem \ref{christmas} (ii). Reviewing the proof of Theorem \ref{christmas}, we see that if we are able to 
decompose an effective  line bundle on $\Par_{n,\mathcal{O},S}$ as tensor product of effective line bundles, then on the strange dual side, we can create a reducible representation with monodromy data $\mathcal{A}$ (and vice versa).

\begin{remark}\label{numerics}
If $\rigid$  is an irreducible rigid local system of rank $\ell$  on $\Bbb{P}^1-S$ with local monodromies $A_i$ (not necessarily semisimple), then one has \cite[Theorem 1.1.2]{Katz}
\begin{equation}\label{eqRi}
\sum_{i=1}^s \dim Z(A_i)= (s-2) \ell^2+2
\end{equation}
Here $Z(A_i)$ is the subgroup of $\op{GL}(\ell,\Bbb{C})$ of matrices that commute with $A_i$. Also note that if $A$ is a semisimple matrix  with eigenvalues $a_1,\dots,a_p$ of multiplicities $n_1,\dots,n_p$, then $\dim Z(A)=\sum_{j=1}^p n_j^2$.
\end{remark}
\begin{lemma}\label{RigidRem}
Suppose  $\rigid$ is a  unitary local system on $\Bbb{P}^1-S$. Let $\rho: \pi_1(\Bbb{P}^1-S,b)\to \op{U}(\ell)$ be the  corresponding representation. Then the following conditions on $\rigid$ are equivalent :
\begin{enumerate}
\item[(a)]$\rigid$ is an irreducible rigid local system in the sense of Katz \cite{Katz}.
\item[(b)] $\rigid$ is an irreducible local system, which is rigid as unitary local system (i.e., any $\rho': \pi_1(\Bbb{P}^1-S,b)\to \op{U}(\ell)$ with the same local monodromies as $\rho$ is conjugate to it).
\end{enumerate}
\end{lemma}
\begin{proof}
This follows from the fact that the complexification of the Lie algebra of $\op{U}(\ell)$ is the Lie algebra of $\op{GL}(\ell)$, and the description of tangent spaces of representation varieties. Here is another direct numerical explanation:
To see this, note that the irreducibility requirements in (a) and (b) are the same. Therefore let us assume these to hold. In this situation (a) is equivalent to \eqref{eqRi}  with $A_i=\rho(\gamma_i)$.
Assuming irreducibility, we have a stable point in the Mehta-Seshadri moduli space of $\rigid$ (call this moduli $\mathcal{M}$). Using \cite[Lemma 5.4]{BTIFR}, (b) is equivalent to the same numerical condition:
Let $\rigid\in\mathcal{M}$ correspond to a parabolic bundle whose underlying bundle is $\mw$. We may assume $\mw$ to be generic, and automorphisms of $\mw$ is a group of dimension $\ell^2$ \cite[Lemma 9.1]{BTIFR}, and (b) says that it has an open orbit in the  corresponding variety of partial flags in the fibers of $\mw$ at points of $S$. The corresponding partial  flag variety at the point $p_i\in S$ should be shown to have dimension $\frac{1}{2}(\ell^2-Z(A_i))$. Using Remark
\ref{numerics}, this follows from \cite[Lemma 1.2]{BTIFR}.
\end{proof}

\begin{remark}\label{Hilberty} Let $n$ and $|S|$ be fixed.  There is a constant $\tilde{C}(|S|,n)$ with the following property:
 Suppose $\mathcal{T}$ is an irreducible  rank $\ell$ unitary local system (possibly non-rigid) on $\Bbb{P}^1-S$ with local monodromy data $\mathcal{A}=(\bar{A}_1,\dots,\bar{A}_s)$ with $\bar{A}_i^n=I_{\ell}$, such that any unitary local system with local monodromy data $\mathcal{A}$ is irreducible. Then, $\ell\leq \tilde{C}(|S|,n)$. This boundedness property follows from part (i) of Theorem \ref{christmas}, and the finiteness  of the number of elements in the Hilbert basis of a saturated semigroup.

Local systems  $\mathcal{T}$ with the  above property are not necessarily rigid. By Theorem \ref{christmas} the strange duals of non F-vertices of $P_n(s)$ give examples of families of such $\mt$ which are not rigid.  An example of a non F-vertex in $P_9(3)$ is given in Section \ref{Rex}. The corresponding strange dual is an one dimensional family of such unitary $\mt$ of rank $3$, with local monodromies $(\zeta^3,\zeta^7,\zeta^8)$ at $p_1$, and ($1,\zeta^3,\zeta^6)$ at $p_2$ and $p_3$, with $\zeta=\exp(2\pi\sqrt{-1}/9)$.  We are unable to construct any such $\mt$ explicitly (so this is only an existence statement).

We find reducible but non-unitary $\mt$ with the same local monodromy data: of the form $\mt_1\oplus\mt_2$, where $\mt_1$ is rank one with local monodromies $\zeta^3$ at all three points, and $\mt_2$ of rank $2$ with local monodromies $(\zeta^7,\zeta^8)$, $(1,\zeta^6)$ and $(1,\zeta^6)$. Such a $\mt_2$ arises from a hypergeometric system with $\alpha_1=7/9$, $\alpha_2=8/9$, $\beta_1=0$ and $\beta_2=3/9$. $\mt_2$ is not unitary (see Theorem \ref{BHhere}).

\end{remark}
\begin{remark}\label{stretch2}
In Theorem \ref{christmas}, the hypotheses on the (moduli of) unitary local systems coming from the strange duals are invariant under scaling the choice of $n$, i.e., if local monodromies are $n$th roots of unity, they are also $(nm)$th roots of unity for any positive integer $m$. Under strange duality, scaling is strange dual to stretching: The $m$-stretch of a line bundle
$\mathcal{B}(\vec{\lambda},\ell)$ on  $\Pic^+(\Par_{n,\mathcal{O},S})$ is a line bundle $\mathcal{B}(\vec{\lambda}[m],\ell)$ on  $\Pic^+(\Par_{mn,\mathcal{O},S})$. Here if $\lambda=\sum_{i=1}^{n-1}  c_i\omega_i$ is a weight for $\mathfrak{sl}_n$ then $\lambda[m] = \sum_{i=1}^{n-1}  c_i\omega_{mi}$ is a weight for $\mathfrak{sl}_{mn}$  (see \cite[Prop 1.3]{GG} for an appearance of this operation in the context of Chern classes of conformal blocks).

We can also scale in $\ell$ on the (B)-side in Proposition \ref{schubertt}, replacing $\mathcal{B}(\vec{\lambda},\ell)$ by a power $\mathcal{B}(m\vec{\lambda},m\ell)$, but now (A) has rank
$m\ell$. The cone structure on the (B)-side in $\Pic^+(\Par_{n,\mathcal{O},S})$ therefore goes into an operation corresponding to direct sums on the (A)-side.

According to Theorem \ref{christmas}, the properties of being a Hilbert basis element, and a F-line bundle (hence for F-vertices), are invariant under stretching. This invariance property for F-vertices (similarly for all vertices), under stretching,  was proved previously in unpublished joint work with  Gibney and Kazanova using the natural (``direct sum'') map $\Par_{n,\mathcal{N},S}\to \Par_{mn,\mathcal{N}^m,S}$, quantum saturation and Fulton conjectures (Proposition   \ref{qTheorems}), and strange duality \ref{BWag}.
\end{remark}

\subsection{Proof of Corollary \ref{egregiumprime}}
\begin{defi}\label{aliB}
If $\lambda=\sum_{a=1}^{n-1} c_a\omega_a$ is a dominant integral weight of $\op{SL}(n)$, and $\op{gcd}(m,n)=1$, let $T_m(\lambda)$ be the dominant integral weight $\sum c_a\omega_{ma \pmod{n}}$.
\end{defi}

\begin{lemma}\label{calcul}
Suppose $\sigma\in \op{Gal}(\bar{\Bbb{Q}}/\Bbb{Q})$ has image $\bar{m}\in (\Bbb{Z}/n\Bbb{Z})^*=\op{Gal}(\Bbb{Q}(\zeta_n)/\Bbb{Q})$.

\begin{enumerate}
\item Then $v(\sigma(\mathcal{A}))= T_m(v(\mathcal{A}))$ ($T_m$ is as defined in Definition  \ref{aliA}).
\item  $\mathcal{B}(\vec{\lambda},\ell)$ is indivisible as a grade zero line bundle if and only if
$\mathcal{B}(T_m(\vec{\lambda}),\ell))$ is indivisible as a grade zero line bundle ($T_m$ is  defined in Definition  \ref{aliB}).
\end{enumerate}
\end{lemma}
\begin{proof}
Consider a  semisimple class $\bar{A}\in \op{GL}(\ell,\Bbb{C})$ with eigenvalues $\exp(2\pi\sqrt{-1}\mu_j/n)$ with $n>\mu^i_1\geq \mu^i_2\geq \dots \geq \mu^i_{\ell}\geq 0$. Let $\lambda=\mu^T$ be expressed in the form $\sum_{a=1}^{n-1} c_a\omega_a$. The eigenvalues of $\bar{A}$ are $\zeta_n^a$ with multiplicity  $c_a$  for $1\leq a \leq n$, and $1$ with multiplicity $\ell-\sum c_a$. The action of $\sigma$ on $\bar{A}$ changes $\lambda$ to $\sum_{a=1}^{n-1} c_a\omega_{ma \pmod n}$ since the element of $\op{Gal}(\Bbb{Q}(\zeta_n)/\Bbb{Q})$ corresponding to $m$ takes $\zeta_n$ to $\zeta_n^m$.
The lemma now follows from a small computation.
\end{proof}
We now prove Corollary \ref{egregiumprime}. Suppose there is a rigid local system $\rigid$  with local monodromy data $\mathcal{A}$, with finite global monodromy. Suppose the corresponding matrices are $(A_1,\dots,A_s)$. Then we may assume
$A_i$ have coefficients in $\bar{\Bbb{Q}}$ (in fact $\rigid$ is defined over the ring of integers of a cyclotomic
number field by a theorem of Katz \cite{Katz}). Then the twists $(\sigma(A_1),\dots,\sigma(A_s))$ also have finite monodromy for all $\sigma\in \op{Gal}(\bar{\Bbb{Q}}/\Bbb{Q})$. Hence there is a unitary irreducible rigid local system corresponding to $\mathcal{A}^{\sigma}$ for every $\sigma$. We can complete the forward implication in the corollary using Theorem \ref{egregium} and Lemma \ref{calcul}.

For the reverse implication let $\rigid$ be the unitary irreducible rigid local system corresponding to $\mathcal{A}$. By rigidity, and our assumptions,  all Galois conjugates of $\rigid$ are unitary. Hence by a standard argument, see e.g., \cite[Theorem 11]{BLocal}, Katz's theorem that $\rigid$ is
defined over the ring of integers of a cyclotomic number field implies that $\rigid$ has finite global monodromy.
\subsection{More consequences and examples}
Theorem \ref{christmas} therefore imposes numerical conditions on the weights of F-line bundles:
\begin{corollary}\label{katzkatz}
Suppose $\mathcal{B}(\vec{\lambda},\ell)$ is an F-line bundle on $\Par_{n,\mathcal{O},S}$. Write $\lambda^i=\sum_{b=1}^{n-1}c^b_i\omega_b$ for $i=1,\dots,s$. Then,
\begin{equation}\label{egale}
\sum_{i=1}^s \bigl((\ell-\sum_{b=1}^{n-1} c^b_i)^2 \ + \ \sum_{b=1}^{n-1}(c^b_i)^2\bigr) =(s-2)\ell^2+2.
\end{equation}
\end{corollary}
\begin{proof}
Then the transpose of $\lambda^i$ has rows of length $b$ repeated $c^b_i$ times and a row of length $0$ repeated $\ell-\sum_{b=1}^{n-1} c^b_i$ times. Equation \eqref{egale} now follows from Theorem \ref{christmas} and equation \eqref{eqRi}.
\end{proof}

\subsection{Some examples}\label{smalln}
In this section we use known explicit descriptions of vertices $P_n(3)$ for $n\leq 6$ to list all unitary irreducible rigid local systems whose local monodromies are $n$th roots of identity.  To do this we  start with a F-vertex use Proposition \ref{corresp} and Remark \ref{inpractice} to find the corresponding F-line bundle, and then apply Theorem \ref{christmas} to find the corresponding rigid local system. For $n\leq 6$ all vertices of $P_n(3)$ are F-vertices.

By direct computation $P_2(3)\subseteq[0,\frac{1}{2}]^3$ is a polytope with vertices $(\frac{1}{2}, \frac{1}{2},0)$ and various permutations. Note that $\Delta_2=\{(a,-a): a\geq 0,\ a\leq -a +1\}=[0,\frac{1}{2}]$.
The strange duals of these vertices are of level one, so we only get rigid local systems of rank $1$. Therefore there are no unitary and irreducible  rigid local systems on $\Bbb{P}^{1}-\{p_1,p_2,p_3\}$ of rank $>1$ whose local monodromies have eigenvalues $\pm 1$.

By an explicit computation of $P_3(3)$ a similar conclusion holds for $n=3$: There are no unitary and irreducible  rigid local systems on $\Bbb{P}^{1}-\{p_1,p_2,p_3\}$ of rank $>1$ whose local monodromies have eigenvalues third roots of unity.

For $n=4$, by \cite[Section 7]{BLocal}, there are F-vertices of levels are bigger than one: These correspond to the  F-line bundle $\mathcal{B}(\vec{\lambda},\ell)$ with $\ell=2$ and $\lambda^1=\lambda^2=\lambda^3=\omega_1+\omega_3$. Therefore in the corresponding the rank $2$ local system   all local monodromies are conjugate to the diagonal matrix with entries $\exp(2\pi\sqrt{-1}/4)$ and $\exp(6\pi\sqrt{-1}/4)$, i.e., $\sqrt{-1}$ and $-\sqrt{-1}$.  The corresponding local system corresponds to the equation $\sigma_1\sigma_2\sigma_3=I_2$ where
\begin{equation}\label{dfnsigma}
   \sigma_1=
  \left[ {\begin{array}{cc}
   0 & 1 \\
   -1 & 0 \\
  \end{array} }\right],   \sigma_2=
  \left[ {\begin{array}{cc}
   -\sqrt{-1} & 0 \\
   0 & \sqrt{-1} \\
  \end{array} }\right],  \sigma_3=
  \left[ {\begin{array}{cc}
   0 & \sqrt{-1} \\
   \sqrt{-1} & 0 \\
  \end{array} }\right].
  \end{equation}
This local system and all of its twists by $4$th roots of unity  are the only unitary irreducible rigid local systems on $\Bbb{P}^{1}-\{p_1,p_2,p_3\}$ of rank $>1$ whose local monodromies have eigenvalues fourth roots of unity.
(By twisting we mean the following. Let $a_1,a_2,a_3$ be $n$th roots of unity with $a_1a_2a_3=1$.  Then the twist is just the tensor product with the rank one local system given by $(a_1,a_2,a_3)$). They all have finite global monodromy by Corollary \ref{egregiumprime}, and appear in Schwarz's list from the 19th century \cite{Schwarz}, and are of hypergeometric type (up to a twist).

For $n=5,6$ we use computations of $P_n(3)$ reported in \cite{thaddy}.

For $n=5$ there are two orbits under $\tau_5(3)\rtimes S_3$ (see Section \ref{ABWin})  of vertices of $P_5(3)$, both consisting of F-vertices. The first one is the orbit of $(0,0,0,0,0)^3$ whose strange duals are rank one. The second is the orbit of $(\frac{1}{2},0,0,0,-\frac{1}{2})^3$, the point $(\frac{1}{2},0,0,0,-\frac{1}{2})$ corresponds to $(1/2,0,0,1/2)\in\Delta'_5$ in the notation of Remark \ref{aliA}.
The corresponding unitary irreducible rigid local system is of rank $2$. By Theorem \ref{KatzQ}, none of the second kind of local systems have finite global monodromy (which can also be checked here by using Corollary \ref{egregiumprime}).

For $n=6$, let us first note that the Galois group $\op{Gal}(\Bbb{Q}(\zeta_n)/\Bbb{Q})$ has order two, and hence we just need to list the strange duals of all vertices of $P_6(3)$, all of these will have finite global monodromy (the action on $\Delta_n^3$ is either the identity or by ordinary duals, and both preserve $P_n(3)$). There are three orbits under $\tau_6(3)\rtimes S_3$ (see Section \ref{ABWin})  of vertices of $P_5(3)$,  all consisting of F-vertices. There strange duals give us the following local systems
on $\Bbb{P}^{1}-\{p_1,p_2,p_3\}$. These and their central twists by $6$th roots of unity are the only unitary (hence finite for $n=6$) rigid local systems on $\Bbb{P}^{1}-\{p_1,p_2,p_3\}$ with local monodromies  sixth roots of unity, here $\zeta=\exp(\frac{2\pi\sqrt{-1}}{6})$. They are all of classical
hypergeometric type (up to a twist)  as in Section \ref{classH}. We indicate the classical type below.
\begin{enumerate}
\item The first local system is a rank $2$ local system with  all local monodromies conjugate, the eigenvalues are $(\zeta^5,\zeta)$. The twist consisting of local monodromies $(\zeta^5,\zeta)$, $(1,\zeta^2)$, and $ (1,\zeta^4)$ corresponds to a rank $2$ hypergeometric system with $\alpha_1=\frac{1}{6}$, $\alpha_2=\frac{5}{6}$, $\beta_1=0$ and  $\beta_2=4/6$.
\item The second local system is a rank $2$ local system with  eigenvalues of local monodromies  given by
$(\zeta,\zeta^5), (1,\zeta^3)$, and $(1,\zeta^3)$.  This corresponds to a rank $2$ hypergeometric system with $\alpha_1=\frac{1}{6}$, $\alpha_2=\frac{5}{6}$, $\beta_1=0$ and  $\beta_2=3/6$.
\item The third local system is a rank $3$ local system with eigenvalues of local monodromies $(1,1,\zeta^3)$, $(1,\zeta^2,\zeta^4)$, and $(\zeta,\zeta^3,\zeta^5)$. This  corresponds to a rank $3$ hypergeometric system with $\alpha_1=\frac{1}{6}$, $\alpha_2=\frac{3}{6}$,  and $\alpha_3= \frac{5}{6}$, $\beta_1=0$,  $\beta_2=\frac{2}{6}$ and $\beta_3=\frac{4}{6}$.
\end{enumerate}

\subsection{Special properties of rigid unitary local systems and the proof of Theorem \ref{KatzQ}} \label{KatzSection}
\begin{proposition}\label{propP}
Let $\mt$ be a rank $\ell$ irreducible rigid local system with semisimple local monodromy data
$\mathcal{A}=(\bar{A}_1,\dots,\bar{A}_s)$ in $\op{GL}(\ell,\Bbb{C})$, with $\bar{A}_i^n=1$. Let $i\in[s]$ and let $\alpha$ and $\beta$ be two eigenvalues of $\bar{A}_i$.
\begin{enumerate}
\item If $\mt$ is unitary, then $\alpha^{-1}\beta\neq \zeta_n$ where $\zeta_n=e^{2\pi\sqrt{-1}/n}$ (i.e., two different eigenvalues of the local monodromy at a point of $S$ cannot be ``as close as possible" for unitary rigid irreducible local systems).
\item If $\mt$ has finite global monodromy, then $\alpha^{-1}\beta$ is not a primitive $n$th root of unity.
\end{enumerate}
\end{proposition}
Theorem \ref{KatzQ}, and property (P) from the introduction are clearly  immediate consequences of the above proposition.
\begin{proof}
The Galois group $\op{Gal}(\Bbb{Q}(\zeta_n)/\Bbb{Q})$ acts transitively on the set of primitive $n$th roots of unity and therefore the second part follows from the first part (use that any Galois conjugate of $\mt$ is a unitary local system in addition to being rigid and irreducible).

For the first part we proceed as follows. Assume the contrary, then by the action of $\tau_n(s)$, we may assume $\alpha=1$, $\beta= \exp(2\pi \sqrt{-1}/n)$. Let the corresponding F-line bundle (by the correspondence between data  (A) and (B) in Section \ref{SD}) on $\Par_{n,\mathcal{O},S}$ be $\ml=\mb(\vec{\lambda},\ell)$. Write  $\lambda^i=\sum_{b=1}^{n-1}c^b_i\omega_b$ for $i=1,\dots,s$. We also have from Lemma \ref{divisorE}, an expression $\ml=\mathcal{O}(E)$ where  $E=C(d,r,\mathcal{O},n,\vec{J})$.

By Proposition \ref{enumerative}, parts (C) and (D), and the strange duality operation (see Definition \ref{Shosta}) the following properties hold which imply that $1$ and $\zeta_n$ cannot both be eigenvalues of $A_i$:
\begin{enumerate}
\item If  $1\leq b<n$,  then $\exp(2\pi\sqrt{-1}b/n)$ is an eigenvalue of $A_i$ with multiplicity $c^b_i$ (possibly zero). If this multiplicity is non-zero, then $b\in J^i$ and $b+1\not\in J^i$.
\item The multiplicity of the eigenvalue  $1$ of $A_i$ is  $\ell-\sum_{b=1}^{n-1}c^b_i$ (possibly zero). If this multiplicity is non-zero,  then $n\in J^i$ and $1\not\in J^i$.
\end{enumerate}
\end{proof}
\begin{remark}\label{moreR} The assumption of rigidity is needed in property (P) and Proposition \ref{propP}: Enlarge $S$ by two points and introduce local monodromies at the points that are inverses of each other.
\end{remark}

\subsection{Rigid local systems from quantum Schubert calculus}\label{qSC}
Theorems \ref{Adagio} and \ref{brahms} produce F-line bundles $\mathcal{O}(D(a,j))$, and hence rigid irreducible unitary local systems as in Corollary \ref{existence}. Let us recall the steps:
\begin{enumerate}
\item [(i)]Start with a situation where
 $\langle\sigma_{I^1}, \sigma_{I^2},\dots,\sigma_{I^s}\rangle_d= 1$ in a Grassmannian $\Gr(r,n)$.
\item[(ii)] Choose $(a,j)$: Let $(a,j)$ be a pair with $a\in[n]$ and $1\leq j\leq s$ such that
\begin{enumerate}
\item $a>1$, $a\in I^j$, and $a-1\not\in I^j$, or
\item $a=1$, $1\in I^j$, $n\not\in I^j$.
\end{enumerate}
\item[(iii)] Define subsets $J^1,\dots,J^s$ of $[n]$ each of cardinality $r$ and an integer $d'$ as follows:
In case (1) let $J^k=I^k$ if $k\neq j$, $J^j=(I^j-\{a\}) \cup \{a-1\}$ and $d'=d$, and in case (2) let  $J^k=I^k$ if $k\neq j$, $J^j=(I^j-\{1\}) \cup \{n\}$ and $d'=d-1$.
\item[(iv)] Construct the F-line bundle $\mathcal{B}(\vec{\lambda},\ell)$ on $\Par_{n,\mathcal{O},n}$
by the formulas $\lambda^i=\sum_{b=1}^{n-1}c^i_b\omega_b$ with formulas for $\ell$ and $c^i_b$
as given in Theorem \ref{brahms}.
 \item[(v)] Finally the rigid local system $\mt$ is of rank $\ell$ on $\Bbb{P}^1-S$ with local monodromy at $p_i$ for $i\in[s]$
 conjugate to a diagonal matrix with eigenvalues $\exp(2\pi\sqrt{-1}b/n)$ with multiplicity $c^i_b$, for $b=1,\dots n-1$, and eigenvalue $1$ with multiplicity $\ell-\sum_{b=0}^{n-1} c^i_b$. Note that the KZ differential equation which gives rise (up to a cyclic twist) to $\mathcal{T}$ is given explicitly by Theorem \ref{KZe} and Proposition \ref{generalKZ}, since the data of the divisor $E$ in the proof of Theorem \ref{KZe} is explicit in this situation.
 \end{enumerate}

 Quantum Schubert calculus gives a method of computing all $\langle\sigma_{I^1}, \sigma_{I^2},\dots,\sigma_{I^s}\rangle_d$, and in the case $d=0$ there are many other methods for computing these numbers (Littlewood-Richardson rule, the honeycomb and puzzle formulations of Knutson and Tao). As for systematic methods of obtaining situations where $\langle\sigma_{I^1}, \sigma_{I^2},\dots,\sigma_{I^s}\rangle_d= 1$, we note the following:
\begin{enumerate}
\item Let $\lambda,\mu$ and $\nu$ be dominant integral weights of $\op{SL}_r$ such that there is a Weyl group  element $w\in S_r$ with $\lambda+w\mu=\nu$. In this case, by a result of Kostant \cite[Lemma 4.1]{Kostant}, $V_{\nu}$ appears with multiplicity one in $V_{\lambda}\tensor V_{\mu}$. Since Littlewood-Richardson numbers also compute structure coefficients in the classical cohomology of suitable Grassmannians $\Gr(r,n)$, this result of Kostant produces a family of $\langle\sigma_{I^1}, \sigma_{I^2},\dots,\sigma_{I^s}\rangle_d=1$ with $d=0$.
\item The coefficients that appear in the quantum Pieri rule are all one \cite{bertie}. These correspond to $s=3$ and $\lambda(I^1)=(a,0,\dots,0)$ (or their duals). Some of the unitary local systems that arise from this situation seem to be a subset of  those coming from the classical  hypergeometric systems (Section \ref{hyper}), but there is an infinite series of new ones  as shown by Kiers and Orelowitz (see Section \ref{KO} for their examples).
\item Let $V_1,\dots,V_s$ be  subspaces in general position of a vector space $W$ so that $\sum_{i=1}^s \dim V_i= \dim W+1$. Now there is a unique subspace $T\subset W$ of dimension $s-1$ that meets all $V_i$ non-trivially.
This sets up a classical intersection number $1$ situation in $\Gr(s-1,W)$ (we learned of this example from Nori): $\langle\sigma_{I^1},\dots,\sigma_{I^s}\rangle_0=1$ where $\lambda(I^j)=(a_j,0,\dots,0)$ (see Definition  \ref{correspondence}) with $a_j = n-(s-1) +1-\dim V_j$. The corresponding unitary irreducible rigid local systems seem to include some classical Pochhammer systems (Section \ref{pokhie}).
\item There are examples of Hobson \cite{Hobson} of rank one conformal blocks in type A, which then give rise
to a situation where $\langle\sigma_{I^1}, \sigma_{I^2},\dots,\sigma_{I^s}\rangle_d= 1$ using the Witten dictionary (see Section \ref{QCB}).
\item Suppose we have a classical intersection number one   situation in a Grassmannian $\Gr(r,n)$, i.e., $\langle\sigma_{I^1}, \sigma_{I^2},\dots,\sigma_{I^s}\rangle_d= 1$ with $d=0$. We can then create an intersection number one situation in $\Gr(r,n+1)$:   $\langle\sigma_{J^1}, \sigma_{J^2},\dots,\sigma_{J^s},\sigma_{J^{s+1}}\rangle_0= 1$ with $J^{s+1}=\{n-1,n-2,\dots,n-r-1\}$ and $J^i$ obtained by adding $1$ to elements of $I^i$ for $i\in[s]$.  Running the above process from the two intersection one situations above, one gets two local systems (of possibly different ranks), one on $\Bbb{P}^1-\{p_1,\dots,p_s\}$ and the other on $\Bbb{P}^1-\{p_1,\dots,p_s,p_{s+1}\}$, and the local monodromies have eigenvalues that are $n$th and $(n+1)$th roots of unity respectively.
\item Any classical GW intersection number (i.e., $d=D=0$) in a Grassmannian $\Gr(r+1,r+\ell+2)$ breaks up as a sum of two GW numbers for $d=1$ on Grassmannians $\Gr(r,r+\ell+1)$ and $\Gr(\ell+1,r+\ell+1)$ by \cite[Section 9.9]{BGM}. If the classical intersection number is one, then exactly of the two GW numbers is one.
\item Theorem \ref{Adagio} inductively produces situations where $\langle\sigma_{I^1}, \sigma_{I^2},\dots,\sigma_{I^s}\rangle_d= 1$, since the line bundles $D(a,j)$ are F-line bundles: The new Grassmannian
is $\Gr(n,n+\ell)$ where $\ell$ is the level of $\mathcal{O}(D(a,j))$ (and using Lemma \ref{newlabel2}).
\end{enumerate}
In Sections \ref{E1} and \ref{KO}, we list examples of unitary rigid local systems that are not of hypergeometric and Pochhammer type (which have at least one  matrix with only two distinct eigenvalues). Several examples of $\langle\sigma_{I^1}, \sigma_{I^2},\dots,\sigma_{I^s}\rangle_d= 1$ for $n\leq 16$ are given in \cite{OO} which also includes a discussion of computational aspects of Gromov-Witten invariants.
\subsection{All unitary irreducible rigid local systems}\label{qSC2}
  One may ask if all unitary irreducible rigid local systems with local monodromy $n$ roots of identity arise from Corollary \ref{existence}, i.e., the procedure described in Section \ref{qSC}. By Theorem \ref{christmas} (ii), this is equivalent to the question  of whether all F-line bundles on $\Par_{n,\mathcal{O},S}$ are of the form $\mathcal{O}(D(a,j))$. By Lemma \ref{powercord} (also see Proposition  \ref{practical}), this is equivalent to asking
\begin{question}\label{Q}
Let $\mathcal{B}(\vec{\lambda},\ell)$ be an F-line bundle on $\Par_{n,\mathcal{O},S}$. Write $\lambda^i=\sum_{b=1}^{n -1} c^b_i\omega_b$. Then is it necessarily true that  there is an $i$ such that
$c^b_i=1$ for some $b$, or $\ell -\sum_{b=1}^{n-1} c^b_i=1$?
\end{question}
The strange dual way of asking this question is the following (see the proof of Corollary \ref{katzkatz}): Let $\mathcal{T}$ be a irreducible rigid unitary local system on $\Bbb{P}^1-S$. Is it necessarily the case that the local monodromy of $\mathcal{T}$ around some $p_i\in S$ has an eigenvalue of multiplicity one? See Section \ref{wilson} for an example, due to A. Wilson, in which
$c^b_i\neq 1$ for all $b$ and $i$ (coming from a classical context), but in this example  $\ell -\sum_{b=1}^{n-1} c^b_i=1$ for some $i$. The rigidity equation \eqref{eqRi} does not settle this:
For example a rigid local system of rank $7$ with $|S|=3$, and multiplicity decomposition $(3,2,2)$
at each point is allowed by \eqref{eqRi}. We do not know if there are any unitary irreducible rigid local systems with these eigenvalue decompositions of local monodromies.

Note the hypergeometric and Pochhammer unitary local systems satisfy Question
\ref{Q}, and hence arise from Corollary \ref{existence}.

\subsection{Rigid irreducible local systems without the unitarity condition}\label{afterN}
We consider the following question:
\begin{question} Fix a natural number $n$. Consider the set of  irreducible rigid local systems (possibly non-unitary) on $\Bbb{P}^1-S$  such that the local  monodromies are semisimple with eigenvalues $n$ roots of unity.
Is this set finite, i.e., are the ranks $\ell$ of such local systems bounded above?
\end{question}
In the terminology of Katz's book \cite[Section 6.1]{Katz},  the above question asks, in characteristic zero,  whether  realizable  classes in $\op{NumData}(D,\Gamma)$, such that the ``local monodromies'' are semisimple, form a finite set when $\Gamma$ is the cyclic group of order $n$. We do not know the answer to this question.  We examine  below whether plausible (\cite[Definition  6.3.3]{Katz}) elements  in $\op{NumData}(D,\Gamma)$ with semisimple local monodromies form a finite set.

We consider
two natural  necessary conditions on the local monodromy data $\mathcal{A}=(\bar{A}_1,\dots,\bar{A}_s)$ of such a local system of rank $\ell$. The first one is the rigidity equation  \eqref{eqRi}. The second one, a consequence of irreducibility, is the following condition  which can be applied to all cyclic twists of the data (see \cite[Lemma III.9.10]{MM}):
\begin{equation}\label{irreducible1}
\sum_{i=1}^s \rk(A_i-I_{\ell})\geq 2\ell
\end{equation}
(This is the inequality that $\chi(\Bbb{P}^1-S, j_*\mathcal{T})\leq 0$ for irreducible $\mathcal{T}$ and $j:\Bbb{P}^1-S\to \Bbb{P}^1$, and follows from Poincare duality, \cite[Section 1.1]{Katz}.)

These two conditions do not seem to put an upper bound on $\ell$ in terms of $n$ by the following argument. Fix $s=3$ for concreteness. We can then put \eqref{irreducible1} in the following form
\begin{equation}\label{irreducible2}
\sum_{i=1}^3 n_{i,0}\leq \ell,
\end{equation}
where $n_{i,0}$ is the rank of the zero eigenspace of $A_i$.
Consider ``two extreme cases":
\begin{enumerate}
\item All local monodromies are the same with the multiplicity of each $n$th root of unity the same.
In this case \eqref{irreducible2} is valid  for all cyclic twists since $3(\ell/n)\leq \ell$ if $n\geq 3$. The left hand side of  the equality \eqref{eqRi} would then be $3\ell^2/n$ which is less than $\ell^2+2$ (since $n>3$).
\item There are only two eigenvalues at each of the  three points, with multiplicities $\ell/2$ and $\ell/2$ each. The left hand side of the equality \eqref{eqRi} would then be $3\ell^2/2$ which is greater than $\ell^2+2$ for  $\ell$ large. For \eqref{irreducible2} to fail, all three matrices must have a common eigenvalue (after central twisting). But that does not seem to be forced by the numerical requirements.
\end{enumerate}
Therefore assuming \eqref{irreducible2}, and  that $\ell$ is large  respect to $n$, it is not the case that equality \eqref{eqRi} always fails in the same direction. One could also check whether Katz' algorithm
applied to the data formally produces a rank which is $\geq 0$. When $s=3$, the rank is of the form $2\ell -(\ell_1 +\ell_2+\ell_3)$ where $\ell_i$ are ranks of suitable eigenspaces at the three points (see e.g., the proof of Theorem III.9.5 in \cite{MM}). In both of the extreme cases this number is positive.

\section{Formula for divisor classes}\label{enume}
Let $\vec{J}=(J^1,\dots,J^s)$ be an $s$-tuple of subsets  of $[n]$ each of cardinality $r$. Recall from Definition \ref{cyclist} that the effective cycle $E=C(d,r,\mn,n,\vec{J})$ on $\Par_{n,\mn,S}$ is defined as a pushforward of
$\Omega(d,r,\mn,n,\vec{J})$ (see Definition  \ref{mindful}).  We will compute $E$ when the codimension of this class is one, i.e.,
\begin{equation}\label{cdc}
dn +\deg(\mn)r +r(n-r)-\sum_{i=1}^s |\sigma_{J^i}|=-1.
\end{equation}
We will use the following enumerative computation to prove Theorem \ref{brahms} in Section \ref{rumble}. It is also used to compare KZ local systems to strange duals of F-line bundles in Section \ref{EqKZ}. Recall the generalization of Gromov-Witten numbers given in  Definition \ref{GWgen}.
\begin{proposition}\label{enumerative}
Assume the codimension condition \eqref{cdc} and
write
$$\mathcal{O}(E)=\mathcal{B}(\vec{\lambda},\ell).$$
Write $\lambda^i=\sum_{b=1}^{n-1} c_{i}^{b}\omega_b$, where $\omega_b$ is the $b$th fundamental weight.
\begin{enumerate}
\item[(A)] Let $J^{s+1}=\{1,n-r+1,n-r+2,\dots,n-1\}$. Then
$$\ell=\langle\sigma_{J^1}, \sigma_{J^2},\dots,\sigma_{J^s},\sigma_{J^{s+1}}\rangle_{d+1,D}.$$
\item[(B)] $c_{i}^{b}=0$ if $b\not\in J^i$, or $b+1\in J^i$. If $b\in J^i$ and $b+1\not\in J^i$, define subsets $J'^1,\dots,J'^s$ of $[n]$
 each of cardinality $r$ as follows: $J'^k=J^k$ if $k\neq i$, and $J'^i=(J^i-\{b\})\cup \{b+1\}$. Then
$$c_i^{b}= \langle\sigma_{J'^1}, \sigma_{J'^2},\dots,\sigma_{J'^s}\rangle_{d,D}.$$
\item[(C)] If $1\leq b<n-1$ and $i\in[s]$, then $c^b_i\neq 0$ implies $c^{b+1}_i=0$.
\item[(D)] If $c^1_i\neq 0$ then $\ell=\sum_{b=1}^{n-1} c^b_i$.  Similarly if $c^{n-1}_i\neq 0$ then $\ell=\sum_{b=1}^{n-1} c^b_i$.
\end{enumerate}
\end{proposition}
We prove this proposition in the rest of this section. The  proof is modelled after the proofs of analogous enumerative counts in \cite{BHermit,BKiers,Kiers}.
\subsection{Formula (B) implies Formula (A)}
Add a new point $p_{s+1}$ and let $J'^{s+1}=\{n-r+1,\dots, n\}$, and $J'^i=J^i$ for $1\leq i \leq s$. Then, by Lemma \ref{propagation} below,
$$\mathcal{O}(C(d,r,\mathcal{O}(D),n,\vec{J'}))=\mathcal{D}^{\ell}\tensor\tensor_{i=1}^s\ml_{\lambda^i,p_i}\tensor \ml_{\vec{0},p_{s+1}}.$$
Now, by Proposition \ref{easily},
$$\operatorname{ISh}_{p^{s+1}}^*\mathcal{O}(C(d,r,\mathcal{O}(D),n,\vec{J'}))=\mathcal{D}^{\ell}\tensor\tensor_{i=1}^s\ml_{\lambda^i,p_i}\tensor\ml_{\lambda{s+1},p_{s+1}}$$
where $\lambda^{s+1}=(0,0,\dots,-\ell)$.

 $\operatorname{Sh}_{p^{s+1}}$ maps $C(d,r,\mathcal{O}(D),n,\vec{J'})$ isomorphically to $C(d,r,\mathcal{O}(D-1),n,\vec{K})$,
where $K^i=J^i$ for $1\leq i \leq s$ and $K^{s+1}=\{n-r,n-r+1,\dots,n-1\}$. Therefore,
$$\operatorname{ISh}_{p^{s+1}}^*C(d,r,\mathcal{O}(D),n,\vec{J'})=C(d,r,\mathcal{O}(D-1),n,\vec{K}).$$

Therefore we obtain the formula
$$\mathcal{O}(C(d,r,\mathcal{O}(D-1),n,\vec{K}))= \mathcal{D}^{\ell}\tensor\tensor_{i=1}^s\ml_{\lambda^i,p_i}\tensor\ml_{\lambda^{s+1},p_{s+1}}$$
where $\lambda^{s+1}=(0,0,\dots,-\ell)$. We have another formula for $\ell= \lambda_{n-1}^{s+1}-\lambda_{n}^{s+1}$ using  (B):
\begin{equation}\label{ABL1}
\ell= \lambda_{n-1}^{s+1}-\lambda_{n}^{s+1}=\langle\sigma_{I^1}, \sigma_{I^2},\dots,\sigma_{I^s},\sigma_{I^{s+1}}\rangle_{d,D-1}
\end{equation}
with $I^j=K^j$ for $1\leq j\leq s$ and $I^{s+1}=\{n-r,\dots,n-2,n\}$. We can use the shift equalities, see Remark \ref{stuffed},  to get the stated formula for $\ell$.

\begin{remark}\label{earlyAM}
We also get formulas for $\ell$ without adding new points.  Pick any $p_i$. then $\ell-(\lambda^i_{1}-\lambda^i_{r})=0$ if either $1\in J^i$ or $n\not\in J^i$. If $1\not\in J^i$ and $n\in J^i$, then
\begin{equation}\label{ABL4}
\ell-\sum_{b=1}^{n-1}c^b_i=\ell-(\lambda^i_{1}-\lambda^i_{r})=\langle\sigma_{J'^1}, \sigma_{J'^2},\dots,\sigma_{J'^s}\rangle_{d+1,D}
\end{equation}
where $J'^j=J^j$ for $j\neq i$ and $J'_i= (J^{i}-\{n\})\cup \{1\}$. Parallel to \eqref{ABL1}, we can shift this to
$$\ell-(\lambda^i_{1}-\lambda^i_{r})=\langle\sigma_{I^1}, \sigma_{I^2},\dots,\sigma_{I^s}\rangle_{d,D-1}$$
where $I^j=J^j$ for $j\neq i$ and $I^i= \{a\mid a=n\text{ or } a+1\in J^i\}$. Recall that $\lambda^i_1-\lambda^i_r=\sum_{b=1}^{n-1} c^b_i$.
\end{remark}

\begin{lemma}\label{propagation}
 Let $p_{s+1}\in \Bbb{P}^1-S$, and set $I^j=J^j$ for $\leq j\leq s$ and  $I^{s+1}=\{n-r+1,n-r+2,\dots,n\}$ so that $\sigma_{I^{s+1}}=1\in H^0(\Gr(r,n),\Bbb{Z})$. Let $S'=S\cup\{p_{s+1}\}$ and $\tau: \Par_{n,\mn,S'}\to \Par_{n,\mn,S}$ be the natural map which forgets the quasi-parabolic structure at $p_{s+1}$. Then,
 $$\tau^*C(d,r,\mn,n,\vec{J})=C(d,r,\mn,n,\vec{I}).$$
 \end{lemma}
\subsection{Proof of Formula (B)}
Consider a curve $C\leto{\sim} \Bbb{P}^1$ in $\Par_{n,\mn,S}$ consisting of points $(\mw,\me,\gamma)$, with $\mw$ (general and fixed) and $\gamma$ fixed, and all flags $\me^p$ fixed except when $p=p_i$. The flag $\me^{p_i}$ is the only variable, and $E^{p_i}_k$ is fixed except for $E^{p_i}_{b}$ which varies in a one dimensional family, with $E^{p_i}_{b-1}\subset E^{p_i}_b\subset E^{p_i}_{b+1}$. Here $E^p_0=0$. The number $c^i_b$ is the degree of $\mathcal{O}(E)$ on this curve. If $b\not\in J^i$ or $b+1\in J^i$  the rank condition at $b$ is not needed in the set of conditions defining the closed Schubert variety $\Omega_{J^i}(\me_p)$. It is now easy to see that, in this case, $E$ does not meet $C$ (for general choices of the fixed elements above) and the degree is zero, as desired.

Now suppose $b\in J^i$ and $b+1\not\in J^i$ and $J'^1,\dots,J'^s$ be as in the statement of (B). Let  $\Omega=\Omega(d,r,\mn,n,\vec{J})$ and $\Omega'=\Omega(d,r,\mn,n,\vec{J}')$. Then $\Omega\subseteq \Omega'$ and let $\pi:\Omega'\to \Par_{n,\mn,S}$ be the natural projection. Now $\pi$ is a generically finite map of degree $m=\langle\sigma_{J'^1}, \sigma_{J'^2},\dots,\sigma_{J'^s}\rangle_{d,D}$.  The fibers of $\Omega'$ over points of $C$ are finite sets of subbundles of $\mw$. These sets do not vary with the point on $C$. Therefore $\pi^{-1}(C)$ is $m$ copies of $\Bbb{P}^1$. The variation in these $\Bbb{P}^1$'s is only in the  $E^{p_i}_b$ elements of the flag (at $p_i$). Each of these $\Bbb{P}^1$'s meet $\Omega\subseteq \Omega'$ in exactly one point (and transversally so) by \cite[Lemma 8.8]{BHermit}. Therefore, $c^i_b =\deg(\pi_*[\Omega]\cap C)= \deg([\Omega]\cap \pi^{-1}(C))=m$, as desired.
\subsection{Proof of parts (C) and (D)}
Part (C) is immediate from the formula in (B) because $c_i^b\neq 0$ implies $b+1\notin J^i$ hence
$c_i^{b+1}=0$. Part (D) follows similarly from Remark \ref{earlyAM}.

\begin{remark}\label{ajabaja}
If $r=1$ in Proposition \ref{enumerative}, then $\ell=1$ as well since the only numbers that appear in the small quantum products of projective spaces are $0$ and $1$. In this case, if $D=0$ by normalization, it is easy to see that $\mathcal{O}(E)=\mathcal{B}(\vec{\lambda},1)$ with $\vec{\lambda}=(\omega_{a_1},\dots,\omega_{a_s})$ with $n|\sum a_i$.
\end{remark}
\subsection{Divisors of F-line bundles}
Let $\ml=\mathcal{B}(\vec{\lambda},\ell)$ be an F-line bundle  on $\Par_{n,\mathcal{O},S}$. Write it as $\mathcal{O}(E)$ for a reduced irreducible divisor  $E\subseteq \Par_{n,\mathcal{O},S}$.
Write $\lambda^i=\sum_{b=1}^{n-1} \gamma_{i}^{b}\omega_b$.

\begin{lemma}\label{powercord}
The following are equivalent :
\begin{enumerate}
\item $\mathcal{L}$ is of the form $\mathcal{O}(D(a,j))$ for some choice of $(a,j)$ in Theorem \ref{Adagio}.
\item $\gamma^j_b$ =1 for some choice of $j,b$, or $\ell-\sum_{b=1}^{n-1} \gamma^j_b=1$ for some $j\in[s]$.
\end{enumerate}
\end{lemma}
\begin{proof}
It is easy to see that (1) implies (2) if $a\neq 1$, since  $\gamma^j_{a-1}=1$. If $a=1$, then we use Remark \ref{earlyAM}.

Assume (2). By Lemma \ref{divisorE}. $E$ is equal to an effective cycle $C(d,r,\mathcal{O},n,\vec{J})$ (Definition  \ref{cyclist})  of codimension one i.e., $dn  +r(n-r)-\sum_{i=1}^s |\sigma_{J^i}|=-1$. Applying Proposition \ref{enumerative}, we find an enumerative formula for $\ml$ in terms of the data of the cycle $C(d,r,\mathcal{O},n,\vec{J})$. If $\gamma^j_b=1$ for some $j,b$ take $I^i=J^i$ for $i\neq j$ and $I^j=(J^j-\{b\})\cup\{b+1\}$ and apply Theorem \ref{Adagio} for the cycle data $(d,r,\mn,I^1,\dots,I^s)$, and consider the pair $(b+1,j)$   If $\ell-\sum_{b=1}^{n-1} \gamma^j_b=1$, we can take $I^j=(J^j-\{n\})\cup\{1\}$, and the pair $(1,j)$, and replace $d$ by $d+1$ and use Remark \ref{earlyAM}.
\end{proof}

\section{Unitary irreducible rigid local systems and KZ equations}\label{EqKZ}
\subsection{KZ equations on conformal blocks}\label{chocolate}
 Let $\vec{\nu}=\nu^1,\dots,\nu^s,\nu^{s+1}$ be an $(s+1)$-tuple of dominant integral weights of $\mathfrak{sl}_r$ at level $k$ (a representation $\nu=\sum_{i=1}^{r-1} c_i\omega_i$ is said to be of level $k$ if $\sum_{i=1}^{r-1}c_i\leq k$). Assume $\sum_{i=1}^{s+1} \nu^i$ is in the root lattice, i.e., $\sum_{i=1}^{s+1}|\nu^i|$ is divisible by $r$. Let $\mathcal{A}_{s+1}$ be the configuration space of $s+1$ distinct points on $\Bbb{A}^1$:
$$\mathcal{A}_{s+1}=\{(z_1,\dots,z_s,z_{s+1})\mid z_i\neq z_j, 1\leq i<j \leq s+1\}\subseteq \Bbb{A}^{s+1}.$$
There is a bundle of conformal blocks $\mathcal{V}_{\vec{\nu}}$ on $\mathcal{A}_{s+1}$ \cite{tuy,sorger}. These bundles are quotients of the (constant bundles) of  coinvariants in the tensor product of the finite dimensional  representations $V_{\nu^i}$ of $\mathfrak{sl}_r$ of highest weight $\nu^i$. There is a flat logarithmic connection on $\mathcal{V}_{\vec{\nu}}$ over $\mathcal{A}_{s+1}$ without any projective ambiguities.
\begin{remark}
The fiber of $\mathcal{V}_{\vec{\nu}}$ over $(z_1,\dots,z_{s+1})$ is dual to $H^0(\Par_{r,\mathcal{O},S},\ml)$ where $S=\{z_1,\dots,z_{s+1}\}$, and $\ml=\mathcal{B}(\vec{\nu},k)$ \cite{pauly}. The following additional properties of conformal blocks are included for completeness and are not needed in this paper.
\begin{enumerate}
\item Let $e_{\theta}$ be the root operator corresponding to the highest root $\epsilon_1-\epsilon_r$ of $\mathfrak{sl}_r$. Let $\Bbb{V}= V_{\nu_1}\tensor V_{\nu_2}\tensor\dots \tensor V_{\nu_{s+1}}$ and let $T:\Bbb{V}\to \Bbb{V}$ be the operator $T=\sum_{i=1}^{s+1} z_i e_{\theta}^{(i)}$ with $e_{\theta}^{(i)}$ acting on the $i$th tensor coordinate.  Then the fiber of $\mathcal{V}_{\vec{\nu}}$
    at $(z_1,\dots,z_{s+1})$ is the quotient
    $$\frac{\Bbb{V}}{\mathfrak{sl}_r\cdot \Bbb{V} +\im T^{k+1}}$$ with the natural diagonal action of the Lie algebra $\mathfrak{sl}_r$  on $\Bbb{V}$.

\item The connection on $\mathcal{V}_{\vec{\nu}}$ is induced by a connection on the constant bundle with fibers $\Bbb{V}$. The  operator corresponding to $\frac{\partial}{\partial z_i}$ acts on a constant tensor $v_1\tensor v_2\tensor\dots\tensor v_{s+1}$ as follows:
    $$\frac{\partial}{\partial z_i}(v_1\tensor v_2\tensor\dots\tensor v_{s+1})=
    -\frac{1}{r+k}\sum_{1\leq j\leq s+1, j\neq i}\frac{ {\Omega_{i,j}}(v_1\tensor v_2\tensor\dots\tensor v_{s+1})}{z_i-z_j}.$$
Here $\Omega_{i,j}$ is the Casimir operator acting on the $i$ and $j$ tensor summands (see \cite[(5.40)]{Ueno} for more details).
\end{enumerate}
\end{remark}
\subsubsection{Local monodromy}\label{clarinet}
Let us pull back the KZ system to $\Bbb{A}^1-\{z_1,\dots,z_s\}$, under the map $t\mapsto (z_1,\dots,z_s,t)\in\mathcal{A}_{s+1}$, so that $z_{s+1}$ is the moving point and the other $z_i$ are fixed. The monodromies of this pullback local system
on  $\Bbb{A}^1-\{z_1,\dots,z_s\}$ are the following (see the computations of \cite[Section 3.1]{naf}, placing the trivial representation at $\infty$). Recall that the local monodromy is  the exponential of $2\pi\sqrt{-1}$ times the residue of the logarithmic connection (since the local monodromy is diagonalizable in our setting).
\begin{itemize}
\item[-] The residue around $\infty$ is central with residue  $\frac{c(\mu_{s+1})}{r+k}$ times the identity matrix (see Definition  \ref{ls} below).
\item[-] To compute the monodromy as $t$ goes around the point $z_i$, for ease of notation assume $z_i=z_s$. Note that there is a rank formula
$$ \dim \mathcal{V}_{\vec{\nu}}=\sum_{\gamma}\rk \mathcal{V}_{\{\nu^1,\dots,\nu^{s-1},\gamma\}} \rk \mathcal{V}_{\{\nu^s,\gamma^*,\nu^{s+1}\}}.$$
The local exponents (of the residue)  around $t=z_s$ are (see Definition  \ref{ls} below)  $\frac{1}{2(r+k)}(c(\gamma)-c(\nu^s)-c(\nu^{s+1}))$ with multiplicity equal to $\rk \mathcal{V}_{\{\nu^1,\dots,\nu^{s-1},\gamma\}} \rk \mathcal{V}_{\{\nu^s,\gamma^*,\nu^{s+1}\}}$.
\end{itemize}

\begin{defi}\label{ls} Let $\frh_r$  be the Cartan subalgebra of $\mathfrak{sl}_r$. For $\lambda\in\frh_r^*$ let
$c(\lambda)=(\lambda,\lambda+2\rho)$ where $(\ ,\ )$ is the normalized Killing for on $\frh_r^*$ (i.e., all roots $\alpha$ have $(\alpha,\alpha)=2$)  and $\rho\in\frh_r^*$ the half sum of positive roots.
\end{defi}
\subsection{Unitary irreducible rigid local systems come from KZ equations}\label{KZstuff}
Let $E=C(d,r,\mathcal{O},n,\vec{J})$ be an effective cycle of codimension $1$ on $\Par_{n,\mathcal{O},S}$ (see Definition  \ref{cyclist}). Let $\ml=\mathcal{O}(E)=\mathcal{B}(\vec{\lambda},\ell)$ be the corresponding effective line bundle on $\Par_{n,\mathcal{O},S}$.  Write $\lambda^i=\sum_{b=1}^{n-1} c^b_i\omega_b$. Let
$\ml'=\mathcal{B}(\vec{\mu},n)$ be the  line bundle on $\Par_{\ell,\widetilde{\mathcal{N}},S}$ as in Section \ref{SD}. We have a local system (L1) and a nonempty family of local systems (L2) in this situation.
\begin{enumerate}

\item[(L1)] The first local system is the KZ local system on $\Bbb{A}^1-\{z_1,\dots,z_s\}$ corresponding to the data $\nu^1,\dots,\nu^{s+1}$ with $\nu^{s+1}=\omega_1$ (as in Section \ref{chocolate}) at level $k=n-r$.
The weights $\nu^i$ are defined as follows. First define $\mu^i=\lambda(J^i)$ for $j\neq1$, and $\mu^1$ the $d$-fold shift at level $k=n-r$ of $\lambda(J^1)$ (i.e., a single shift replaces $\mu=(\mu_1,\dots,\mu_r)$ by $(\mu_2,\dots,\mu_r,\mu_1-k)$. These are Young diagrams that fit into an $r\times (n-r)$ box. Let $\nu^i$ be their (ordinary) duals, $i=1,\dots,s$, i.e $\nu^i_j= (n-r)-\mu^i_{r+1-j}$.
\item[(L2)]The effective line bundle $\ml'$ produces a non-empty family of  rank $\ell$ unitary local systems on $\Bbb{P}^1-(S\cup\{\infty\})$ as in Corollary \ref{sansD} applied with $\ell$ and $n$ interchanged (note that the local monodromy at $\infty$ is trivial).  Let $m$ be the remainder when $d$ is divided by $r$. Twist this family of local systems by $\zeta_n^{-j^1_m}$ at the point $p_1$ and $\zeta_n^{j^1_{m}}$ at $\infty$ (set $j^1_m=0$ if $m=0$). This twisted family is (L2).
\end{enumerate}

We will show that these two local systems (L1) and (L2) have essentially equivalent local monodromy up to central twists (in particular these have the same ranks), so that the KZ local system
(up to twists by rank one local systems) lies in the family (L2). More precisely
\begin{proposition}\label{generalKZ}
Twist the local system in (L1) by $\exp(-2\pi\sqrt{-1}|\mu^i|/rn)$ at the points $p_i$ and $\exp(2\pi\sqrt{-1}\sum_{i=1}^s |\mu^i|/rn)$ at infinity. This twisted local system is in the family of local systems with fixed local monodromy given by (L2).
\end{proposition}
Here twisting means that we multiply the local monodromy elements by the central elements given.
Proposition \ref{generalKZ} will be proved in Section \ref{KZe}. It implies the following:
\begin{theorem}\label{KZe}
Let $\rigid$ be a unitary and irreducible rigid rank $\ell>1$ local systems such that the eigenvalues of the local monodromies are  $n$th roots of unity.  Then, there exists $1<r<n$ and a KZ local system on $\Bbb{P}^1-S\cup\{\infty\}$, for $\mathfrak{sl}_r$ at level $n-r$, with central monodromy
around infinity, which is, up to a twist by  $rn$ roots of unity, isomorphic to $\rigid$.
\end{theorem}
\begin{proof}
Let $\rigid$ be a unitary irreducible  rank $\ell>1$ rigid local system on $\Bbb{P}^1-S$, $S=\{p_1,\dots,p_s\}$. Assume $S\subseteq \Bbb{A}^1$. Let $\mathcal{A}$ be the conjugacy class data of
$\rigid$, and $\ml=\mathcal{B}(\vec{\lambda},\ell)$ be the corresponding F-line bundle on $\Par_{n,\mathcal{O},S}$ (Theorem \ref{christmas}(ii)).  By Lemma \ref{divisorE} $\ml=\mathcal{O}(E)$ with $E$ for a reduced irreducible divisor of the form  $E=C(d,r,\mathcal{O},n,\vec{J})$ (Definition  \ref{cyclist}), a cycle  of codimension one i.e., $dn +r(n-r)-\sum_{i=1}^s |\sigma_{J^i}|=-1$. The theorem now follows from Proposition \ref{generalKZ} applied to $C(d,r,\mathcal{O},n,\vec{J})$ .
\end{proof}

\begin{remark}\label{hodge}
The  works  \cite{Ram,BKZ,BM}   realise conformal blocks as suitable  isotypical spaces (under the action of a  finite group) in the $H^{M,0}$ of the cohomology of smooth projective varieties, consistent with connections. They do not however show that conformal blocks carry an integral Hodge structure over the ring of integers of a cyclotomic field (i.e., the part of $H^{M,0}$ is not shown to be a Hodge substructure of $H^M$ over the ring of integers of a number field). Note that Katz's motivic realizations are over the ring of integers of a cyclotomic field \cite[Theorem 8.4.1]{Katz}. It seems likely however that the conformal blocks picture should also be over the ring of integers of a cyclotomic field (see \cite[Question 6.2]{Loo}).
\end{remark}

The works \cite{BH} and \cite{Haraoka} produce an infinite collection of unitary rigid (irreducible) local systems with infinite monodromy, and therefore the corresponding KZ connections have infinite monodromy groups.
\begin{example}
For  the strange dual of Thaddeus' example (see Section \ref{E1}), we see that the KZ local system on $\mathcal{A}_4$ with  $\mathfrak{sl}_4$ and the representations $\lambda^1=2\omega_2+\omega_3$,
$\lambda^2=\lambda^3=\omega_1+2\omega_2$ and $\lambda^4=\omega_1$ (and the trivial representation at infinity) has infinite monodromy (even infinite projective monodromy). Examples of TQFT representations (which are equivalent to the monodromy of KZ equations) with infinite monodromy were first given in \cite{Masbaum}.
\end{example}
\subsection{Proof of  Proposition \ref{generalKZ}}
\subsubsection{First steps}\label{atonal}
 Since the codimension of the divisor
$E=C(d,r,\mathcal{O},n,\vec{J})$ on $\Par_{n,{\mathcal{O}},S}$ is one, we have
 \begin{equation}\label{ccdd}
\sum_{i=1}^s |\sigma_{J^i}|=dn +r(n-r)+1.
\end{equation}
Consider the shifted F-line bundle on $\Par_{n,\mathcal{O}(-D),n}$ obtained by applying the shift operation at $p_1$ $j^1_m$-times. Here $D=-j^1_m$. This F-line bundle has divisor  $E'=C(0,r,\mathcal{O}(-D),n,\vec{K})$ where  $K^1$ is the $j^1_m$ fold cyclic shift of $J^1$ (the cyclic shift applied once subtracts one from all elements of ${J}^1$ with zeros replaced by $n$). Let $K^j=J^j$ for $j>1$. Write $\mathcal{O}(E')= \mathcal{B}(\vec{\lambda},\ell)$.
We note the following:
\begin{enumerate}
\item [(I)]  $\mu^i$ in the local system (L1) are equal to $\lambda(K^i)$ (see Definition  \ref{correspondence}),
\item[(II)](L2) is the rank $\ell$ local system with local monodromy data (with eigenvalues $n$th roots of identity) given by the local monodromy  data $(\lambda^i)^T$ at the points $p_i$ and
$\zeta_n^{j^1_{m}}I_{\ell}$ at infinity.
\item[(III)] Formulas for  $\ell$ and $\vec{\lambda}$ are given in Section \ref{enume}, we recall them here: write
$\lambda^i=\sum_{b=1}^{n} c_{i}^{b}\omega_b$, where $\omega_b$ is the $b$th fundamental weight.
Also note that \eqref{ccdd} gives
\begin{equation}\label{cccddd}
\sum_{i=1}^s |\sigma_{K^i}|=rj^1_m +r(n-r)+1.
\end{equation}
 \end{enumerate}

\begin{enumerate}
\item[(A)] Let $K^{s+1}=\{n-r,n-r+1,\dots,n-2,n\}$. Then (see Formula \ref{ABL1})
$$\ell=\langle\sigma_{K^1}, \sigma_{K^2},\dots,\sigma_{K^s},\sigma_{K^{s+1}}\rangle_{0,D-1}.$$

\item[(B)] Suppose $b<n$. then $c_{i}^{b}=0$ if $b\not\in K^i$, or $b+1\in K^i$. If $b\in K^i$ and $b+1\not\in K^i$, define subsets $K'^1,\dots,K'^s$ of $[n]$
 each of cardinality $r$ as follows: $K'^k=K^k$ if $k\neq i$, and $K'^i=(K^i-\{b\})\cup \{b+1\}$. Then
$$c_i^{b}= \langle\sigma_{K'^1}, \sigma_{K'^2},\dots,\sigma_{K'^s}\rangle_{0,D}.$$
\item[(C)] For $b=n$, the formula is given in Remark \ref{earlyAM}, which we recall here:
$$\ell-(\lambda^i_{1}-\lambda^i_{r})=\langle\sigma_{I^1}, \sigma_{I^2},\dots,\sigma_{I^s}\rangle_{0,D-1}$$
where $I^j=K^j$ for $j\neq i$ and $I^i= \{a\mid a=n\text{ or } a+1\in K^i\}$.
\end{enumerate}
\subsubsection{Rank equality verification of the two local systems}\label{listy}
We first verify that the ranks are the same.
We can readily convert the formula (A) into the rank of a conformal block for $\mathfrak{sl}_r$  at level $k=n-r$ using Lemma \ref{newlabel2}. Note that $\lambda(K^{s+1})=\omega_{r-1}$. Passing to (ordinary) dual weights (which does not change the rank of conformal blocks), we see that $\ell$ equals the rank of the  KZ local system on $\Bbb{A}^1-\{z_1,\dots,z_s\}$ corresponding to the data $\nu^1,\dots,\nu^{s+1}$ with $\nu^{s+1}=\omega_1$.
\subsubsection{The local monodromy of the KZ equation}
We now specialize  the formulas of Section \ref{clarinet} to the case $\nu^{s+1}=\omega_1$ to compute the local monodromies of the KZ system (L1). The following formulas facilitate this computation:
\begin{enumerate}
\item[(a)] If $\mu=(\mu_1,\dots,\mu_r)$, and $\nu=(\nu_1,\dots,\nu_r)$, then $(\mu,\nu)=(\sum_{i=1}^r\mu_i\nu_i) -\frac{|\mu||\nu|}{r}$. Here $|\mu|=\sum \mu_i$ (similarly $|\nu|$).
\item[(b)] $c(\omega_1)=(r^2-1)/r$.
\item[(c)]  Assume $\nu_r=0$. Then $\rk \mathcal{V}_{\nu,\gamma^*,\omega_1}=1$ if  $\gamma$ is of the form $\nu+L_i$ for $1\leq i<r$ (and $k\geq\gamma_1\geq\dots\geq \gamma_r=0$), and zero otherwise (see \cite{Gepner}). If $\nu_r>0$, the above implies that  we will have to allow one more value of $\gamma$ if in addition  $\nu_1=k$. This entry is $\nu+L_1-(1,\dots,1)$. Note that $c(\nu+L_1-(1,\dots,1))=c(\nu+L_1)$.
\item[(d)] This motivates the following computation keeping in mind the second residue formula above in Section \ref{clarinet}: $c(\nu+L_i)-c(\nu)-c(\omega_1)= 2(\nu_i -\frac{|\nu|}{r} -i +1)$.
\end{enumerate}

We first compute the monodromy of the KZ system first around $t=z_s$. The $\gamma$ that appear are of three types (use item (c) above):
\begin{itemize}
\item[-] $\gamma=\nu^s+L_j$ with  $j>1$, $\nu^s_{j-1}> \nu^s_j$. This just means that $\mu^s_{r+2-j}<\mu^s_{r+1-j}$,
i.e, $k^s_{r+2-j}> k^s_{r+1-j} +1$, or that if $b=k^i_{r+1-j}$, then $b\in K^i$ and $b+1\not\in K^i$. The residue of the part corresponding to $\gamma$ (of multiplicity  $\rk \mathcal{V}_{\nu^1,\dots,\nu^{s-1},\gamma}$) is $\frac{1}{r+k}=\frac{1}{n}$ times
$$(\nu^s_j -\frac{|\nu|}{r} -j +1)= (n-r)-\mu^s_{r+1-j}-\frac{|\nu|}{r} -j +1$$
using the formula $\mu^s_{n+r-i}=(n-r)+(r+1-j)-k^s_{r+1-j}=n+1-k^s_{r+1-j}$ and $\frac{|\nu|}{r}=n-r-\frac{|\mu|}{r}$, we see that this residue is equal to $\frac{1}{n}$ times $k^s_{r+1-j}-n +\frac{|\mu^s|}{r}$.
\item[-] $\gamma=\nu^s+L_1$ with $\nu^s_{1}< (n-r)$, i.e., $k^s_r<n$.  The formulas are similar here and the residue is equal to  $\frac{1}{n}$ times $k^s_{r}-n +\frac{|\mu^s|}{r}$. Let $b=k^s_r$. Then in this case $n\not\in K^j$.
\item[-] $\gamma=\nu^s -(1,1,\dots,1) +L_1$ when $\nu^s_1=n-r$ (i.e., $k^s_r=n$) and $\nu^s_r>0$, i., $k^s_1\neq 1$. Setting $b=k^s_r$, this case corresponds to $b=n$, and $1\not\in K^i$.
In this case the residue is $\frac{1}{n}$ times $\frac{|\mu^s|}{r}$ which can be written in the same form as $k^s_{r}-n +\frac{|\mu^s|}{rn}$.
\end{itemize}
Around $t=\infty$ the residue is a constant $(r^2-1)/rn$. Subtracting  $\frac{|\mu^i|}{rn}$ (the codimension of $\sigma_{K^i}$ divided by $rn$) at the finite points $z_1,\dots,z_s$ and adding $\sum_i|\mu^i|/rn$ at infinity, we see that the monodromy at infinity has residue (using \eqref{cccddd})
$$(r^2-1)/rn +\sum|\mu^i|/rn = \frac{1}{n}((r^2-1)/r +n-r+ \frac{1}{r} +j^1_m) =\frac{j^1_m}{n}+1.$$
The local monodromy at the finite points $z_i$ have residues $b-n, b\in K^i$ (with possibly zero multiplicities). Therefore the equation of the fundamental group of $\Bbb{P}^1-\{z_1,\dots,z_s,\infty\}$ is a product of $\ell\times \ell$ invertible matrices
\begin{equation}\label{KZMatrix}
A_1\dots A_s A_{\infty}=I_{\ell}
\end{equation}
with
\begin{itemize}
\item  $A_i, i\neq \infty$ conjugate to a matrix with eigenvalues (with varying multiplicities)
$$\exp(2\pi\sqrt{-1}(k^i_{r+1-j}-n)/n)= \exp(2\pi\sqrt{-1}k^i_{r+1-j}/n), j=1,\dots,r$$
with multiplicities given by  ranks of conformal blocks $\rk \mathcal{V}_{\nu^1,\widehat{\nu_i}\dots,\nu^{s-1},\gamma}$ with $\nu^i$ omitted and $\gamma=\nu^i+L_j$ (if dominant of level $k=n-r$, count as zero otherwise).
\item $A_{\infty}=\exp(2\pi\sqrt{-1}j^1_m/n)I_{\ell}$.
\end{itemize}
\subsubsection{The local monodromies of (L2)}
By Item (I2) before equation \eqref{cccddd}, (L2) corresponds to matrix equations
\begin{equation}\label{sdMatrix}
B_1\dots B_s B_{\infty}=I_{\ell}
\end{equation}
 where
\begin{itemize}
\item  $B_{\infty}=\exp(2\pi\sqrt{-1}j^1_m/n)I_{\ell}$.
\item For the conjugacy class of $B_i$ we proceed as follows. Note that $\lambda^i=\sum_{b=1}^{n-1} c^i_b\omega_b$. Here the set of $b$ is constrained  by (B) in Section \ref{atonal}. First, the set of $b$ is a subset of $K^i$. It is also easy to see that we need in addition one of the following to hold
\begin{enumerate}
\item  $b=n$ and $1\not\in K^i$.
\item $b<n$ and $b+1\not\in K^i$.
\end{enumerate}
Note that the transpose of $\sum_{b=1}^n c^b_i\omega_b$ has rows of size $b$ with multiplicities $c^b_i$ (and $\sum c^b_i=\ell$).  Pick $b=k^i_{n+r-j}$. The eigenvalues of $B_i$ are therefore of the form $\exp(2\pi\sqrt{-1}k^i_{n+r-i}/n)$ with $b\in K^i$ with multiplicities $c^i_b$ (possibly zero).
\end{itemize}
\subsubsection{Comparison of local monodromies}
We have two matrix equations \eqref{KZMatrix} and \eqref{sdMatrix}. We want to show that the matrices that appear are conjugate.
Clearly $A_{\infty}=B_{\infty}$. It remains to show that $A_i$ is conjugate to $B_i$ for $i=1,\dots,s$.

The eigenvalues of $A_i$ and $B_i$ are both of the form $\exp(2\pi\sqrt{-1}b/n)$ with $b\in K^i$ with multiplicities (possibly zero). We need to verify that the multiplicities are the same.
Let $b=k^i_{r+1-j}$. We need to verify that
\begin{equation}\label{desiree}
c^i_b=\rk \mathcal{V}_{\nu^1,\widehat{\nu_i}\dots,\nu^{s-1},\gamma}
\end{equation}
with $\nu^i$ omitted and $\gamma=\nu^i+L_j$. Unless $b=n$ and $1\not\in K^i$, or $b<n$ and $b+1\not\in K^i$, both multiplicities are zero by the above discussion.

If $b<n$, then $c^b_i$ is given by case (B) in Section \ref{listy}. Let $J'^i=(K^i-\{b\})\cup \{b+1\}$. Let $\mu'=\lambda(J'^i)$ and $\nu'=\mu'^*$. t is easy to see that $\nu'+L_j$, and we use  \cite[Proposition 3.4]{BHorn} again, to write $c^b_i$ as the rank of a conformal block for $\mathfrak{sl}_r$  at level $k=n-r$. This gives \eqref{desiree}.

If $b=n$ and $1\not\in K^i.$ We use formula (C) from section \ref{listy}. Let $I^i= \{a\mid a=n\text{ or } a+1\in K^i\}$. We need to show that $\lambda(I^i)^*=\nu_i +L_1 -(1,1,\dots,1)$ which is easy to check. The rest of the computation is similar to the previous steps.
\section{Compactified parameter spaces}\label{compact}
The remaining aims for this paper are the following: completing the proofs of Theorems \ref{Adagio},\ref{brahms}, and Proposition \ref{condensed}, and  setting up the induction apparatus to induce extremal rays from Levi subgroups giving rise to the maps \eqref{productstructure} and \eqref{inductiones}.

Assume that the generalized Gromov-Witten number $\langle\sigma_{I^1}, \sigma_{I^2},\dots,\sigma_{I^s}\rangle_{d,D}=1$, where $D=-\deg(\mn)$ (Definition  \ref{GWgen} (We use this general set up with $D$ possibly non-zero because we will need to use shift operations to handle the case $1\in I^j$ for some $j$). This implies that $\Omega(d,r,\mn,n,\vec{I})\to \Par_{n,\mn,S}$ is birational and
\begin{equation}\label{timer1}
dn +\deg(\mn)r +r(n-r)-\sum_{i=1}^s |\sigma_{I^i}|=0.
\end{equation}

Recall that  $\Omega(d,r,\mn,n,\vec{I})$ is a compactification over $\Par_{n,\mn,S}$ of $\Omega^0(d,r,\mn,n,\vec{I})$ (which is the set of subbundles of the universal bundle of determinant $\mn$ and rank $n$ with fibers in given Schubert states). $\Omega(d,r,\mn,n,\vec{I})$ has several defects: (1) for a  point $(\mv,\mw,\me,\gamma)$ not in $\Omega^0(d,r,\mn,n,\vec{I})$,  $\mw/\mv$ can have torsion and one does not have an inductive structure of a  subbundle and quotient which are both bundles with induced flags, (2) even when $\mw/\mv$ is locally free the induced flag structures cannot be assigned without discontinuities ,and (3) $\Omega(d,r,\mn,n,\vec{I})$ could be singular.

For the induction map and the proof of Proposition \ref{fulC} we need a relative
partial compactification of $\Omega^0(d,r,\mn,n,\vec{I})\to \Par_{n,\mn,S}$ without these defects, as well as stacks for the relevant Levi subgroups. We will also need to understand the ramification properties of $\Omega(d,r,\mn,n,\vec{I})\to \Par_{n,\mn,S}$.  This will result in a (known) multiplicative generalization of a conjecture of Fulton, which we prove for completeness here (see Proposition \ref{fulC} below).

\subsection{Stacks for Levi subgroups}
\begin{defi}\label{parul}
Let $\ParL=\ParL(-d,r,n-r,d-D)$ be the moduli of tuples $(\mv,\mq,\mf,\mg,\gamma)$ where
\begin{enumerate}
\item $\mv$ (resp. $\mq$) is a degree $-d$ (resp. $d-D$), rank $r$ (resp. rank $n-r$) vector bundle on $\Bbb{P}^1$,
\item  $\mf\in \Fl_S(\mv)$, $\mg\in \Fl_S(\mq)$,
\item $\gamma$ is an isomorphism $\det \mv\otimes \det\mq\leto{\sim}\mn=\mathcal{O}(-D)$.
\end{enumerate}
\end{defi}
There is a diagram
\begin{equation}\label{Dbasic}
\xymatrix{
 & \Omega^0(d,r,\mn,n,\vec{I})\ar[dl]^{\pi}\ar[dr]\\
\Par_{n,\mn,S}   &  & \ParL\ar@/_/[ul]_i\ar[ll]^{i'=\pi\circ i}
}
\end{equation}
Here $i$ is a ``direct sum morphism", $i(\mv,\mq,\mf,\mg,\gamma)= (\mw,\me,\gamma')\in \Par_{n,\mn,S}$ where $\mw=\mv\oplus\mq$, $\gamma': \det(\mw)=\det(\mv)\otimes\det(\mq)\leto{\gamma} \mathcal{O}(-D)$, and $\me$ is given by the following rule. For $p=p_i$, write $I^i=\{i_1>i_2>\dots>i_r\}$. Define $E^p_{\bull}$ as follows: For  $1\leq k\leq n$, define
$$E^p_{j}=F^p_{a}\oplus G^p_{j-a}$$
where  $i_a\leq j <i_{a+1}$ (with $i_0=0$ and $i_{r+1}=n$). Let $\mathcal{R}\subseteq  \Omega^0(d,r,\mn,n,\vec{I})$ be the ramification locus of $\pi$ (see Section \ref{RamDiv}).

\subsection{Desired properties of partial compactifications}
One of our basic techniques, inspired by Ressayre's work \cite[Section 4.1]{R1} (see \cite[Section 8.6]{BHermit} for a comparison)  is to take line bundles on $\ParL$, and lift them to $\Omega^0(d,r,\mn,n,\vec{I})$.  By Zariski's main theorem, $\Omega^0(d,r,\mn,n,\vec{I})-\mathcal{R}$ is an open substack of $\Par_{n,\mn,S} $, and we thus get a line bundle on an open substack of $\Par_{n,\mn,S} $.

\begin{remark}\label{explainmore}
This open substack misses all the basic divisors $D(a,j)$ from the introduction (in the case $\mn=\mathcal{O}$).
\end{remark}
By the above remark the complement of this open subset may have codimension 1 components, and we cannot then naturally extend the line bundle to  $\Par_{n,\mn,S}$. One looks for an enlargement of   $\Omega^0=\Omega^0(d,r,\mn,n,\vec{I})$ that does not have this defect (and which still admits maps to $\Par_{n,\mn,S} $ and $\ParL$).

In fact this happens in the simplest possible fashion whenever $1\not\in I^i$ for any $i$, and the shift operations allow us to reduce the general case to this case. This partial compactification will be inside $\Omega=\Omega(d,r,\mn,n,\vec{I})$.

 We do this in two steps: First we enlarge to $\widehat{\Omega}^0\to \Par_{n,\mn,S}$ where we only weaken the Schubert conditions, the subsheaves remaining subbundles. The map   $\Omega^0\to \ParL$ will be shown to extend to $\widehat{\Omega}^0\to \ParL$ when $1\not\in I^i$ for any $i$.

The second step will extend this to a substack  $\widehat{\Omega}_{S}^0\subseteq \Omega$ over $\Par_{n,\mn,S}$, where the subbundle condition is imposed only at points of $S$. We will then have
$$\Omega^0\subseteq \widehat{\Omega}^0\subseteq \widehat{\Omega}_{S}^0\subseteq \Omega.$$ The map $\widehat{\Omega}^0\to \ParL$ will not be shown to extend to $\widehat{\Omega}_{S}^0\to \ParL$. Instead we will show in Proposition \ref{bee} that   $\widehat{\Omega}_{S}^0-  \widehat{\Omega}^0$ lies inside the ramification divisor $\mathcal{R}$, but  irreducible components of $\Omega-\widehat{\Omega}_{S}^0$ are of codimension $\geq 2$. The desired enlargement of $\Omega^0(d,r,\mn,n,\vec{I})$ will then be $\widehat{\Omega}^0$ (and not $\widehat{\Omega}_{S}^0$!).
\subsection{Construction of the partial compactification}
For convenience, we weaken our assumptions to the following:
Suppose $\vec{I}=(I^1,\dots,I^s)$ is an $s$-tuple of subsets  of $[n]$ each of cardinality $r$. Let $\Omega=\Omega(d,r,\mn,n,\vec{I})$.  The representable morphism $\Omega\to \Par_{n,\mn,S}$ has relative dimension $dn +\deg(\mn)r +r(n-r)-\sum_{i=1}^s |\sigma_{I^i}|$. We will make the assumption that this relative dimension is zero:
\begin{equation}\label{timer}
dn +\deg(\mn)r +r(n-r)-\sum_{i=1}^s |\sigma_{I^i}|=0.
\end{equation}

\subsubsection{Definition of $\widehat{\Omega}^0=\widehat{\Omega}^0(d,r,\mn,n,\vec{I})$}\label{bruckner92}
We create a  smooth open subset of $\Omega(d,r,\mn,n,\vec{I})$ containing $\Omega^0(d,r,\mn,n,\vec{I})$ in the following way.

\begin{defi}\label{ToI}
\hspace{2em}
\begin{enumerate}
\item For a subset $I$ of $[n]$ of cardinality $r$, let $T(I)\subset [n]$ be the subset of all $j\in [n]$ such that one of the following holds $j\in I$ or $j+1\not\in I$ (equivalently $a\not\in T(I)$ if and only if $a\not\in I$ and $a+1\in I$). Note that $T(I)$ contains $n$.
\item In the setting of Definition \ref{ineffable}, define a smooth variety $\widehat{\Omega}^0_I(E_{\bull})$, with $\Omega^0_I(E_{\bull})\subseteq \widehat{\Omega}^0_I(E_{\bull})\subseteq {\Omega}_I(E_{\bull})$ by
$$\widehat{\Omega}^0_I(E_{\bull})=\{V\in \Gr(r,W)\mid \rk (V\cap E_j)=a, \text{ if } i_a\leq j<i_{a+1}\text { and } j\in T(I)\}.$$
It is easily seen that  $\widehat{\Omega}^0_I(E_{\bull})$ carries a transitive action of a connected group and is hence smooth. It is easy to see that all irreducible components of
${\Omega}_I(E_{\bull})-\widehat{\Omega}^0_I(E_{\bull})$ are of codimension $\geq 2$ in ${\Omega}_I(E_{\bull})$, see \cite[Lemma 7.9]{BHermit}.
\end{enumerate}
\end{defi}
Define $\widehat{\Omega}^0=\widehat{\Omega}^0(d,r,\mn,n,\vec{I})\subseteq \Omega(d,r,\mn,n,\vec{I})$ to be the smooth substack parametrizing tuples $(\mv,\mw,\me,\gamma)$ with $\mv$ a subbundle of $\mw$, and $\mv_p\in \widehat{\Omega}^0_{I^i}(\me_p)\subseteq \Gr(r,\mw_p)$ for all $p=p_i\in S$. Clearly, $\widehat{\Omega}^0(d,r,\mn,n,\vec{I})\supseteq {\Omega}^0(d,r,\mn,n,\vec{I})$.
By \cite[Lemma 7.10]{BHermit}), the morphism  $\Omega^0 \to \ParL$ extends to a morphism $\widehat{\Omega}^0\to\ParL$, and we get a commutative diagram
 \begin{equation}\label{morph1}
\xymatrix{
 \Omega^0 \ar[r]\ar[d] &  \widehat{\Omega}^0\ar[dl]\\
\ParL
}
\end{equation}
\subsubsection{Definition of $\widehat{\Omega}_{S}^0=\widehat{\Omega}_{S}^0(d,r,\mn,n,\vec{I})$}

There is an even larger  open subset: Instead of requiring that $\mv$ is a subbundle of $\mw$, we can require only that it is a subbundle near points of $S$. Denote this open subset by $\widehat{\Omega}_{S}^0=\widehat{\Omega}_{S}^0(d,r,\mn,n,\vec{I})$ with the subscript $S$ denoting that the subbundle property is only imposed at points of $S$. Clearly $$\Omega^0\subseteq \widehat{\Omega}^0\subseteq \widehat{\Omega}_{S}^0\subseteq \Omega.$$
\begin{proposition}\label{bee}
\hspace{2em}
\begin{enumerate}
\item[(a)] If $1\not\in I^i$ for all $i$, then   all irreducible components of $\Omega-\widehat{\Omega}_{S}^0$ are of codimension $\geq 2$.
\item[(b)] $\widehat{\Omega}_{S}^0$ is a smooth Artin stack.
\item[(c)] Any point in $\widehat{\Omega}_{S}^0-\widehat{\Omega}^0$ is in the ramification locus of $\widehat{\Omega}_{S}^0\to \Par_{n,\mn,S}$.
\end{enumerate}
\end{proposition}
Part (a) in the above theorem follows from the computations in Sections \ref{bruck02} and \ref{bruck03}.
Note that if $J$ is as in Section \ref{bruck03}, then ${\Omega}^0_J(E_{\bull}) \subseteq \widehat{\Omega}^0_{I^1}(E_{\bull})$, and therefore $\widehat{\Omega}_{S}^0(d,r,\mn,n,\vec{I})$ includes the codimension 1 degenerations of Section \ref{bruck03}.
 $\widehat{\Omega}_{S}^0$ is smooth and representable over $\op{Quot}(d,r,\mn,n)$ which implies (b).
 Part (c) of Proposition \ref{bee} is proved in  Section \ref{martel}.
\subsection{Ramification divisors}\label{RamDiv}
Suppose $\mathcal{X}\leto{\pi} \mathcal{Y}$ is a representable morphism of smooth Artin stacks of the same dimension over $\Bbb{C}$. We construct a natural  ramification Cartier divisor $\mathcal{R}\subseteq \mathcal{Y}$ as follows. Let $U\to \mathcal{Y}$
be an atlas. We get a morphism of smooth schemes of the same dimension
$U\times_{\mathcal{Y}}\mathcal{X}\leto{\pi'} U$. The map on tangent spaces results in  a map of bundles of the same rank on $U\times_{\mathcal{Y}}\mathcal{X}$ given by $T (U\times_{\mathcal{Y}} \mathcal{X}) \to \pi'^* TU$. From this data we obtain  a line bundle $\det \pi'^* TU\tensor \det T (U\times_{\mathcal{Y}}\mathcal{X})^{-1}$ and a section. We denote the corresponding Cartier divisor by $\mathcal{R}$ and the line bundle by $\mathcal{O}(\mathcal{R})$. The line bundle and the section glue in the smooth topology of $\mathcal{X}$.

To prove Part (c) of Proposition \ref{bee}, we introduce a line bundle on $\widehat{\Omega}^0_S(d,r,\mn,n,\vec{I})$ with a section which we will later prove to be equal to $\mathcal{O}(\mathcal{R})$ and the corresponding section. Consider a point
$$x=(\mv,\mw,\me,\gamma)\in \widehat{\Omega}^0_S(d,r,\mn,n,\vec{I}).$$
Here,
\begin{itemize}
\item[--]$\mv\subset \mw$ is a locally free subsheaf. Let $\mathcal{Q}=\mw/\mv$ which is by assumption, locally free near points of $S$.
\item[--]Let $\mathcal{F}$ (respectively $\mathcal{G}$) be the induced flags on the fibers of $\mv$ (respectively $\mathcal{Q}$) at points $p\in S$.
\item[--] For each $p=p_j\in S$ define a subspace $T_p$ of $\Hom(\mv_p,\mq_p)$ as follows:
$$T_p=\{\phi: \mv_p\to \mq_p\mid \phi(F^p_{k})\subseteq G^p_{i_j^{k}-k}, k=1,\dots,r \}.$$
Note that $T_p$ is the tangent space of the Schubert cell ${\Omega}^0_{I^j}(E_{\bull})$ of $\Gr(r,\mw_p)$ at $\mv_p\subseteq \mw_p$ and has dimension $|\sigma_{I^j}|$.
\end{itemize}
Consider the kernel $\mk_x$ of the natural surjective mapping of sheaves on $\Bbb{P}^1$  (the right hand side is a sum of skyscraper sheaves)
\begin{equation}\label{sky}
\shom(\mv,\mathcal{Q})\to \bigoplus_{p\in S}\frac{\Hom(\mv_p,\mq_p)}{T_p}\mid_p.
\end{equation}
The Euler characteristic $\chi(\Bbb{P}^1,\mk_x)$ is easily computed to be zero  using \eqref{timer} (note that $\mv$ is locally free), and hence we get a determinant of cohomology line (i.e., a one dimensional complex vector space) $D(\mk_x)=\det H^0(\Bbb{P}^1,\mk_x)^*\tensor  \det H^1(\Bbb{P}^1,\mk_x)$
together with a canonical section (i.e., an element in this line). This construction (as $x$ varies) gives
\begin{itemize}
 \item[-] A locally free sheaf $\mk$ on $\widehat{\Omega}^0_S(d,r,\mn,n,\vec{I})\times\Bbb{P}^1$.
\item[-] A line bundle $\Theta$ on $\widehat{\Omega}^0_S(d,r,\mn,n,\vec{I})$, together with a section $\theta$ on $\widehat{\Omega}^0_S(d,r,\mn,n,\vec{I})$. The fiber of this bundle over $x\in \widehat{\Omega}^0_S(d,r,\mn,n,\vec{I})$ is the determinant of cohomology $\mathcal{D}(\mk_x)$ of the coherent sheaf $\mk_x$ (as above) on $\Bbb{P}^1$.
\end{itemize}
\begin{proposition}\label{torro}
The Cartier divisor on $\widehat{\Omega}^0_S(d,r,\mn,n,\vec{I})$  corresponding to $\theta$ equals $\mathcal{R}$, the ramification divisor of the representable morphism of smooth Artin stacks $\pi:\widehat{\Omega}^0_S(d,r,\mn,n,\vec{I})\to \Par_{n,\mn,S}$.
\end{proposition}
The above equality of divisors at the level of sets is both easy (see the first paragraph below) and sufficient for our main results (the proof of Proposition \ref{bee}(c)). 
 \begin{proof}
 Note that the ramification divisor $\mathcal{R}$ is supported on the points where the fiber of $\pi$ is not smooth. The tangent space of a fiber of $\pi$ at a point $x$ is the same as $H^0(\Bbb{P}^1,\mathcal{K}_x)$, which is non zero only at points at which $\theta$ vanishes. Therefore the desired equality is true at the level of sets. We now show that we have equality as Cartier divisors.

 We ignore the data of $\det(\mw)\leto{\gamma}\mathcal{N}$ for simplicity. There exists an $N>0$, and a large open substack $\op{Bun}'_{\mn}(n)$ of $\op{Bun}_{\mn}(n)$ (complement of codimension $\geq 2$) such that
\begin{itemize}
\item[-] For points $(\mw,\gamma)\in \op{Bun}'_{\mn}(n)$, the vanishing $H^1(\Bbb{P}^1,\mw(N))=0$ holds, and $H^0(\Bbb{P}^1,\mw(N))$ is generated by global sections.
\end{itemize}
 Suppose that the $H^0(\Bbb{P}^1,\mw(N))$  above is $M$-dimensional. Let $\mt=(\mathcal{O}^{\oplus M})\tensor \mathcal{O}$. There is a suitable component $Q'$ of the quot scheme of quotients of $\mt$ such that there is a smooth atlas $Q'\to \op{Bun}'_{\mn}(n)$ so that $\op{Bun}'_{\mn}(n)$ is the stack quotient
$[Q'/G]$ with $G= \operatorname{Aut}(\mt)=\operatorname{GL}(M)$.

The space $\Par_{n,\mn,S}$ is a bundle over $\op{Bun}'_{\mn}(n)$ whose fibers are of the form $\Fl(n)^s$. We can base change  $\widehat{\Omega}^0(d,r,\mn,n,\vec{I})\to \Par_{n,\mn,S}$ to $Q'$. The base changed objects are now smooth schemes with actions of $G$, so that the original objects are stack quotients by $G$ (but only over $\op{Bun}'_{\mn}(n)$). Let the resulting map of smooth schemes be denoted by $\pi':W\to P$.

Note that $W$ parametrizes points $(\mv,\mw,\me,\gamma)\in \widehat{\Omega}^0_S(d,r,\mn,n,\vec{I})$ together with a surjection $\mt\twoheadrightarrow  \mw(N)$, so that $\mw^*$ is a subbundle of $\mt^*(N)$. Therefore $W$ is a fiber-bundle over a hyperquot scheme. $P$ is the data of $(\mw,\me,\gamma)$  with a surjection $\mt\twoheadrightarrow \mw(N)$.

We claim that given a point $x\in W$, there exist sheaves $\mathcal{A}_x$ and $\mathcal{B}_x$ on $\Bbb{P}^1$, and a surjective map of sheaves $\mathcal{A}_x\to \mathcal{B}_x$  so that
\begin{enumerate}
\item $TW_x =H^0(\Bbb{P}^1, \mathcal{A}_x)$, $TP_{\pi'(x)}=H^0(\Bbb{P}^1,\mathcal{B}_x)$ and $H^1(\Bbb{P}^1,\mathcal{A}_x)=H^1(\Bbb{P}^1,\mathcal{B}_x)=0$.
\item The sheaves $\mathcal{A}_x$ and $\mathcal{B}_x$ fit together and give sheaves $\mathcal{A}$ and $\mathcal{B}$ on $W\times \Bbb{P}^1$ together with a surjective map $\mathcal{A}\to \mathcal{B}$.
\item The sheaf $\mathcal{K}$ pulled back to $W\times\Bbb{P}^1$ equals the kernel of  $\mathcal{A}\to \mathcal{B}$. The pushforward, in the derived category, of (this pullback of) $\mk$ to $\Omega$
    is represented by a two term complex of vector bundles with fibers $H^0(\Bbb{P}^1, \mathcal{A}_x)$ and $H^0(\Bbb{P}^1,\mathcal{B}_x)$.
\end{enumerate}
This will imply that $\mathcal{D}(\mk_x)$ (together with the section) is canonically identified with $\det H^0(\ma_x)^*\otimes \det H^0(\mb_x)$ (together with the canonical section), and hence with the ramification line bundle and its section, and hence complete the proof.

To prove the claim, we appeal to the description of tangent spaces to hyperquot scheme parametrizing hyperquotients of the trivial bundles (i.e., of flags of subsheaves of the trivial bundles), see e.g., \cite{CF} (particularly Theorem 1.2, and Proposition E in the appendix). Ignoring the flag structures, both $W$ and $P$ are hyperquot schemes and the description of $\mk$ as the kernel of $\ma\to \mb$ follows from \cite[Prop. E]{CF}.

When flag structures are introduced, we need to replace $P$ by a product (locally) over $\op{Bun}'_{\mn}(n)$ of the form $\op{Bun}'_{\mn}(n)\times (\Fl(n))^s$ because the fibers at $p\in S$ of the universal bundles on $\op{Bun}'_{\mn}(n)\times \Bbb{P}^1$ are locally trivial. The replacement of $\ma_x$ will be a subsheaf $\ma'_x$ of $\ma_x\oplus T_e(\Fl(n))^s$ with the quotient a sum of skyscraper sheaf (on fibers), see Remark \ref{afternoon} below. We want $H^1$ of this sheaf to be zero to apply the above reasoning. This follows because we know that the dimension of $H^0(\ma'_x)$ is the expected one ($\widehat{\Omega}^0_S(d,r,\mn,n,\vec{I})$ is smooth of the expected dimension).
\end{proof}
\begin{remark}\label{afternoon}
Consider the universal Schubert variety $\Omega_I$ in $\Gr(r,n)\times\Fl(n)$, and a smooth point $x=(V,E_{\bull})$ in $\Omega_I$. Then the tangent space of $\Omega_I$ fits into an exact sequence
$$0\to T_x(\Omega_I) \to T(\Gr(r,n))_V\oplus T(\Fl(n))_{E_{\bull}}\to Q\to 0$$ where $Q$ is a vector space of dimension given by the codimension of $\Omega_I(E_{\bull})$ in $\Gr(r,n)$. The map $T(\Gr(r,n))_V\oplus T(\Fl(n))_{E_{\bull}}\to Q$ is surjective on each factor.
\end{remark}
There is a morphism  $\ParL\to {\Omega}^0(d,r,\mn,n,\vec{I})$. Let $\mathcal{R}_L\subseteq \ParL$ be the pullback of $\mathcal{R}$. From the description of $\mathcal{R}$ as the zeroes of a determinant given in \eqref{sky}, it follows that
\begin{lemma}
The Cartier divisor $\mathcal{R}$ on  $\widehat{\Omega}^0(d,r,\mn,n,\vec{I})$ equals $\pi^*\mathcal{R}_L$ where
$\pi: \widehat{\Omega}^0(d,r,\mn,n,\vec{I})\to \ParL$.
\end{lemma}

Note that \eqref{morph1} restricts to
 \begin{equation}
\xymatrix{
 \Omega^0-\mathcal{R} \ar[r]\ar[d] &  \widehat{\Omega}^0-\mathcal{R}\ar[dl]\\
\ParL-\mathcal{R}_L
}
\end{equation}
\subsection{Proof of Proposition \ref{bee}, (c)}\label{martel}
We need to show that a point $x\in\widehat{\Omega}_{S}^0-\widehat{\Omega}^0$, $H^0(\mk_x)\neq 0$. This is clear because $\shom(\mv,\mathcal{Q})$ has torsion supported outside of $S$, and therefore gives rise to torsion in $\mk_x$.
This proves (c) of Proposition \ref{bee}. Proposition \ref{bee} has the following important corollary:
\begin{corollary}\label{carole}
Suppose $\langle\sigma_{I^1}, \sigma_{I^2},\dots,\sigma_{I^s}\rangle_{d,D}=1$, and $1\not\in I^i$ for all $i$. Then $\widehat{\Omega}^0(d,r,\mn,n,\vec{I})-\mathcal{R}$ maps isomorphically to an open substack of $\Par_{n,\mn,S}$ whose complement has codimension $\geq 2$.
\end{corollary}
\begin{proof}
 The assertion follows from Propositions \ref{maya} and \ref{bee}, the surjectivity and birationality of $\Omega\to \Par_{n,\mn,S}$, and Zariski's main theorem.
\end{proof}

\section{Proofs of Theorems \ref{Adagio} and \ref{brahms}}
\subsection{Shift operations}
The spaces ${\Omega}^0(d,r,\mn,n,\vec{I})$ (but not $\widehat{\Omega}^0(d,r,\mn,n,\vec{I})$) together with their maps to $\ParL$  behave well under shift operations:

Let $\deg\mn= -D$ and $p=p_i\in S$. Parallel to Diagram \eqref{basicD} we have a diagram
 \begin{equation}\label{basicDnew}
\xymatrix{
 \Omega^0(d,r,\mn,n,\vec{I})\ar[d]\ar[r]^{\operatorname{Sh}_p}  & \Omega^0(d',r,\mn(p),n,\vec{K})\ar[d]\\
\ParL(-d,r,n-r,d-D)\ar[r]^{\operatorname{Sh}_p}  & \ParL(-d',r,n-r,d'-D+1)
}
\end{equation}
where the terms $d'$ and $\vec{K}$ are determined as in Diagram \eqref{basicD}:
 \begin{enumerate}
 \item $d'=d$ if $1\notin I^i$, and $d'=d-1$ if $1\in I^i$
 \item $K^k=I^k$ for $k\neq i$.
 \item $I^i =\{i^i_1-1,i^i_2-1,\dots, i^i_r-1\}$ if $1\notin I^i$, and $K^i =\{i^i_2-1,\dots, i^i_r-1,n\}$ if $1\in I^i$.
 \end{enumerate}
 The shift operation in the bottom row of \eqref{basicDnew} is shift on one of the factors and identity on the other. It is a shift on the rank $r$ vector bundle factor if $d'=d-1$ and on the  rank $n-r$ vector bundle if $d'=d$.
 \subsection{A scaling property}\label{FSP}
We have a self-contained proof of a generalization of a conjecture of Fulton (known before, see \cite[Remark 8.5]{BKq}). In Remark \ref{Bee132}, we explain how the version in Proposition \ref{qTheorems} (2) follows from the following:
\begin{proposition}\label{fulC} For all integers $n\geq 0$,
$\dim H^0(\ParL,\mathcal{O}(n\mathcal{R}_L))=1$.
\end{proposition}
\begin{proof}
 Note that $\mathcal{R}_L$ is the pullback of $\mathcal{R}$ under $\ParL\to {\Omega}^0(d,r,\mn,n,\vec{I})$. We can therefore reduce to the case  $1\not\in I^i$ for all $i$ using diagrams \eqref{basicDnew} and \eqref{basicD}.
Under this assumption, a function on $\ParL-\mathcal{R}_L$ pulls back to a function on $\widehat{\Omega}^0-\mathcal{R}$ which embeds in $\Par_{n,\mn,S}$ with complement of codimension $\geq 2$ by Corollary \ref{carole}. Since $h^0(\Par_{n,\mn,S},\mathcal{O})=1$, the proof is complete.
\end{proof}
\begin{proposition}\label{rigid}
$H^0(\Omega^0-\mathcal{R},\mathcal{O})=\Bbb{C}$.
\end{proposition}
\begin{proof}
Let $f\in H^0(\Omega^0-\mathcal{R},\mathcal{O})$. By Proposition \ref{fulC}, the restriction of the function $f$ to $\ParL$ is a constant $c$. Given a point $p\in \Omega^0$, there is a mapping $g:\Bbb{A}^1\to \Omega^0$ which
\begin{itemize}
\item[-] Sends $1$ to $p$ and $0$ lies in the image of $\ParL$ and equals the image of $p$ under the map $\Omega^0\to \ParL$.
\item[-] The isomorphism class of $g(t)$ remains constant for $t\neq 0$.
\end{itemize}
The existence of such a mapping can be explained more easily in the general group situation. Let $P$ be the standard parabolic for the Grassmannian $\Gr(r,n)$. Let $L$ be the corresponding Levi subgroup. A point of $\Omega^0$ corresponds to a principal $P$-bundle with some extra structure at the fibers over points of $S$ (see e.g., Lemmas 3.3 and 3.4 in \cite{BKiers}). We now use the Levification operation in  \cite[Section 3.8]{BKq} to produce the family $g$.
\end{proof}
\begin{remark}\label{Bee132}
Assume $d=0$ in Proposition \ref{fulC} (but $D$ is arbitrary). By a basic computation (see Proposition \ref{basicC} below), $\dim H^0(\ParL,\mathcal{O}(\mathcal{R}_L))$ is a tensor product
$H^0({\Par}_{r,\mathcal{O}(-d),S},\ma)\otimes H^0({\Par}_{n-r,\mathcal{O}(D-d),S},\mb)$ where the line bundles $\ma$ and $\mb$ are in strange duality (Definition \ref{Shosta}), and hence the dimensions of their spaces of sections are the same. The formula for $\ma$ in Proposition \ref{basicC} shows that it includes all line bundles of grade zero on ${\Par}_{r,\mathcal{O},S}$. This implies Proposition \ref{qTheorems} (2).
\end{remark}

\subsection{Proof of Theorem  \ref{Adagio} and Theorem \ref{brahms}}\label{rumble}
Most of the steps are similar to the proof of \cite[Theorem 1.9]{BHermit}. The proof will use shift operations and we will prove  a generalization to ${\Par}_{n,\mn,S}$. We start with a $\langle\sigma_{I^1}, \sigma_{I^2},\dots,\sigma_{I^s}\rangle_{d,D}=1$, where $D=-\deg(\mn)$.  The inequality \eqref{wallnew} is replaced by \eqref{wallneww}, so that $\mf$ is now given inside  $\Pic_{\Bbb{Q}}({\Par}_{n,\mn,S})$ by equality in \eqref{wallneww}.

We first deal with the case that $a\neq 1$. We will show that $\Omega^0(d,r,\mathcal{N},n,\vec{J})\subset \widehat{\Omega}^0(d,r,\mathcal{N},n,\vec{I})$ is not contained in $\mathcal{R}$, the ramification locus of $\widehat{\Omega}^0(d,r,\mathcal{N},n,\vec{I})\to  {\Par}_{n,\mn,S}$.

Choose a point $x=(\mv,\mw,\me,\gamma)\in\Omega^0(d,r,\mathcal{N},n,\vec{I})$ such that $\widehat{\Omega}^0(d,r,\mathcal{N},n,\vec{I})\to  {\Par}_{n,\mn,S}$ is etale at $x$, and then modify an appropriate element in the flag $E^{p_j}_{\bull}$, so that the point lies now in  $\Omega^0(d,r,\mathcal{N},n,\vec{J})$. The modification is the following: replace  $E^{p_j}_{a-1}$ by $E^{p_j}_{a-2}+\mv_{p_j}\cap E^{p_j}_{a}$. This way we get a new $\me'\in \Fl_S(\mw)$, and $y=(\mv,\mw,\me',\gamma)\in\Omega^0(d,r,\mathcal{N},n,\vec{J})$. The tangent space ``situation'' at $y\in \widehat{\Omega}^0(d,r,\mathcal{N},n,\vec{I})$ is the same as at $x$, since the element $E^{p_j}_{a-1}$ does not appear in the rank conditions defining $\widehat{\Omega}^0(d,r,\mathcal{N},n,\vec{I})$. Therefore $y\not\in\mathcal{R}$.

The birational morphism $\widehat{\Omega}^0(d,r,\mathcal{N},n,\vec{I})\to  {\Par}_{n,\mn,S}$ is therefore  etale at $y\in \widehat{\Omega}^0(d,r,\mathcal{N},n,\vec{I})$. Since $D(a,j)$ is the  image of the irreducible  $\Omega(d,r,\mathcal{N},n,\vec{J})$, it is irreducible and it coincides with the cycle theoretic image $\widehat{\Omega}^0(d,r,\mathcal{N},n,\vec{I})\to  {\Par}_{n,\mn,S}$. This proves part (1) of Theorem \ref{Adagio}. Proposition \ref{enumerative} now implies Theorem \ref{brahms}.

 Note that $D(a,j)$ lies in the complement of the open subset $\Omega^0(d,r,\mathcal{N},n,\vec{I})-\mathcal{R}$ of ${\Par}_{n,\mn,S}$, since otherwise the image $y\in \Omega^0(d,r,\mathcal{N},n,\vec{J})$ in ${\Par}_{n,\mn,S}$ will have a disconnected inverse image in $\widehat{\Omega}^0(d,r,\mathcal{N},n,\vec{I})$. Part (2) of Theorem \ref{Adagio} now follows from Proposition \ref{rigid}. Part (3) of Theorem \ref{Adagio} follows from parts (1) and (2) using a suitable generalization of \cite[Lemma 2.1]{BHermit}. In fact, part(3) is already covered by Lemma \ref{elgar}.

Finally for part (4) we want that the line bundle $\mathcal{O}(D(a,j))$ on ${\Par}_{n,\mn,S}$ lies on the face $\mathcal{F}$. The numerical equality in \eqref{wallneww} is then just the statement that the center of $L$ acts trivially on $\mathcal{O}(D(a,j))$ restricted to $\Par_L$, which follows from the fact that the restriction is effective: Using arguments of Ressayre \cite{R1} this is equivalent to the line bundle restricted to $\ParL$ being effective, which is certainly the case because $D(a,j)$ does not contain the image of $\ParL$. Note  $\ParL-\mathcal{R}_L$ does not meet $D(a,j)$ since $\ParL-\mathcal{R}_L\subseteq \Omega^0(d,r,\mathcal{N},n,\vec{I})-\mathcal{R}$, and $D(a,j)$ does not intersect the last set.  This proves (4).

Finally we can remove the assumption $a\neq 1$. Assume $a=1\in I^j$. Consider the shifted data
$(d',r,\mathcal{N}(p_j),n,\vec{K})$ associated to $(d,r,\mn,n,\vec{I})$, with $d'=d-1$, with the shift applied at $p_j$, so that $K^j=\{i^j_2-1,\dots, j^j_n -1,n\}$. We refer to Theorem \ref{Adagio} (also see \eqref{basicD}) with the $\vec{I}$ replaced by $\vec{K}$. We now consider $\tilde{D}(n,j)$ for the tuple $(d',r,\mathcal{N}(p_j),n,\vec{K})$ and the above reasoning applies to it.  Using the inverse of the shift operation at $p_j$, and properties of the shift operation, we obtain the desired statement about $D(1,j)$.\subsection{Proof of Corollary \ref{existence}}
$\mathcal{B}(\vec{\lambda},\ell)$ is a F-line bundle by Theorem \ref{brahms}. The corollary now follows from Theorem \ref{christmas} and Corollary \ref{katzkatz}.

\section{The face $\mf$ as a product}
We want to show the product structure \eqref{productstructure} of the face $\mf$. Recall the setting:
we have a tuple $(d,r,I^1,\dots,I^s)$ such that
$\langle\sigma_{I^1}, \sigma_{I^2},\dots,\sigma_{I^s}\rangle_{d}=1$. This defines a face $\mf$  of the monoid
$\Pic^+_{\Bbb{Q}}(\Par_{n,\mathcal{O},S})$ given by equality in inequality
\eqref{wallnew}. In Theorem \ref{Adagio} we have produced extremal ray generators $[D(a,j)]$ of $\Pic^+_{\Bbb{Q}}(\Par_{n,\mathcal{O},S})$ which lie on $\mf$.

\begin{defi}\label{Br2}
\hspace{2em}
\begin{enumerate}
\item Let $\Pic'=\Pic'(\mf)\subseteq\Pic(\Par_{n,\mathcal{O},S})$ be the subgroup formed by  $\mathcal{B}(\vec{\lambda},\ell)$ such that for any $j$, both of the following two conditions hold : (a) $\lambda^j_{a-1}=\lambda^{j}_{a}$ whenever $a\in I^j$, $a>1$ and $a-1\not \in I^j$ and (b) if $n\not\in I^j$ and $1\in I^j$ then $\lambda^j_1=\lambda^{j}_{n}+\ell$.
\item  Let $\Pic'^{+,\deg=0}(\mf)=\mf\cap \Pic'(\mf)$.
\end{enumerate}
\end{defi}

The product structure \eqref{productstructure} follows from the following:
\begin{proposition}\label{Br4} Let $D_1,\dots, D_q$ be the list of elements $[D(a,j)]$ produced by Theorem \ref{Adagio}. There is an isomorphism
$\bigoplus_{i=1}^{q}\Bbb{Q}D_i \oplus  \Pic' (\mf)_{\Bbb{Q}}\leto{\sim} \Pic^{\deg=0}_{\Bbb{Q}}(\Par_{n,\mathcal{O},S})$, where $\Pic^{\deg=0}_{\Bbb{Q}}(\Par_{n,\mathcal{O},S})$ is the subset of $\Pic(\Par_{n,\mathcal{O},S})$ formed by line bundles for which equality holds in
\eqref{wallnew}.
Setting $\mf^{(2)}= \Pic'^{+,\deg=0}_{\Bbb{Q}}(\mf) $, we obtain a bijection
$$\prod_{i=1}^{q}\Bbb{Q}_{\geq 0}D_i \times  \mf^{(2)} \leto{\sim} \mf.$$
\end{proposition}

\begin{lemma}\label{oneoff}
Each of the divisors $D(a,j)$ in Theorem \ref{Adagio}, fails exactly one of the equalities defining $\Pic'$:
Write $\mathcal{O}(D(a,j))=\mathcal{B}(\vec{\lambda},\ell)$.
\begin{enumerate}
\item If $a>1$, then $\lambda^j_{a-1}=\lambda^j_{a}+1$.
\item If $a=1$, then $\lambda^j_n =\lambda^j_1-\ell +1$.
\end{enumerate}
\end{lemma}
\begin{proof}
These statements follow from Theorem \ref{brahms} and  equation \eqref{ABL4}.
\end{proof}

\subsection{Proof of Proposition \ref{Br4}}

The proof is a direct generalization of \cite[Theorem 1.15]{BHermit}. The main point is the following: The first isomorphism follows immediately from Lemma \ref{oneoff}. For the second bijection we proceed as follows. Suppose we have  a point $\mathcal{B}(\vec{\lambda},\ell)\in\mf$ not in $\Pic'$. Assume first that it fails an inequality $\lambda^j_{a-1}=\lambda^{j}_{a}$ with $a\in I^j$, $a>1$ and $a-1\not\in I^i$. We will then show that any global section of  $\mathcal{B}(\vec{\lambda},\ell)$ vanishes at any point of the divisor $D(a,j)$, so that $\mathcal{B}(\vec{\lambda},\ell)(-D(a,j))$ is still effective. It still lies on $\mf$ because $D(a,j)$ is on $\mf$ by Theorem \ref{Adagio}.

To show the vanishing, let $(\mv,\mw,\me,\gamma)$ be a general point of $D(a,j)$. The semistability inequality \eqref{wallnew} holds as an equality for $\mathcal{B}(\vec{\lambda},\ell)$. The semistability inequality corresponding to $\mv$  for $\mathcal{B}(\vec{\lambda},\ell)$  replaces $\lambda^j_a$ in the above equality by $\lambda^j_{a-1}$ and hence fails because of Lemma \ref{oneoff}. Therefore any global section of $\mathcal{B}(\vec{\lambda},\ell)$ vanishes on $D(a,j)$ as desired.

If $\mathcal{B}(\vec{\lambda},\ell)$  fails an inequality  $\lambda^j_1=\lambda^{j}_{n}+\ell$ with $n\not\in I^j$ and $1\in I^j$ then we proceed similarly: In this case $\mv$ is only a subsheaf of $\mw$, but we can consider the subbundle given by $\mv$ (``the saturation''). The semistability inequality corresponding to the saturation is violated at a generic point of $D(1,j)$: In this case the left hand side of $\eqref{wallnew}$ decreases by $\ell-1$ and the right hand side by $\ell$ since $d$ drops by one (by Lemma \ref{oneoff}).

By repeatedly carrying out these subtractions we can write any element of  $\mf$ as a sum corresponding to the second bijection.
\begin{remark}
As in \cite{BHermit,BKiers,Kiers}, we may distinguish two types of rays of $\mathcal{F}$. A ray
$\Bbb{Q}_{\geq 0}(\vec{\lambda},\ell)$ is a type I ray of $\mathcal{F}$ if $\mathcal{B}(\vec{\lambda},\ell)$ is not in $\Pic'(\mf)$, and type II otherwise. The extremal rays  generated by $D(a,j)$  are  type I on $\mf$. An extremal ray of  $\Pic^+_{\Bbb{Q}}(\Par_{n,\mathcal{O},S})$ may lie on two different regular facets. It is possible that it is a type I ray of one facet and a type II of the other. All rays that are type I on some regular facet are generated by F-line bundles.
\end{remark}

\section{The induction morphism}\label{noncan}
Our aim now is to show that  $\Pic'=\Pic'(\mf)$ and $\Pic'^{+,\deg=0}(\mf)$ (Definition  \ref{Br2} and Proposition \ref{Br4}) are both controlled by the Levi-subgroup by an induction procedure. Since we need shifting procedures in the proof, we operate in a more general setting:
Assume that the Gromov-Witten number $\langle\sigma_{I^1}, \sigma_{I^2},\dots,\sigma_{I^s}\rangle_{d,D}=1$, where $D=-\deg(\mn)$. This implies that $\Omega(d,r,\mn,n,\vec{I})\to \Par_{n,\mn,S}$ is birational, and \eqref{timer1} holds.

\begin{defi}
Define $\Pic'=\Pic'(\mf)\subseteq\Pic(\Par_{n,\mn,S})$ as follows (generalizing  Definition  \ref{Br2} to the case of arbitrary $\mathcal{N}$): It is subgroup formed by  $\mathcal{B}(\vec{\lambda},\ell)$ such that for any $j$, both of the following two conditions hold : (a) $\lambda^j_{a-1}=\lambda^{j}_{a}$ whenever $a\in I^j$, $a>1$ and $a-1\not \in I^j$ and (b) If $n\not\in I^j$ and $1\in I^j$ then $\lambda^j_1=\lambda^{j}_{n}+\ell$.
\end{defi}
It is easy to see that $\Pic'$ behaves well under shift operations of the tuple $(d,r,\mn,n,\vec{I})$: In \eqref{basicD} with $\vec{J}=I$, the operation $\op{Sh}_p$ pulls back $\Pic'$ for the tuple $(d',r,\mn(p),n,\vec{K})$ to $\Pic'$ for $(d,r,\mn,n,\vec{I})$ (use Proposition  \ref{easily}).

Let $\Pic^{\deg=0}(\Par_{n,\mn,S})$ be the subgroup given by line bundles $\mathcal{B}(\vec{\lambda},\ell)$ such that  equality holds in the following inequality (this is the same inequality as \eqref{wallnew} when $D=0$)
\begin{equation}\label{wallneww}
\frac{1}{r}\bigl(-d+\sum_{j=1}^s\sum_{k\in I^j} \frac{\lambda^j_k}{\ell}\bigr)\ \leq\  \frac{1}{n}\bigl(-D +\sum_{j=1}^s \frac{|\lambda^j|}{\ell}\bigr)
\end{equation}
Finally $\mf\subseteq \Pic^{+,\deg=0}_{\Bbb{Q}}(\Par_{n,\mn,S})$ is the effective $\Bbb{Q}$ subcone.
\begin{defi}\label{Pin}
\begin{enumerate}
\item Any line bundle $\ml$ on $\ParL$ (or on $\ParL-\mathcal{R}_L$ ) has an index $m(\mathcal{L})\in \Bbb{Z}$: $t\in\Bbb{C}^*$ acts on points $(\mv,\mq,\mf,\mg,\gamma)\in\ParL$, by multiplication
by $t^{n-r}$ on $\mv$ and $t^{-r}$ on $\mq$ and leaving $\gamma$ unchanged. This action lifts to the line bundle $\ml$, as multiplication by $t^m$, for some $m\in\Bbb{Z}$. We set $m(\mathcal{L})=m$.
\item Let $\Pic^{\deg=0}_{\Bbb{Q}}(\ParL)\subset \Pic_{\Bbb{Q}}(\ParL)$ (similarly $\Pic^{\deg=0}_{\Bbb{Q}}(\ParL-\mathcal{R}_L)\subset \Pic_{\Bbb{Q}}(\ParL-\mathcal{R}_L)$) be  the subspace of line bundles of index $0$.
\item Let $\Pic^{+,\deg=0}_{\Bbb{Q}}(\ParL)\subset \Pic_{\Bbb{Q}}(\ParL)$ (similarly $\Pic^{+deg=0}_{\Bbb{Q}}(\ParL-\mathcal{R}_L)\subset \Pic_{\Bbb{Q}}(\ParL-\mathcal{R}_L))$be  the subspace of effective line bundles (they automatically have  index $0$).
\end{enumerate}
\end{defi}
In this section, we set up an induction map
\begin{equation}\label{inducto}
\op{Ind}:\Pic(\ParL-\mathcal{R}_L)\to \Pic(\Par_{n,\mn,S}).
\end{equation}
We then show the following properties:
\begin{enumerate}
\item [(I1)] The map $\Pic^{+,\deg=0}_{\Bbb{Q}}(\ParL)\to \Pic^{+,\deg=0}_{\Bbb{Q}}(\ParL-\mathcal{R}_L)$ is surjective.
\item[(I2)] The map $\Pic^{+,\deg=0}_{\Bbb{Q}}(\ParL-\mathcal{R}_L)\to \Pic(\Par_{n,\mn,S})_{\Bbb{Q}}$
is injective with image exactly $\Pic'_{\Bbb{Q}}$.
\item[(I3)]  $\Pic^{+,\deg=0}_{\Bbb{Q}}(\ParL)$ is identified with  $\Pic^+_{\Bbb{Q}}(\Par_{r,\mathcal{O}(-d),S})\times \Pic^+_{\Bbb{Q}}(\Par_{n-r,\mathcal{O}(d-D),S})$,
\end{enumerate}
Putting these together one gets the desired surjection \eqref{inductiones}. Explicit formulas for the induction morphism will be given in Section \ref{explicitF}.

\subsection{Construction of the induction morphism}
If $1\not\in I^i$ for all $i$, we have a morphism  (see \eqref{morph1}) $\widehat{\Omega}^0(d,r,\mn,n,\vec{I})\to \ParL$ which fits into a diagram:
\begin{equation}\label{Dbasic1.5}
\xymatrix{
 & \widehat{\Omega}^0(d,r,\mn,n,\vec{I})\ar[dl]^{\pi}\ar[dr]\\
\Par_{n,\mn,S}   &  & \ParL\ar@/_/[ul]_i\ar[ll]^{i'}
}
\end{equation}
Suppose an $s$-fold application of shifting (at various points of $S$, perhaps repeated) leads to $1\not\in I^i$ for all $i$. Note that there is a minimal choice of these shift operations, and we will use this choice (note that all shifts will turn out to produce the same induction map).

Let the shifted data be $(d',r,\mn',n,\vec{K})$. Set $D'=-\deg(\mn')$ and let $\op{Sh}$ denote the composites $\ParL(-d,r,n-r,d-D)\to \ParL(-d',r,n-r,d'-D')$, as well as $\Par_{n,\mn,S}\to \Par_{n,\mn',S}$(see the diagram \eqref{basicDnew}). To define the induction map, we take line bundles on  $\ParL(-d,r,n-r,d-D)$, obtain line bundles on $\ParL(-d',r,n-r,d'-D')$ by the inverse of shifting procedure $\Sh$. These line bundles then give line bundles on $\widehat{\Omega}^0(d',r,\mn',n,\vec{K})$, hence on $\widehat{\Omega}^0(d',r,\mn',n,\vec{K})-\mathcal{R}'$, where $\mathcal{R}'$ is the ramification locus of $\widehat{\Omega}^0(d',r,\mn',n,\vec{K})\to \Par_{n,\mn',S}$. By Corollary \ref{carole}, $\widehat{\Omega}^0(d',r,\mn',n,\vec{K})-\mathcal{R}'$ is an open substack of $\Par_{n,\mn',S}$ whose complement is of codimension $\geq 2$. Therefore we obtain a line bundle on
$\Par_{n,\mn',S}$  and we pull these line bundles to $\Par_{n,\mn,S}$ via $\Sh$.
\subsection{Proof of property (I1)}

\begin{lemma}\label{prius2}
The restriction mapping $\Pic(\ParL)\to \Pic(\ParL -\mathcal{R}_L)$ has a section and is hence surjective.
\end{lemma}
\begin{proof}
We may assume that $1\not\in I^i$ for all $i$ without loss of generality. Let $\mathcal{M}\in\Pic(\ParL -\mathcal{R}_L)$. To define the section we may just restrict
$\op{Ind}(\mathcal{M})$  via $\pi\circ i:\ParL \to \Par_{n,\mn,S}$.
\end{proof}
Property (I1) is the following:
\begin{proposition}\label{M11}
The restriction map $\Pic^{+,\deg=0}(\ParL)\to \Pic^{+,\deg=0}(\ParL -\mathcal{R}_L)$ is a surjection.
\end{proposition}
\begin{proof}
Let $\mathcal{M}$ be an effective line bundle on $\ParL -\mathcal{R}_L$ with section $s$. Then $\op{Ind}(\mathcal{M})$ is defined as a pullback of $\mathcal{M}$ and   therefore has a non-zero section $\tilde{s}$ defined as a pullback of $s$. The restriction of a $\tilde{s}$  to $\ParL$ coincides with $s$ on $\ParL -\mathcal{R}_L$ and is hence non-zero. This proves the proposition.
\end{proof}

\subsection{Proof of property (I2)}

\begin{defi}  When $1\not\in I^i$ for all $i$, we consider $\Par'_{n,\mn,S}$ the parameter space of rank $n$ parabolic bundles with $\gamma:\det \mw\leto{\sim}\mn$, where for each $i$,  $\me^{p_i}$ now parametrizes {\em partial flags} where only $E^{p_i}_a$ with $a\in T(I^i)$ are defined (see Definition  \ref{ToI}).
\end{defi}
Using  the method of \cite{LS}, we obtain
\begin{lemma}
 Assume $1\not\in I^i$ for all $i$. There is a natural mapping $\Par_{n,\mn,S}\to \Par'_{n,\mn,S}$.
The subgroup of $\Pic(\Par_{n,\mn,S})$ formed by $\Pic(\Par'_{n,\mn,S})$ coincides with the subgroup
$\Pic'$ defined in Proposition \ref{grill}.
\end{lemma}

\begin{proposition}\label{grill}
The image of
$$\op{Ind}:\Pic(\ParL -\mathcal{R}_L)\to \Pic(\Par_{n,\mathcal{N},S})$$
 is contained in  $\Pic'$.
\end{proposition}
\begin{proof}(of Proposition \ref{grill})
Using the shift operations, we may assume that $1\not\in I^i$ for all $i$.
Some parts of the flag structures on $\widehat{\Omega}^0(d,r,\mn,n,\vec{I})$ are ``unnecessary" for its definition. In fact there is a natural  representable $\widehat{\Omega'}^0(d,r,\mn,n,\vec{I})\to \Par'_{n,\mn,S}$ such that the base change to $\Par_{n,\mn,S}\to\Par'_{n,\mn,S}$ gives $\widehat{\Omega}^0(d,r,\mn,n,\vec{I})\to \Par_{n,\mn,S}$. Note also that by \cite[Lemma 7.10]{BHermit}, we have a natural map $\Par'_{n,\mn,S}\to \ParL$ which factorizes $\Par_{n,\mn,S}\to \ParL$.

Therefore assuming $1\not\in I^i$ for all $i$, the induction map \eqref{inducto} factors through  $\Pic(\Par'_{n,\mn,S})$ as desired.
\end{proof}

 \subsubsection{The mapping  \eqref{inducto} is an injection with image $\Pic'$}
\begin{proposition}\label{outside}
The induction morphism \eqref{inducto} is a bijection of $\Pic(\ParL -\mathcal{R}_L)$ with $\Pic'$.
\end{proposition}

Again to prove this proposition we may assume  $1\not\in I^i$ for all $i$. It follows immediately from the two lemmas below.

\begin{lemma}\label{later}
Assume $1\not\in I^i$ for all $i$.
The composite
$$\Pic(\Par'_{n,\mn,S})\to \Pic(\Par_{n,\mn,S})\to  \Pic({\Omega}^0(d,r,\mn,n,\vec{I})-\mathcal{R})$$ is an isomorphism.
\end{lemma}
\begin{proof}
We will construct an inverse of the mapping $\Pic(\Par'_{n,\mn,S})\to  \Pic({\Omega}^0(d,r,\mn,n,\vec{I})-\mathcal{R})$. Take a line bundle on ${\Omega}^0(d,r,\mn,n,\vec{I})-\mathcal{R}$ and extend it to $\Par_{n,\mn,S}$. There are many choices here, any two choices differ by a linear combination of the $D(a,j)$. There is a unique such extension which is in $\Pic'(\Par'_{n,\mn,S})$ by Lemma \ref{oneoff} (that the failure is by $1$  there is crucial). This gives the desired inverse. The two composites are identity by easy inspection.
\end{proof}

\begin{lemma}\label{muchlater}
Assume $1\not\in I^i$ for all $i$,
The pullback map
$\Pic(\Par_L-\mathcal{R}_L)\to  \Pic({\Omega}^0(d,r,\mn,n,\vec{I})-\mathcal{R})$is an isomorphism.
\end{lemma}
\begin{proof}
The section $i$ in \eqref{Dbasic} gives a section $\ParL-\mathcal{R}_L\to {\Omega}^0(d,r,\mn,n,\vec{I})-\mathcal{R}$ of the pullback map (see the Proof of \eqref{fulC}). Therefore we need to show only that a line bundle $\ma$ on ${\Omega}^0(d,r,\mn,n,\vec{I})-\mathcal{R}$ which restricts via $i$ to the trivial bundle
on $\ParL-\mathcal{R}_L$ is itself trivial. We wish to propagate the non-zero section of $\ma$ on  $\ParL-\mathcal{R}_L$ to a non-zero section on all of ${\Omega}^0(d,r,\mn,n,\vec{I})-\mathcal{R}$.
By  the Levification map (\cite[Section 3.8]{BKq}) we construct a map:
\begin{itemize}
\item[-] A map $\phi:\Bbb{A}^1\times {\Omega}^0(d,r,\mn,n,\vec{I})\to \Omega^0(d,r,\mn,n,\vec{I})$ such that $\phi_0$ equals $\Omega^0(d,r,\mn,n,\vec{I})\to \ParL$ followed by the section (where $\phi_t$ is the restriction of $\phi$ to $t\times \Omega^0(d,r,\mn,n,\vec{I})$).
\end{itemize}

We have a section of $\mathcal{A}$ on $\ParL$. Let $x\in \Omega^0(d,r,\mn,n,\vec{I})$. We obtain a line bundle on $\Bbb{A}^1$ by pulling back  $\mathcal{A}$ along $\Bbb{A}^1\to \Omega^0(d,r,\mn,n,\vec{I})$, $x\mapsto \phi(t,x)$. This line bundle is $\Bbb{G}_m$ equivariant, with trivial Mumford number at $0$. It is therefore canonically trivial (since we have a given section at $t=0$). We therefore obtain a canonical element in $\mathcal{T}_x$ for any $x$. This shows $\mathcal{T}$ is trivial (see \cite[Lemma 8.6]{BHermit})).

If  $\mathcal{T}\in \Pic'(\mf)^{\deg=0}$, then it is easy to see that $\mathcal{M}$ is in $\Pic^{\deg=0}(\ParL -\mathcal{R}_L)$, since the degree on both sides is the weight of the action of the center of the Levi subgroup in both cases.
\end{proof}
\subsubsection{Proof of property (I2)}
Property (I2) follows from Proposition \ref{outside} and  the following:
\begin{proposition}\label{uselater}
Suppose $\mathcal{M}\in \Pic^{\deg=0}(\ParL-\mathcal{R}_L)$. Then
\begin{enumerate}
\item[(a)] $\op{Ind}(\mathcal{M})\in \Pic^{\deg=0}(\Par_{n,\mn,S})$.
\item[(b)] $H^0(\ParL-\mathcal{R}_L,\mathcal{M})\to H^0(\Par_{n,\mn,S},\op{Ind}(\mathcal{M}))$ is an isomorphism
\item[(c)] $H^0(\ParL,\op{Ind}(\mathcal{M})|_{\ParL})\to H^0(\ParL -\mathcal{R}_L,\mathcal{M})$ is also an isomorphism.
\end{enumerate}
\end{proposition}
\begin{proof}
For part (a), let $\op{Ind}(\mathcal{M})=\mathcal{B}(\vec{\lambda},\ell)$. Now $\op{Ind}(\mathcal{M})$ restricted to $\ParL-\mathcal{R}_L$ lies in  $\Pic^{\deg=0}(\ParL-\mathcal{R}_L)$, i.e., has index zero.
The index is easily computed to be  up to a multiple equal to the difference between the two sides of \eqref{wallneww} (Note that multiplication by $t$ on a vector bundle acts on the determinant of cohomology by $t^{\chi}$, with $\chi$ the Euler characteristic of the vector bundle).  Part (a)  follows.

Part (b) is a generalization of Proposition  \ref{fulC} (for $\mathcal{M}=\mathcal{O}$), and the proof is similar. The map in (c) is obviously injective. Since we have  maps (reduce first to the case $1\not\in I^i$ for all $i$)
$$ H^0(\ParL -\mathcal{R}_L,\mathcal{M})\to H^0(\Par_{n,\mn,S},\op{Ind}(\mathcal{M}))\to H^0(\ParL,\op{Ind}(\mathcal{M})|_{\ParL})$$
the surjectivity follows.
\end{proof}
\subsection{Proof of property (I3)}
We have a natural map
\begin{equation}\label{temper2}
{\Par}_{r,\mathcal{O}(-d),S}\times {\Par}_{n-r,\mathcal{O}(d-D),S} \to  \ParL.
\end{equation}

\begin{lemma}\label{BachMass}
\begin{enumerate}
\item  $\Pic^{+,\deg=0}_{\Bbb{Q}}(\ParL)\leto{\sim}\Pic^+_{\Bbb{Q}}(\Par_{r,\mathcal{O}(-d),S})\times \Pic^+_{\Bbb{Q}}(\Par_{n-r,\mathcal{O}(d),S})$.
\item $\Pic^{+}(\ParL)\leto{\sim}\Pic^+(\Par_{r,\mathcal{O}(-d),S})\times \Pic^+(\Par_{n-r,\mathcal{O}(d),S})$.
\item Under the correspondence in (2), a line bundle $\ml$ on $\ParL$ corresponds to tensor product $\ml_1\boxtimes\ml_2$ with one of $\ml_1$ a F-line bundle and the other trivial, if and only if,
$\ml=\mathcal{O}(E)$ with $E$ non-empty, reduced and irreducible.

\end{enumerate}
\end{lemma}
This result follows from Lemma \ref{mass2} below applied to \eqref{temper2}.
It is necessary to take $\Bbb{Q}$-coefficients above for the reverse map in (1).

\subsection{Induction and extremal rays coming from F-line bundles}
We want to determine the set of F-line bundles in $\mf^{(2)}$. Note that using shift operations (Proposition \ref{easily}), we can put the set of F-line bundles on $\Par_{r,\mathcal{N},S}$ in bijection with the set of F-line bundles
on   $\Par_{r,\mathcal{O},S}$ (which are in bijection with F-vertices of $P_r(s)$ by Lemma \ref{corresp}).
\begin{proposition}\label{condensed}
There is a bijection between the following two sets:
\begin{enumerate}
\item The set of  $\ml\in \Pic'\subseteq\Pic(\Par_{n,\mathcal{O},S})$ (see Definition  \ref{Br2}) such that $\ml$ is a F-line bundle on  $\Par_{n,\mathcal{O},S}$.
\item The set of $\mathcal{M}\in  \Pic^{+}(\ParL)$ such that the following two conditions hold
\begin{enumerate}
\item[(a)]$\mathcal{M}$ corresponds to a F-line bundle on one of the factors of (and $0$ on the other)
$$\Pic^+(\Par_{r,\mathcal{O}(-d),S})\times \Pic^+(\Par_{n-r,\mathcal{O}(d),S})\subseteq \Pic^+_{\Bbb{Q}}(\Par_{r,\mathcal{O}(-d),S})\times \Pic^+_{\Bbb{Q}}(\Par_{n-r,\mathcal{O}(d),S}) $$ under the  bijection in Lemma \ref{BachMass}.
\item[(b)] $Ind(\mathcal{M})\neq 0$ and $H^0(\ParL,\op{Ind}(\mathcal{M})|_{\ParL})$ (which breaks up as a tensor product)  is one-dimensional (see Remark \ref{putR} for equivalent conditions).
\end{enumerate}
\end{enumerate}
The bijection takes $\mathcal{M}$ in (2) to its induction $\ml=\op{Ind}(\mathcal{M})$.
\end{proposition}

\begin{proof}
We start by showing that if $\mathcal{M}=\mathcal{O}(D)$ is as in (2), then $\ml=\op{Ind}(\mathcal{M})$ is an F-line bundle.
By Proposition \ref{uselater}, and (b), $H^0(\Par_{n,\mathcal{O},S},\ml)$ is one dimensional. Let $\ml=
\mathcal{O}(E)$ and we want to show $E$ is irreducible and reduced. Suppose $E=E_1+E_2$, let $\ml_1=\mathcal{O}(E_1)$ and $\ml_2=\mathcal{O}(E_2)$. There exist $\mathcal{M}_1$ and $\mathcal{M}_2$ effective line bundles on $\Par_L$ which induce to $\ml_1$ and $\ml_2$. Let $\mathcal{M}_i=\mathcal{O}(D_i)$. By Proposition \ref{uselater}, $D_1+D_2$ and $D$ agree on $\Par_L-\mathcal{R}_L$. Since $D$ is irreducible, either $D_1$ or $D_2$ is supported on $\mathcal{R}_L$, so either $\ml_1$ or $\ml_2$ (being the inductions of $\mathcal{M}_1$ and $\mathcal{M}_2$) is trivial, as desired. If both $D_1$ and $D_2$ are both supported on  $\mathcal{R}_L$, then $D$ is also supported on $\mathcal{R}_L$, and hence $\op{Ind}(\mathcal{M})$ is trivial.

To go the other direction, start with a $\ml=\mathcal{O}(E)$ as in (1). By Proposition \ref{uselater}, $H^0(\Par_L-\mathcal{R}_L,\mathcal{L}|_{\ParL})$ is  one dimensional. Let ${D}\subset \ParL$ be the closure of the zeroes of this section and set $\mathcal{M}=\mathcal{O}(D)$.  We claim that $\mathcal{M}$ satisfies the conditions in (2). By construction $\op{Ind}(\mathcal{M})=\ml$, and hence $\mathcal{M}$ is non trivial. The condition in (b) holds by Proposition \ref{uselater}.

We now show that $\mathcal{M}$ is a F-line bundle. Since
$\mathcal{L}|_{\ParL}$ and $\mathcal{M}$ agree on $\Par_L-\mathcal{R}_L$, $H^0(\Par_L,\mathcal{M})$ is one dimensional.

If $D=D_1+D_2$, then we get a corresponding decomposition of $\ml=\ml_1\tensor\ml_2$ as a tensor product of effective bundles. We see that one of the factors $\ml_1$ or $\ml_2$ is trivial, and hence $D_1$ or $D_2$ is trivial on $\Par_L-\mathcal{R}_L$. Since $D$ is the closure of its intersection with $\Par_L-\mathcal{R}_L$, this makes one of $D_1$ and  $D_2$ trivial, and hence $\mathcal{M}$ is an F-line bundle.

The composite of going from (2) to (1) and back to (2) as described above is the identity. This finishes the proof of the proposition.

\end{proof}
\begin{remark}\label{putR}
The condition in (b) is equivalent to $Ind(\mathcal{M})\neq 0$ and $h^0(\ParL,\mathcal{M}(\mathcal{R}_L))=1$. This is because any section of $H^0(\ParL-\mathcal{R}_L,\mathcal{M})$ lifts to a section of
$H^0(\ParL,\mathcal{M}(m\mathcal{R}_L))$ for some $m>0$, and we know that $\mathcal{M}$ is effective and $h^0(\ParL,\mathcal{M}^m(m\mathcal{R}_L))=1$ by quantum Fulton.

\end{remark}
\section{Formulas for induction}\label{explicitF}
\begin{defi}
Let $\widetilde{\Par}_{n,\mn,S}$ be the moduli stack of pairs $(\mw,\me)$ such that $\mw$ is a vector bundle on $\Bbb{P}^1$ with determinant isomorphic to $\mn$ (but we do not fix an isomorphism), and $\me\in\Fl_S(\mw)$. There is a morphism ${\Par}_{n,\mn,S}\to\widetilde{\Par}_{n,\mn,S}$ surjective on objects but  $\widetilde{\Par}_{n,\mn,S}$ has more isomorphisms. It is easy to see that this morphism is surjective on Picard groups (use Proposition \ref{LaSo}).
\end{defi}
The following follows from Lemma \ref{mass2}.
\begin{lemma}
\begin{enumerate}
\item The Picard group of  $\Pic(\widetilde{\Par}_{n,\mn,S})$ is generated  by line bundles $\mathcal{B}(\vec{\lambda},\ell)$ where $\vec{\lambda}=(\lambda^1,\dots,\lambda^s)$ is an $s$-tuple of weights for $\op{GL}(n)$ (these are $n$ tuples of integers without any equivalence).
\item Two such tuples $(\vec{\lambda},\ell)$ and $(\vec{\mu},\ell')$ give rise to the same element of $\Pic(\widetilde{\Par}_{n,\mn,S})$, if $\ell=\ell'$, $\vec{\lambda}=\vec{\mu}$ as representations
of $\op{SL}(n)$, and $\sum_{i\in S}|\lambda^i|=\sum_{i\in S}|\mu^i|$.
\end{enumerate}
\end{lemma}
\subsection{Numerical objectives}
\begin{lemma}
The map
\begin{equation}\label{temper1}
 \ParL\to \widetilde{\Par}_{r,\mathcal{O}(-d),S}\times \widetilde{\Par}_{n-r,\mathcal{O}(d-D),S}
 \end{equation}
 induces a surjection on Picard groups (the map $\Psi$ in \eqref{mappo} below).
 \end{lemma}
 \begin{proof}
 Using Proposition  \eqref{LaSo}, it is easy to see that the composite of \eqref{temper1} and \eqref{temper2} is surjective on Picard groups, and by Lemma \ref{mass2}, both  \eqref{temper1} and \eqref{temper2} are surjective on Picard groups.
  \end{proof}
 We have two numerical objectives.
\begin{equation}\label{mappo}
\Pic (\widetilde{\Par}_{r,\mathcal{O}(-d),S})\tensor  \Pic (\widetilde{\Par}_{n-r,\mathcal{O}(d-D),S})\leto{\Psi} \Pic(\ParL)\leto{\op{Ind}}  \Pic(\Par_{n,\mn,S}).
\end{equation}
\begin{enumerate}
\item[(O1)] Describe the composite \eqref{mappo}
\item[(O2)] Let  $\mathcal{A}\in\Pic^+_{\Bbb{Q}}(\Par_{r,\mathcal{O}(-d),S})\times \Pic^+_{\Bbb{Q}}(\Par_{n-r,\mathcal{O}(d),S})=\Pic^{+,\deg=0}_{\Bbb{Q}}(\ParL)$ (Lemma \ref{BachMass}). Find the inverse image of $\mathcal{A}$  under the map $\Psi$ in \eqref{mappo}(tensored with rationals).
\end{enumerate}
\subsubsection{The numerical objective (O2)}
Given a  $\mathcal{A}=\mathcal{B}(\vec{\lambda},\ell)\tensor\mathcal{B}(\vec{\mu},m)$  as above, normalise $\vec{\lambda}$ and $\vec{\nu}$ (these are weights for the special linear group, so we can add $1$ to all entries of any $\lambda^i$ etc)
so that multiplication by $t$ acts trivially on each of the two factors of $\mathcal{A}$, i.e.,
$-d+ \sum_{i=1}^s \frac{|\lambda^i|}{\ell}=0$ and $d-D +\sum_{i=1}^s\frac{|\mu^i|}{m}=0$. The lift
to $\Pic_{\Bbb{Q}} (\widetilde{\Par}_{r,\mathcal{O}(-d),S})\tensor  \Pic_{\Bbb{Q}}(\widetilde{\Par}_{n-r,\mathcal{O}(d-D),S})$ is just the same line bundle (note that the $\lambda^i$ and $\mu^i$ may no longer be integral).
\subsubsection{The numerical objective (O1)}
We first describe how to induce line bundles $\mathcal{B}(\vec{\lambda},\ell)\tensor\mathcal{B}(\vec{\mu},m)$ in $\Pic (\widetilde{\Par}_{r,\mathcal{O}(-d),S})\tensor  \Pic (\widetilde{\Par}_{n-r,\mathcal{O}(d-D),S})$, with $\ell=m=0$, so that there are no determinant of cohomology factors.
We will define a line bundle $B(\vec{\delta},0)$ on $\Par_{n,\mathcal{N},S}$ as follows: Write
$I^j=\{i^j_1<i^j_2<\dots < i^j_r\}$ and $[n]-I^j=\{k^j_1<\dots < k^j_{n-r}\}$ for $j\in[s]$. Let $\delta^j_{i^j_c} =\lambda^j_c$ for $c\in [r]$ and $\delta^j_{k^j_d} =\nu^j_d$ for $d\in [n-r]$.

\begin{proposition}\label{formInd}
$$\op{Ind}(\mathcal{B}(\vec{\lambda},0)\tensor\mathcal{B}(\vec{\mu},0)) = B(\vec{\delta},0)\tensor\mathcal{O}(\sum b_{(a,j)}D(a,j))$$
where the sum is over all the divisors $D(a,j)$ in Theorem \ref{Adagio}, and the $b_{(a,j)}$ are as follows:
\begin{enumerate}
\item Suppose $a>1$, then $a\in I^j$ and $a-1\not\in I^j$. In this case, $b_{(a,j)}= \delta^j_{a}-\delta^j_{a-1}$.
\item If $a=1$, then $b_{(a,j)}= \delta^j_{1}-\delta^j_{n}$.
\end{enumerate}
\end{proposition}
\begin{proof}
The method follows the one used in \cite{BKiers}. Both sides of the desired equality restrict to the same line bundle on the open subset ${\Omega}^0(d,r,\mn,n,\vec{I})-\mathcal{R}$ of $\Par_{n,\mathcal{N},S}$. The difference is then a linear combination of the divisors $D(a,j)$ (which are the codimension one components of the complement). But since the left hand side is in $\Pic'$, the values of $b_{(a,j)}$ are as given.
\end{proof}

For the induction of the determinant of cohomology bundles from the factors we proceed as follows.
We know that the induction of  a line bundle $\mathcal{O}(E)$ is zero whenever the support of $E$  lies inside the ramification divisor $\mathcal{R}_L$. The line bundle $\mathcal{O}(\mathcal{R}_L)$ is described in Proposition \ref{basicC} below. We define the terms that appear first.

Introduce $\lambda^1,\dots,\lambda^s$ highest weights for representations of $\op{GL}(r)$ by $\lambda^j=\lambda(I^j)$ using Definition  \ref{correspondence}. Let $\ma$ and $\ma'$ be line bundles on $\widetilde{\Par}_{r,\mathcal{O}(-d),S}$ and $\widetilde{\Par}_{n-r,\mathcal{O}(d-D),S}$ respectively given by
\begin{enumerate}
\item The fiber of $\ma$ over a point $(\mv,\mf)\in \widetilde{\Par}_{r,\mathcal{O}(-d),S}$ is
$$\bigl(\mathcal{D}(\mv^*)^{n-r}\tensor \det \mv_x^{d-D}\tensor \bigotimes_{i=1}^s\bigl( \ml_{\lambda^i}(\mv_{p_i},F^{p_i}_{\bull})\bigr)\bigr).$$
\item The fiber of $\ma'$ over a point $(\mq,\mg)\in \widetilde{\Par}_{n-r,\mathcal{O}(d-D),S}$ is
$$\bigl(\mathcal{D}(\mathcal{T}^*)^r\tensor \det \mathcal{T}_x^{d}\otimes \ml_{(\lambda^i)^T}(\mt_{p_i},\widetilde{G}^{p_i}_{\bull})\bigr)\bigr)$$
where $\mathcal{T}=\mq^*$ and $\widetilde{\mg}$ is the induced flag on $\mathcal{T}$.
\end{enumerate}
See \cite[Section 11]{BGM} and the references therein for a proof of the following result.
\begin{proposition}\label{basicC}
The effective line bundle $\mathcal{O}(\mathcal{R}_L)$ on $\ParL$ is the pullback of the line bundle $\ma\boxtimes \ma'$ on $\widetilde{\Par}_{r,\mathcal{O}(-d),S}\times \widetilde{\Par}_{n-r,\mathcal{O}(d-D),S}$ where $\mathcal{A}$ and $\mathcal{A}'$ are as defined above.
\end{proposition}

Now $\dim H^0(\widetilde{\Par}_{r,\mathcal{O}(-d),S}\times \widetilde{\Par}_{n-r,\mathcal{O}(d-D),S},\ma\boxtimes\ma')=1$ by Proposition \ref{fulC}.  Let $s\tensor s'$ be a non-zero global section of the one dimensional space above  with $s$ and $s'$ global sections of  $\ma$ and $\ma'$ respectively. The zero set of the section $s\tensor 1$ of $\mathcal{A}\boxtimes \mathcal{B}(\vec{0},0)=\mathcal{A}\boxtimes \mathcal{O}$ lies inside $\mathcal{R}_L$. Since we have computed the inductions of all but the determinant of cohomology factor of $\ma\boxtimes\mathcal{O}$ in Proposition \ref{formInd}, we get a formula for the induction of the determinant of cohomology of the subbundle. One has a similar formula for the induction of the determinant of cohomology of the quotient bundle.

\subsection{A general stack theoretic statement}
Let ${\ms}$ be a connected Artin stack with a morphism to the Picard stack of line bundles of degree $d$ on  $\Bbb{P}^1$. Denote the map by $x\mapsto \ml(x)$. Also assume that the action of $\Bbb{C}^*$ on $\Pic(\Bbb{P}^1)$ by the action of $z\in \Bbb{C}^*$ given by multiplication by $z^n$ (for some fixed $n$) on the line bundle (and identity on $\Bbb{P}^1$) lifts to an action on ${\ms}$ by isomorphisms.

Let $\widehat{\ms}$ be the stack that parametrizes objects $x\in {\mathcal{S}}$ together with an isomorphism $\gamma:\ml(x)\to\mathcal{O}_{\Bbb{P}^1}(d)$.
We get a morphism of stacks $A:\widehat{\ms}\to{\ms}$; examples include    ${\Par}_{n,\mn,S}\to  \widetilde{\Par}_{n,\mn,S}$, and also the maps \eqref{temper1} and \eqref{temper2}.
We wish to compare the Picard groups of $\ms$ and $\widehat{\ms}$ under the natural  pullback map 
\begin{equation}\label{temper3}
A^*:\Pic({\ms})\to \Pic(\widehat{\ms})
\end{equation}
\begin{enumerate}
\item Let $\mf$ be the line bundle on ${\mathcal{S}}$ with fiber at $x$ given by $\ml(s)_p$ where $p$ is any point of $\Bbb{P}^1$ (these line bundles are (non-canonically) isomorphic for all $p\in\Bbb{P}^1$: choose an isomorphism $\ml(x)\to \mathcal{O}_{\Bbb{P}^1}(d)$ and use a non-canonical map $\mathcal{O}(d)_p\to \mathcal{O}(d)_q$). Clearly $\mf$ maps to the trivial bundle under \eqref{temper3}.
\item Every line bundle on ${\ms}$ has an index corresponding to the action of $\Bbb{C}^*$ (this number for $\mf$ is $n$). Let $\Pic({\ms})^{\deg=0}$ be the subgroup of index $0$. Effective line bundles on ${\ms}$ have index zero.
\end{enumerate}
\begin{lemma}\label{mass2}
\hspace{2cm}
\begin{enumerate}
\item[(i)]The kernel of \eqref{temper3} is generated by $\mf$.
\item[(ii)] $\Pic({\ms})^{\deg=0}_{\Bbb{Q}}\leto{\sim} \Pic_{\Bbb{Q}}(\widehat{\ms})$.
\item[(iii)] $\Pic^+({\ms})\leto{\sim} \Pic^+(\widehat{\ms})$. If $\ml$ is a degree zero  line bundle on ${\ms}$, then the pullback map
$H^0({\ms},\ml)\to H^0(\widehat{\ms},A^*\ml)$ is an isomorphism.
\end{enumerate}
\end{lemma}
\begin{proof}
Suppose the line bundle $\mg$ on ${\ms}$ pulls back  to the trivial bundle on $\widehat{\ms}$ under \eqref{temper3}. Therefore for every choice of $x$ and isomorphism $\eta:\ml(x)\to \mathcal{O}_{\Bbb{P}^1}(d)$
we get a non zero element in $s(x,\eta)\in\ml(x)$.
There is an isomorphism between the  objects $(x,\eta)$ and $(x,t^n\eta)$ of $\widehat{\ms}$ (induced by the $\Bbb{C}^*$ action). Therefore $s(x,t^n\eta)=t^m s(x,\eta)$ forcing $m$ to be a multiple of $n$. It is easy to see that replacing $\mg$ by $\mg\tensor\mf^{m/n}$ we  may assume that $m=0$ and get that $\mg$ is itself trivial. This proves (i).

For (ii) let $\mk$ be an arbitrary line bundle on $\widehat{\ms}$. Given $x\in\ms$ choose an arbitrary isomorphism $\eta:\ml(x)\to \mathcal{O}_{\Bbb{P}^1}(d)$. We will show that the tensor powers $\mk^n_{x,\eta}$ and $\mk^n_{x,t\eta}$ are canonically isomorphic for any $t\in \Bbb{C}^*$. We have given isomorphisms
$\mk_{x,\eta}\to \mk_{x,t^n\eta}$. One gets an action of $\mu_n$ on $\mk_{x,\eta}$, which is therefore
trivial on the tensor power $\mk^n$. Let $t'$ be an $n$th root of $t$, we get an isomorphism between $\mk^n_{x,\eta}$ and $\mk^n_{x,t\eta}$ by using the $n$th tensor power of the given isomorphism $\mk_{x,\eta}\to \mk_{x,t'^n\eta}$. This isomorphism does not depend upon the choice of $t'$  since any two choices have a ratio in $\mu_n$ and hence can be used to define a line bundle corresponding to $\mk^n$ on ${\ms}$ of index $0$. The map in (ii) is therefore surjective. By (i) the map is also injective.

For (iii), note that effective line bundles on ${\ms}$ have index zero. We only need to show the surjection part, and proceed as in (ii), and note that the action of $\mu_n$ on an effective line bundle $\mk$ on $\widehat{\ms}$ is trivial. The map on global sections is an isomorphism by a similar argument.
\end{proof}
\subsection{Some remarks on induction in the case $d=D=0$}
Consider the open substack $U$ of $\widetilde{\Par}_{r,\mathcal{O},S}$ formed by bundles which are trivial. Let $\mv\in U$. We have a natural
isomorphism $H^0(\pone,\mv)\to\mv_x$. Since $H^1(\pone,\mv)$ is trivial, we have $\mathcal{D}(\mv^*)\leto\sim \det(\mv_x)$. It is easy to see that the corresponding map on $\widetilde{\Par}_{r,\mathcal{O},S}$ has a simple pole along the complement of $U$.

This plays a role in the analysis of the induction of $\mathcal{D}(\mv^*)$. Assume $1\not\in I^i$ for any $i$. The induction of $\mathcal{D}(\mv^*)$ will be exactly equal to the induction of $\det(\mv_x)$ as long as the locus $L$  in $\widehat{\Omega}^0(d,r,\mn,n,\vec{I})$ where the subbundle is non-trivial is codimension $\geq 2$. Clearly the locus $L$ lies over points of ${\Par}_{n,\mathcal{O},S}$
 where the underlying rank $n$-bundle is non-trivial. A count of the number of non-trivial $\mv$ over $\mw=\mathcal{O}^{\oplus n}\oplus \mathcal{O}(1)\oplus\mathcal{O}(-1)$ (with general flags) was made in \cite[Section 9.6]{BGM}. If $\vec{\lambda}$ are the corresponding $\op{SL}(r)$ weights, then the number of non-trivial $\mv$ is equal to the rank of $\op{SL}(r)$ conformal blocks of level  $n-r-1$. If this number is zero, then there are no non-trivial
 $\mv$ in codimension 1, and the induction of $\mathcal{D}(\mv^*)$ coincides with the induction of $\det\mv$. If this number is one, then the dual situation in $\Gr(n-r,n)$ has the same properties as above (see \cite[Prop 1.1]{BGM}) so that the determinant of cohomology of $\mq$ has a formula in terms of $\mq_x^*$. Since $1\not\in I^i$ for any $i$, we have $\mathcal{D}(\mv^*) \tensor \mathcal{D}(\mq^*)= \mathcal{D}(\mw^*)$, we will then get formulas for $\mathcal{D}(\mv^*)$.

\subsection{Vertices always lie on regular faces}
The following easy property stated in the introduction shows that we have constructed all vertices of
$P_s(n)$.
\begin{lemma}\label{pinky}
Any vertex of $P_{s}(n)$ is on a wall given by equality in one of the inequalities \eqref{wall}, with
$\langle\sigma_{I^1}, \sigma_{I^2},\dots,\sigma_{I^s}\rangle_d= 1$.
\end{lemma}
\begin{proof}
If a vertex $\vec{a}$ of $P_{s}(n)$ is not on such a wall, then $\vec{a}$ is a vertex of $\Delta_n^s$ which is in $P_{n}(s)$. Such vertices correspond to scalar matrices $\zeta^{m_i}I$ with $\sum m_i$ a multiple of $n$, and $\zeta$ a primitive $n$th root of unity.

Furthermore, any point of $\Delta_n^s$ in a suitable neighbourhood of $\vec{a}$ would again be in $P_n(s)$, but this is a contradiction because the only point in $P_{s}(n)$  with first $s-1$ coordinates equalling $\zeta^{m_i}I$, $i=1,\dots,s-1$ is $\vec{a}$, since the entries of $\vec{a}$ correspond to central elements of $\operatorname{SU}(n)$ (such an argument breaks down in other types).
\end{proof}
\section{Examples}\label{somemore}
In the following we have computed Gromov-Witten numbers in the following way: Convert them into ranks of conformal block bundles by cyclically shifting of the weights $d$ times using the formulas given in Section \ref{QCB}. The ranks of conformal blocks were computed using Swinarski's software package \cite{Swin}.

\subsection{An example from $\op{Gr}(2,4)$}\label{oldie}
Let $I^1=\{1,4\}$,  $I^2=I^3=\{1,3\}$, and $d=1$. The Gromov-Witten number $\langle \sigma_{I^1},\sigma_{I^2},\sigma_{I^3}\rangle_1$ equals one by using formulas from Section \ref{stuffed}. Consider the pair $(1,4)$, and the divisor $D(1,4)$ this corresponds to $J^1=J^2=J^3=\{1,3\}$, and $d'=1$. Writing $\mathcal{O}(D(a,j))=\mathcal{B}(\vec{\lambda},\ell)$, we find formulas for $\lambda^1=\lambda^2=\lambda^3$ and $\ell$ using Theorem \ref{brahms}. The formula for $\ell$ is  $\langle \sigma_{J^1},\sigma_{J^2},\sigma_{J^3},\sigma_{J^4}\rangle_2$ with $J^4=\{1,3\}$. Using formulas from Section \ref{stuffed} again, we get
$\langle \sigma_{J^1},\sigma_{J^2},\sigma_{J^3},\sigma_{J^4}\rangle_2=\langle \sigma_{A},\sigma_{A},\sigma_{A},\sigma_{A}\rangle_{2,4}$ which in turns equals
$\langle \sigma_{A},\sigma_{A},\sigma_{A},\sigma_{A}\rangle_{0,0}=2$, where $A=\{2,4\}$, the classical count of the number of lines passing through $4$ general lines in space.

To compute $\lambda^1$ we will need $\langle \sigma_{C},\sigma_{J^2},\sigma_{J^3},\rangle_1$ and  $\langle \sigma_{D},\sigma_{J^2},\sigma_{J^3}\rangle_1$ where $C=\{2,3\}$ and  $D=\{1,4\}$. The latter number is already computed as one. The former is also easily computed to be one by the formulas in Remark \ref{stuffed}.
Therefore $\lambda^1=\lambda^2=\lambda^3= \omega_1+\omega_3=(2,1,1,0)$. The vertex in $P_4(3)$ is $(\vec{a},\vec{a},\vec{a})$ with
$\vec{a} =\frac{(1,0,0,-1)}{\ell}=(\frac{1}{2},0,0,-\frac{1}{2})$. This vertex was first found in \cite[Section 7]{BLocal}. See Section \ref{smalln} for the rank $2$ rigid unitary local system which corresponds to this vertex.

\subsection{An example of Thaddeus}
Consider the ``non-abelian vertex" on page 25 of \cite{thaddy}. Let $\sigma_1,\sigma_2$ and $\sigma_3$ be as in \eqref{dfnsigma}.
The $8\times 8$ matrices
\begin{itemize}
\item $A=\op{diag}(\sigma_1,\sqrt{-1}, -\symbo, \symbo,-\symbo,-1,-1)$
\item $B=\op{diag}(\sigma_2,\sqrt{-1},-\sqrt{-1},-1,-1, \symbo, -\symbo)$, and
\item $C=\op{diag}(\sigma_2,-1,-1,\sqrt{-1},-\sqrt{-1}, \symbo, -\symbo)$
\end{itemize}
multiply to the identity and correspond to a ``non-abelian vertex'' of $P_8(3)$. These matrices are all conjugate to each other and correspond to the point $a=(\frac{1}{2}, \frac{1}{4},\frac{1}{4},\frac{1}{4},-\frac{1}{4},-\frac{1}{4},-\frac{1}{4}, -\frac{1}{2})$. The point $(a,a,a)\in P_8(3)$ comes from one of our basic extremal rays for the data
$d=2$, $\Gr(4,8)$, $I^1=\{2,3,4,7\}$, $I^2=I^3=\{1,3,4,7\}$ with $a=2$ and $j=1$.  We compute
$\langle\sigma_{I^1},\sigma_{I^2},\sigma_{I^3}\rangle_2=1$, and the level $\ell$ is computed as $4$, and
$$\lambda^1=\lambda^2=\lambda^3=(4,3,3,3,1,1,1,0)=\omega_1+2\omega_4+\omega_7.$$ It is easy to see that
$\frac{1}{\ell}\kappa(\lambda)=a$.
\subsection{The strange dual of Thaddeus' example}\label{E1}
We compute the strange dual of the data consisting of $\lambda^1=\lambda^2=\lambda^3=(4,3,3,3,1,1,1,0)=\omega_1+2\omega_4+\omega_7$ at level $\ell=4$.
The data of $\mathcal{A}$ corresponding to the F-line bundle $\mathcal{B}(\vec{\lambda},\ell)$ on
$\Par_{8,\mathcal{O},S}$ (by Theorem \ref{egregium} and Proposition  \ref{corresp}) is $\mu^1=\mu^2=\mu^3=(7,4,4,1)$ with $n=8$.
This point corresponds to a rigid, irreducible unitary local system of rank $4$ on $\Bbb{P}^1-\{p_1,p_2,p_3\}$, with the same local monodromy ($\zeta^{-1},-1,-1,\zeta)$ at $p_1,p_2,p_3$, where
$\zeta=\zeta_8=\exp(2\pi\sqrt{-1}/8)$. By Prop \ref{propP}, this local system has infinite global monodromy.

There are 3 distinct eigenvalues at each point. The rigidity equation \eqref{egale} holds: $2 +(3-2)4^2= 3(2^2+1^2+1^2)$. This rigid local system belongs to the Goursat-III family of rigid local systems of rank $4$  (see Goursat's 1886 paper \cites{EGoursat}).

There is only one possible (optimal) Katz level lowering operation (see Section \ref{earlyfeedback}) for this example, because  at each point the eigenvalue with maximal multiplicity is unique. By a simple calculation, the resulting local system is of rank $2$, which is after a tensor product with a rank one local system, a local system on $\Bbb{P}^1-\{p_1,p_2,p_3\}$ with local monodromies $(-\zeta,-\zeta^{-1})=(\zeta^5,\zeta^4)$, $(1,\zeta^2)$ and $(1,\zeta^{-2})=(1,\zeta^6)$. This is a local system of hypergeometric type which is not unitary (see Theorem \ref{BHhere} below), since $(\zeta^{-5},\zeta^{-4})$ and
$(1,\zeta^2)$ do not interlace on the unit circle.

However  the choice of eigenvalues $-1$, $-1$ and $\zeta$ for the Katz operations at the three points gives rise to a rank $3$ unitary local system of hypergeometric type (after tensoring with a rank one)  with local monodromies $(-\zeta^{2},-1,1)=(\zeta^6,\zeta^4,1)$,$(-\zeta^{2},-1,1)=(\zeta^6,\zeta^4,1)$ and $(-\zeta^{-1},-\zeta^{-1},\zeta^{-2})=(\zeta^3,\zeta^3,\zeta^6)$. The unitarity is checked using Theorem \ref{BHhere}: Consider the local system corresponding to a twist by a rank one local system which has local monodromies $(1,1,\zeta^3)$, $(\zeta^6,\zeta^4,1)$ and $(\zeta,\zeta^3,\zeta^7)$. We may take $(\alpha_1,\alpha_2,\alpha_3)=(0,\frac{4}{8},\frac{6}{8})$ and $(\beta_1,\beta_2,\beta_3)=(1-\frac{7}{8},1-\frac{3}{8},1-\frac{1}{8})=(\frac{1}{8},\frac{5}{8},\frac{7}{8})$. The interlacing condition
therefore holds.  Also see Lemma \ref{hyperreduction}, and Question \ref{genphen}.
\subsection{An example that gives the determinant of cohomology  line bundle}
Consider  the case $d=0$, $1\in I^1$, $n\not\in I^1$. Consider $D(1,1)$, noting that  $d'=-1$.
The divisor $D(1,1)$ is clearly supported on  the divisor of $\Par_{n,\mathcal{O},S}$ where the underlying vector bundle is not trivial as a bundle (for otherwise it cannot have subsheaves of positive degree).  The corresponding line bundle is $\mathcal{D}$,  the determinant of cohomology. and hence $D(1,1)=\mathcal{D}$ a fact that follows from our numerical formulas as well.
It is interesting to note that we cannot have $1\in I^2, n\not\in I^2$  as well because the corresponding classical intersection number is zero since the line  $E^{p_1}_1$ is not  a subset of  $E^{p_2}_{n-1}$ for general flags.

\subsection{A rank $6$ rigid local system}\label{wilson}
This example is due to A. Wilson. In the context of Proposition  \ref{practical}, consider the cycle $E=C(0,3,\mathcal{O},9,\vec{J})$ with $|S|=3$ on $\Par_{9,\mathcal{O},S}$ with
$J^1=\{2,6,9\}$, $J^2=J^3=\{3,6,9\}$. The line bundle $\mathcal{O}(E)=\mathcal{B}(\vec{\lambda},\ell)$ computed by Proposition  \ref{enumerative} has $\ell=6$ and $\lambda^1= 3\omega_2+2\omega_6$ and $\lambda^2=\lambda^3=2\omega_3+ 2\omega_6$. The inequality in Proposition  \ref{practical} (ii) holds and therefore the line bundle $\mathcal{B}(\vec{\lambda},\ell)$ is an F-line bundle on $\Par_{9,\mathcal{O},S}$. The corresponding unitary irreducible rigid local system is of rank $6$ with multiplicity decompositions $(3,2,1)$, $(2,2,2)$ and $(2,2,2)$ at $p_1,p_2$ and $p_3$. The rigidity equation \eqref{eqRi} holds:
$6^2+2=(3^2+2^2+1^2) + 2(2^2+2^2+2^2)$. The eigenvalues of the local monodromies at the 3 points are: at $p_1$, $\zeta^2$ (with multiplicity $3$) $\zeta^6$ (multiplicity $2$) and $1$ with multiplicity $1$; and at $p_2$ and $p_3$, the eigenvalues are $\zeta^3,\zeta^6$ and $1$ (with multiplicity $2$ each). Here $\zeta=\exp(2\pi\sqrt{-1}/9)$. The multiplicity of one at the first point shows that the cycle is of the form $\mathcal{O}(D(1,1))$ for $I^1=\{1,2,6\}$ and $I^2=I^3=\{3,6,9\}$, $d=1$ in Theorem \ref{Adagio} on $\Gr(3,9)$.

The Katz rank lowering algorithm (see Section \ref{earlyfeedback}) when applied to this rank six irreducible unitary rigid local system leads to irreducible rigid local systems of rank $5$ with non-semisimple local monodromy for some  of the choices in the algorithm (e.g., we pick the eigenvalues of maximal multiplicity to be $\zeta^2$, $1$ and $1$ at $p_1,p_2$ and $p_3$ respectively.

\subsection{An infinite collection of irreducible  unitary rigid local systems}\label{KO}
J. Kiers and G. Orelowitz have reported the following series of examples that they have worked out in the context of  their joint work with S. Gao and A. Yong \cite{Kiers_etal}.

Let  $\Gr(k,n)$ with $n=3k-1$ for $k\ge2$, $I^1=I^3=\{2,5,8,\dots,3k-1\}$, $I^2=\{k,2k+1,2k+2,\dots 3k-1\}$ and $d=0$. The corresponding intersection number is computed to be one (this is a case of the Pieri formula).  In the setting of Theorems \ref{Adagio} and \ref{brahms}, consider $D(2k+1,2)=\mathcal{B}(\vec{\lambda},\ell)$. Here $d'=0$ and $J^1=J^3=I^1=I^3$, and $J^2$ is obtained by replacing the $2k+1$ in $I^2$ by $2k$.

After a careful combinatorial count of  intersection numbers (using the Littlewood-Richardson rule, the rim-hook formulas for the quantum numbers \cite{BCFF}, and Remark \ref{earlyAM}), Kiers and Orelowitz have obtained the following formulas: $\ell=(k-1)^2+1$,
$\lambda^1=\lambda^3= (k-1)(\omega_2 +\omega_5 + \omega_8+ \dots +\omega_{3k-4})$ and $\lambda^2=(k-1)\omega_{k} +\omega_{2k}$.

Setting $\zeta=\exp(2\pi\sqrt{-1}/n)$, the corresponding irreducible unitary rigid local system has rank $(k-1)^2+1$ and has local monodromy at $p_1$ and $p_3$ conjugate to the diagonal matrix with the following eigenvalues: $\zeta^2,\zeta^5,\zeta^8,\dots, \zeta^{3k-4}$ (each with multiplicity $k-1$) and $1$ (with multiplicity $1$). The local monodromy at $p_2$ is conjugate to the diagonal matrix with eigenvalues $\zeta^k$ (with multiplicity $k-1$), $\zeta^{2k}$ (with multiplicity $1$) and $1$ (with multiplicity $(k-1)(k-2)$).  The rigidity equation  \eqref{eqRi}  can be verified from these formulas.

For $k\ge 5$, these examples violate the speculation $\ell<n$ from an earlier version of this paper.

\subsection{An example of a vertex which is not an F-vertex}\label{Rex}
The following example of Ressayre (reported in \cite[Section 10.5]{BHermit}) was made in the classical context:
Consider $G=\op{SL}(9)$, and  $\lambda=\omega_8+\omega_7+\omega_3$, $\mu=\nu=\omega_6+\omega_3$. The ray $\Bbb{Q}_{\ge 0}(\lambda,\mu,\nu)$ is an extremal ray of the tensor cone, but the rank of $(V_{\lambda}\tensor V_{\mu}\tensor V_{\nu})^G$ is $2>1$. When converted into a conformal block, the level is $4$, but the tuple $(\lambda,\mu,\nu,3)$ at level three has the same rank. Any level less than three will not
lead to a point in the fundamental alcove (for the first point) and therefore $(\lambda,\mu,\nu,3)$ gives an extremal ray of $\Par_{9,\mathcal{O,S}}$ with $|S|=3$ (use Remark \ref{classicall}, any expression of a multiple as a sum will need to project to the trivial decomposition in the classical context, so one of the summands should be a multiple of $\mathcal{D}$, making the other summand have levels that violate \eqref{weyl} in Theorem \ref{blacktea} at the first point),  and hence $(\lambda,\mu,\nu,3)$ leads to a vertex of $P_8(3)$.

Analysing this example further as in \cite[Section 10.5]{BHermit}, we take the classical intersection $d=0$ and $I_1=\{3,7,8\}$, $I_2=I_3=\{3,6,9\}$ in $\Gr(3,9)$, and show that $(\lambda,\mu,\nu,3)$ is induced from the extremal ray corresponding to $\op{SL}(3)$, and the tuple $(\omega_2,\omega_2,\omega_2,1)$ (which is a basic extremal ray  coming from Theorem \ref{Adagio} for $\op{SL}(3)$).

\section{Some classical irreducible rigid local systems}\label{classH}
The hypergeometric functions ${}_{\ell}F_{\ell-1}$ which give irreducible rigid rank $\ell$ local systems on $\Bbb{P}^1-\{0,1,\infty\}$, and the Pochhammer hypergeometric functions, which give irreducible rank $\ell$ local systems on $\Bbb{P}^1-\{\ell+1\text{ points}\}$, are classically well-studied rigid local systems. We recall \cite{BH, Haraoka} criteria for when these are unitary. The corresponding strange duals (Theorem \ref{egregium}) will give us F-vertices of suitable $P_k(3)$ and $P_k(\ell+1)$ respectively. This produces a rich collection of vertices in the multiplicative eigenvalue problem.
\subsection{The hypergeometric functions ${}_{\ell}F_{\ell-1}$}\label{hyper}
Let $\alpha_1,\dots,\alpha_{\ell};\beta_1,\dots,\beta_{\ell}\in \Bbb{C}$. Let $D=z\frac{d}{dz}$ and consider the differential equation
\begin{equation}\label{math383}
z(D+\alpha_1)\cdots (D+\alpha_{\ell})F=(D+\beta_1-1)\cdots (D+\beta_{\ell}-1)F
\end{equation}
By classical theory this is a differential equation with regular singularities, which are at $0,1,\infty$.
The local exponents are $1-\beta_1,\dots,1-\beta_{\ell}$ at $z=0$, $\alpha_1,\dots,\alpha_{\ell}$, at $z=\infty$ and
$0,1,\dots,\ell-2, -1+ \sum_{i=1}^{\ell}(\beta_i-\alpha_i)$ at $z=1$ (the local monodromies are $\exp(2\pi\sqrt{-1}x)$ where $x$ runs through the local exponents).
\begin{defi}
For $\alpha\in \Bbb{R}$ define its fractional part $\langle \alpha\rangle$ by $\langle \alpha\rangle=\alpha-\left \lfloor{\alpha}\right \rfloor $.
\end{defi}

\begin{theorem}\label{BHhere}\cite[Corollary 4.7]{BH}
The monodromy of the differential equation \eqref{math383} can be conjugated to the unitary group $U(\ell)$ if and only if the $\alpha_i$ and $\beta_i$ are real, and
either of the following interlacing conditions hold (after rearranging the $\alpha_i$ and $\beta_i$ so that their fractional parts are weakly increasing):
$$\langle\alpha_1\rangle <\langle\beta_1\rangle<\langle\alpha_2\rangle <\langle\beta_2\rangle<\dots<\langle\alpha_{\ell}\rangle <\langle\beta_{\ell}\rangle,\ \ \ \text {or }\ \langle\beta_1\rangle<\langle\alpha_1\rangle <\langle\beta_2\rangle<\langle\alpha_2\rangle <\dots<\langle\beta_{\ell}\rangle<\langle\alpha_{\ell}\rangle .$$
\end{theorem}
\subsection{The Pochhammer equation}\label{pokhie}
Let $t_1,\dots,t_{\ell}$ be $\ell>1$ distinct points in $\Bbb{A}^1=\Bbb{P}^1-\{\infty\}$, and let $\lambda_1,\dots,\lambda_{\ell},\rho$ be complex numbers
satisfying $\sum_{i=1}^{\ell}\lambda_i\neq \ell\rho$. Let

$$A(\lambda,\rho)=
  \left[ {\begin{array}{cccc}
   \lambda_1 & \lambda_1-\rho & \dots & \lambda_1-\rho \\
   \lambda_2-\rho & \lambda_2 &\dots  & \lambda_2-\rho\\
   \vdots & \vdots &\vdots & \vdots \\
   \lambda_{\ell}-\rho & \lambda_{\ell}-\rho & \dots & \lambda_{\ell}\\
\end{array} }\right],\  T= \left[ {\begin{array}{cccc}
   t_1 & 0 & \dots & 0 \\
   0 & t_2 &\dots  & 0\\
   \vdots & \vdots &\vdots & \vdots \\
   0 & 0 & \dots & t_{\ell}\\\end{array} }\right].$$

The system of equations, with $Y=(y_1(z),\dots,y_{\ell}(z))^T$
\begin{equation}\label{math383sys}
(z-T)\frac{dY}{dz}= A(\lambda,\rho)Y
\end{equation}
is the Pochhammer system of rank $\ell$, and is known to have regular singularities. The exponents at $t_i$ are $0$ with multiplicity $\ell-1$ and $\lambda_i$ with multiplicity $1$, and
the local exponents at $\infty$ are $-\rho$ with multiplicity $\ell-1$ and $-\rho'$ of multiplicity $1$ where
$\rho'=\sum_{i=1}^{\ell}\lambda_i -(\ell-1)\rho.$

The corresponding local system is irreducible if and only if
$\lambda_1-\rho,\lambda_2-\rho,\dots,\lambda_{\ell}-\rho,\rho,\rho'$ are all not in $\Bbb{Z}$ \cite[Proposition 1.3]{Haraoka}. We will assume that this is indeed the case. In this case the monodromy group can be conjugated to $U(\ell)$, by a result of Haraoka, if and only if one of the following conditions holds \cite[Proposition 1.5]{Haraoka}
\begin{enumerate}
\item $\langle\rho\rangle <\langle\lambda_j\rangle, j=1,\dots,\ell$ and $\sum_{j=1}^{\ell} \langle\lambda_j\rangle<(\ell-1)\langle\rho\rangle +1$, or
\item $\langle\rho\rangle >\langle\lambda_j\rangle, j=1,\dots,\ell$ and $(\ell-1) \langle\rho\rangle< \sum_{j=1}^{\ell} \langle\lambda_j\rangle$.
\end{enumerate}

\section{The Katz algorithm over complex numbers and unitary rigid local systems}\label{earlyfeedback}
Let $\mu_n\subseteq\Bbb{C}^*$ be the subgroup of $n$th roots of unity. Consider the set $\op{Rig}(n,S)$ of irreducible rigid local systems $\mf$ on $\Bbb{P}^1_{\Bbb{C}}-S$ where $S\subseteq \Bbb{A}_{\Bbb{C}}^1$, all of whose local monodromies (possibly non semisimple) have eigenvalues in $\mu_n$.

In his fundamental work \cite{Katz}, Katz introduces two invertible operations on $\op{Rig}(n,S)$ called the middle convolution $\op{MC}_{c}$, $c\in \mu_n$, and, the middle  tensoring $MT_{\mathcal{L}}$ with rank one local systems $\ml$ on $\Bbb{A}^1-S$.  The effect of these operations on local monodromy classes  is given by explicit formulas. The central monodromy at $\infty\in \Bbb{P}^1$ produced by Katz's operations,  can be removed by the middle tensoring operation. The Galois group of the cyclotomic field $\Bbb{Q}(\mu_n)$ acts on this set, and on Katz's operations. If $n=1$,   $\op{Rig}(n,S)$  is a singleton by Katz's algorithm, see  e.g., \cite[Proof of Theorem 9.15]{MM}, and we will assume $n>1$.

Given an irreducible rigid local system $\mf\in \op{Rig}(n,S)$ as above of rank $>1$, Katz \cite{Katz} describes a successive application of a suitable $MT_{\mathcal{L}}$ and $\op{MC}_{c}$ which lowers the rank of the  local system. To obtain $\mathcal{L}$ and $c$, one needs to choose at each point of $S$, an eigenvalue with  the maximum dimension of the corresponding eigenspace of the local monodromy. This choice can be made canonical if we use a ``smallest angle'' criterion for  the eigenvalues  (with angles taken in $[0,2\pi)$), but this is not a Galois invariant choice.

We will picture $\op{Rig}(n,S)$  as the vertices of a graph  $\op{R}(n,S)$, with a directed arrow pointing downwards for each Katz rank lowering operation.  At the bottom, we find rank one local systems. The example of hypergeometric functions
 \cite[Theorem 3.5]{BH} shows that this graph is infinite. On the other hand it is easy to see that local systems of rank $\ell$, the height in this graph, form a finite subset for each $\ell$.

 The local systems in $\op{Rig}(n,S)$  which are unitary form a   finite, but not Galois invariant subset (Theorem \ref{egregium}), with restricted local monodromies (Proposition \ref{propP}). The characterizations of irreducible unitary rigid local systems given in Theorems \ref{ABW} and \ref{egregium} are  in terms  of the logarithms of the eigenvalues of local monodromies.

The rank lowering operations in Katz's algorithm may take an irreducible unitary rigid local system
to an irreducible rigid local system which is not unitary. A rank $4$ example  is given in Section \ref{E1}. In this example there  are no choices in the Katz algorithm (at the first step) since the eigenvalues with maximal multiplicities are unique. However, another  choice of middle tensor and convolution gives rise  to a unitary rank $3$ local system (in the Katz algorithm the rank is reduced to $2$). This rank $3$ local system is of hypergeometric type, and can be checked to be unitary. We apply a Katz lowering operation to this, and the result is a rank $2$ unitary local system (see Lemma \ref{hyperreduction}).
\begin{question}\label{genphen}
Is this a general phenomenon?, i.e., starting from a given irreducible unitary rigid local system, can we always find a sequence of Katz's operations (not necessarily the optimal ones), which  takes place entirely in the realm of irreducible unitary rigid local systems, down to rank one local systems, with the ranks decreasing in the sequence?
\end{question}
We will not hazard a guess here, but only note that if this is indeed the case, then the lower rank unitary local systems will again give rise (by Theorem \ref{egregium}) to F-vertices of $P_n(s)$ with decreasing levels, and we could try to find the analogue of Katz's operation on the set of F-vertices of $P_n(s)$, which is an interesting problem in quantum Schubert calculus.

 Question \ref{genphen} holds  for all hypergeometric local systems (see Lemma \ref{hyperreduction} below), and for all  irrreducible Pochhammer local systems (which reduce to rank one local systems in one step of the Katz algorithm).  Because of  the example  in Section \ref{E1} and the above discussion, we will add rank lowering Katz operations to the above graph (i.e., add more edges).
 \begin{question}\label{information}
Looking up from a rank one local system, what can we say about the set of paths in  $\op{R}(n,S)$ from  irreducible unitary rigid local systems down to this rank one local system?
\end{question}

See Section \ref{wilson} for an example where a Katz rank lowering operation, applied to an irreducible unitary rigid local system, produces a local system with non-semisimple local monodromy.

\begin{question}\label{Q2}
\hspace{2em}
\begin{enumerate}
\item[(A)] We may decorate the elements of  $\op{Rig}(n,S)$ with the length $h$ of the corresponding Hodge filtration  (computable using \cite{DS}) and the signature $(p,q)$ with $p\leq q$ of the corresponding variation of Hodge structure.

 Is the subset of  $\op{Rig}(n,S)$  with the length of Hodge filtration $h\leq m$ for a fixed $m$, a finite set?  The case of $m=1$  is covered by Theorem \ref{egregium}.  A consequence  in the case of hypergeometric functions is verified below.
\item[(B)] Is there an upper bound for the rank $\ell$ of a rigid local system in terms of $n$ (the orders of the local monodromies), the length  $h$ of the Hodge filtration and  $|S|$?

\item[(C)] How does the ``Katz flow''  downwards behave with respect to the length of Hodge filtration. Is there always a flow which has weakly decreasing lengths of Hodge filtrations down to rank one local systems (which have length 1)?
(This generalizes Question \ref{genphen} above).
\end{enumerate}
\end{question}

The  finiteness assertion (A) implies the following property (A'): The
number of elements of $\op{Rig}(n,S)$  with $p\leq m$, where $(p,q)$ is the signature with $p\leq q$, is a finite set. This is because the corresponding length of the Hodge filtration  is $\leq 2p+1$ (there are no gaps in the Hodge filtration by the Griffiths transversality principle, since the local systems are irreducible).

\begin{lemma}\label{lbelow}
Property (A) holds for hypergeometric local systems.
\end{lemma}
\begin{proof}

For the hypergeometric local systems repeated eigenvalues $\alpha_t$ (using notation from Section \ref{hyper})  correspond to monodromy matrices with full Jordan blocks. Therefore
the multiplicity of any eigenvalue $\alpha_t$ is less than or equal to the length $h$ of the Hodge filtration. This shows that the rank $\ell$ of the local system is no more than $nh$. This proves (A).
\end{proof}
\begin{lemma}\label{hyperreduction}
 Let $\mathcal{F}$ be a rank $\ell$ unitary hypergeometric local system on $\Bbb{P}^1-\{p_1,p_2,p_3\}$ with $\infty\not\in \Bbb{P}^1$ (i.e., we transport a hypergeometric local system from $\Bbb{P}^1-\{0,1,\infty\}$ to $\Bbb{P}^1-\{p_1,p_2,p_3\}$ using an automorphism of $\Bbb{P}^1$). There is a choice of the lowering operation in the Katz algorithm such the result is a rank $\ell-1$ unitary hypergeometric local system on $\Bbb{P}^1-\{p_1,p_2,p_3\}$.
 \end{lemma}
 \begin{proof}
Let $\zeta=\exp(2\pi\sqrt{-1}/n)$, and assume that the eigenvalues of the local monodromies are in $\mu_n$.  Without loss of generality assume (using Theorem \ref{BHhere}) that the eigenvalues at $p_1,p_2$ and $p_3$ of the local monodromy are given by $(1,1,\dots,1,\zeta^m), (1,\zeta^{-a_2},\dots,\zeta^{-a_{\ell}})$, and $(\zeta^{b_1},\dots,\zeta^{b_{\ell}})$ with
  \begin{equation}\label{interlace}
  0<b_1<a_2<b_2<\dots<a_{\ell-1}<b_{\ell}<n, \ \ m=\sum_{i=2}^{\ell} a_i-\sum_{i=1}^\ell b_i.
  \end{equation}

In the Katz algorithm pick the eigenvalue $1$ at $p_1$ and $p_2$, and $\zeta^{b_1}$ at $p_3$. The resulting rank $\ell-1$  local system has central monodromy at infinity. Tensoring by a suitable rank one local system, we find a rank $\ell-1$ local system on $\Bbb{P}^1-\{p_1,p_2,p_3\}$ such that
the eigenvalues at $p_1,p_2$ and $p_3$ of the (semisimple) local monodromy are given by $(1,\dots,1,\zeta^{m+b_1}), (\zeta^{-a_2},\dots,\zeta^{-a_{\ell}})$, and $(\zeta^{b_2},\dots,\zeta^{b_{\ell}})$. This rank $\ell-1$ local system is unitary, and of hypergeometric type by Theorem \ref{BHhere}.
 \end{proof}
 \begin{remark} \label{Q2rem}
 The Hodge numbers of hypergeometric local systems have been determined in \cite{fedorov}. Arguing along the same lines as in Lemma \ref{hyperreduction} (i.e., not assuming the interlacing condition in \eqref{interlace}, just that $a_1=0$), Question \ref{Q2} (C) for hypergeometric local systems can be shown to be true. The Hodge number $h^m$, where $m$ is the multiplicity of the root $1$ at $p_1$ drops by one in the new local system, and other Hodge numbers remain the same.
\end{remark}

\vspace{0.05 in}

\noindent
Department of Mathematics, University of North Carolina, Chapel Hill, NC 27599\\
{{email: belkale@email.unc.edu}}
\vspace{0.05 in}


\begin{bibdiv}
\begin{biblist}
\bib{AW}{article}{
    AUTHOR = {Agnihotri, S.}
    AUTHOR = {Woodward, C.},
     TITLE = {Eigenvalues of products of unitary matrices and quantum
              {S}chubert calculus},
   JOURNAL = {Math. Res. Lett.},
  FJOURNAL = {Mathematical Research Letters},
    VOLUME = {5},
      YEAR = {1998},
    NUMBER = {6},
     PAGES = {817--836},

}
\bib{BLocal}{article} {
    AUTHOR = {Belkale, P.},
     TITLE = {Local systems on {$\Bbb P^1-S$} for {$S$} a finite set},
   JOURNAL = {Compositio Math.},
  FJOURNAL = {Compositio Mathematica},
    VOLUME = {129},
      YEAR = {2001},
    NUMBER = {1},
     PAGES = {67--86},
}




			
\bib{BIMRN}{article}{
    AUTHOR = {Belkale, P.},
     TITLE = {Invariant theory of {${\rm GL}(n)$} and intersection theory of
              {G}rassmannians},
   JOURNAL = {Int. Math. Res. Not.},
  FJOURNAL = {International Mathematics Research Notices},
      YEAR = {2004},
    NUMBER = {69},
     PAGES = {3709--3721},
}

\bib{BTIFR}{incollection} {
    AUTHOR = {Belkale, P.},
     TITLE = {Extremal unitary local systems on {$\Bbb
              P^1-\{p_1,\dots,p_s\}$}},
 BOOKTITLE = {Algebraic groups and homogeneous spaces},
    SERIES = {Tata Inst. Fund. Res. Stud. Math.},
    VOLUME = {19},
     PAGES = {37--64},
 PUBLISHER = {Tata Inst. Fund. Res., Mumbai},
      YEAR = {2007},
}
\bib{BHorn}{article} {
    AUTHOR = {Belkale, P.},
     TITLE = {Quantum generalization of the {H}orn conjecture},
   JOURNAL = {J. Amer. Math. Soc.},
  FJOURNAL = {Journal of the American Mathematical Society},
    VOLUME = {21},
      YEAR = {2008},
    NUMBER = {2},
     PAGES = {365--408},
}

\bib{Betrange}{article} {
    AUTHOR = {Belkale, P.},
     TITLE = {The strange duality conjecture for generic curves},
   JOURNAL = {J. Amer. Math. Soc.},
  FJOURNAL = {Journal of the American Mathematical Society},
    VOLUME = {21},
      YEAR = {2008},
    NUMBER = {1},
}
\bib{BJDG}{article} {
    AUTHOR = {Belkale, P.},
     TITLE = {Strange duality and the {H}itchin/{WZW} connection},
   JOURNAL = {J. Differential Geom.},
  FJOURNAL = {Journal of Differential Geometry},
    VOLUME = {82},
      YEAR = {2009},
    NUMBER = {2},
     PAGES = {445--465},
}
\bib{BKZ}{article} {
    AUTHOR = {Belkale, P.},
     TITLE = {Unitarity of the {KZ}/{H}itchin connection on conformal blocks
              in genus 0 for arbitrary {L}ie algebras},
   JOURNAL = {J. Math. Pures Appl. (9)},
  FJOURNAL = {Journal de Math\'{e}matiques Pures et Appliqu\'{e}es. Neuvi\`eme S\'{e}rie},
    VOLUME = {98},
      YEAR = {2012},
    NUMBER = {4},
     PAGES = {367--389},
}
		
\bib{BHermit}{article} {
    AUTHOR = {Belkale, P.},
     TITLE = {Extremal rays in the {H}ermitian eigenvalue problem},
   JOURNAL = {Math. Ann.},
  FJOURNAL = {Mathematische Annalen},
    VOLUME = {373},
      YEAR = {2019},
    NUMBER = {3-4},
     PAGES = {1103--1133},
}
\bib{BGM}{article} {
    AUTHOR = {Belkale, P.}
    AUTHOR = {Gibney, A.}
    AUTHOR = {Mukhopadhyay,
              S.},
     TITLE = {Vanishing and identities of conformal blocks divisors},
   JOURNAL = {Algebr. Geom.},
  FJOURNAL = {Algebraic Geometry},
    VOLUME = {2},
      YEAR = {2015},
    NUMBER = {1},
     PAGES = {62--90},
}


\bib{BKiers}{article} {
    AUTHOR = {Belkale, P. }
    AUTHOR = {Kiers, J.},
     TITLE = {Extremal rays in the {H}ermitian eigenvalue problem for
              arbitrary types},
   JOURNAL = {Transform. Groups},
  FJOURNAL = {Transformation Groups},
    VOLUME = {25},
      YEAR = {2020},
    NUMBER = {3},
     PAGES = {667--706},
}

\bib{BKq}{article}{
    AUTHOR = {Belkale, P.}
    AUTHOR = {Kumar, S.},
     TITLE = {The multiplicative eigenvalue problem and deformed quantum
              cohomology},
   JOURNAL = {Adv. Math.},
  FJOURNAL = {Advances in Mathematics},
    VOLUME = {288},
      YEAR = {2016},
     PAGES = {1309--1359},
      ISSN = {0001-8708},
}
\bib{BM}{article} {
    AUTHOR = {Belkale, P.}
    AUTHOR = {Mukhopadhyay, S.},
     TITLE = {Conformal blocks and cohomology in genus 0},
   JOURNAL = {Ann. Inst. Fourier (Grenoble)},
  FJOURNAL = {Universit\'{e} de Grenoble. Annales de l'Institut Fourier},
    VOLUME = {64},
      YEAR = {2014},
    NUMBER = {4},
     PAGES = {1669--1719},
}
\bib{bertie}{article} {
    AUTHOR = {Bertram, A.},
     TITLE = {Quantum {S}chubert calculus},
   JOURNAL = {Adv. Math.},
  FJOURNAL = {Advances in Mathematics},
    VOLUME = {128},
      YEAR = {1997},
    NUMBER = {2},
     PAGES = {289--305},
}
\bib{BCFF}{article} {
    AUTHOR = {Bertram, A.}
    AUTHOR=  {Ciocan-Fontanine, I.}
    AUTHOR = {Fulton,
              W.},
     TITLE = {Quantum multiplication of {S}chur polynomials},
   JOURNAL = {J. Algebra},
  FJOURNAL = {Journal of Algebra},
    VOLUME = {219},
      YEAR = {1999},
    NUMBER = {2},
     PAGES = {728--746},
}
	


\bib{BH}{article}{
    AUTHOR = {Beukers, F.}
    AUTHOR=  {Heckman, G.},
     TITLE = {Monodromy for the hypergeometric function {$_nF_{n-1}$}},
   JOURNAL = {Invent. Math.},
  FJOURNAL = {Inventiones Mathematicae},
    VOLUME = {95},
      YEAR = {1989},
    NUMBER = {2},
     PAGES = {325--354},
}
\bib{Biswas}{article} {
    AUTHOR = {Biswas, I.},
     TITLE = {A criterion for the existence of a parabolic stable bundle of
              rank two over the projective line},
   JOURNAL = {Internat. J. Math.},
  FJOURNAL = {International Journal of Mathematics},
    VOLUME = {9},
      YEAR = {1998},
    NUMBER = {5},
     PAGES = {523--533},
}
\bib{Bump}{book} {
    AUTHOR = {Bump, D.},
     TITLE = {Lie groups},
    SERIES = {Graduate Texts in Mathematics},
    VOLUME = {225},
   EDITION = {Second},
 PUBLISHER = {Springer, New York},
      YEAR = {2013},
     PAGES = {xiv+551},
}
\bib{CF}{article} {
    AUTHOR = {Ciocan-Fontanine, I.},
     TITLE = {The quantum cohomology ring of flag varieties},
   JOURNAL = {Trans. Amer. Math. Soc.},
  FJOURNAL = {Transactions of the American Mathematical Society},
    VOLUME = {351},
      YEAR = {1999},
    NUMBER = {7},
     PAGES = {2695--2729},
}
\bib{DS}{article} {
    AUTHOR = {Dettweiler, M.}
    AUTHOR =  {Sabbah, S.},
     TITLE = {Hodge theory of the middle convolution},
   JOURNAL = {Publ. Res. Inst. Math. Sci.},
  FJOURNAL = {Publications of the Research Institute for Mathematical
              Sciences},
    VOLUME = {49},
      YEAR = {2013},
    NUMBER = {4},
     PAGES = {761--800},
}


\bib{naf}{incollection} {
    AUTHOR = {Fakhruddin, N.},
     TITLE = {Chern classes of conformal blocks},
 BOOKTITLE = {Compact moduli spaces and vector bundles},
    SERIES = {Contemp. Math.},
    VOLUME = {564},
     PAGES = {145--176},
 PUBLISHER = {Amer. Math. Soc., Providence, RI},
      YEAR = {2012},
}

\bib{Faltings}{article} {
    AUTHOR = {Faltings, G.},
     TITLE = {Stable {$G$}-bundles and projective connections},
   JOURNAL = {J. Algebraic Geom.},
  FJOURNAL = {Journal of Algebraic Geometry},
    VOLUME = {2},
      YEAR = {1993},
    NUMBER = {3},
     PAGES = {507--568},
}



\bib{fedorov}{article} {
    AUTHOR = {Fedorov, R.},
     TITLE = {Variations of {H}odge structures for hypergeometric
              differential operators and parabolic {H}iggs bundles},
   JOURNAL = {Int. Math. Res. Not. IMRN},
  FJOURNAL = {International Mathematics Research Notices. IMRN},
      YEAR = {2018},
    NUMBER = {18},
     PAGES = {5583--5608},
      ISSN = {1073-7928},
}
\bib{FBull}{article} {
    AUTHOR = {Fulton, W.},
     TITLE = {Eigenvalues, invariant factors, highest weights, and
              {S}chubert calculus},
   JOURNAL = {Bull. Amer. Math. Soc. (N.S.)},
  FJOURNAL = {American Mathematical Society. Bulletin. New Series},
    VOLUME = {37},
      YEAR = {2000},
    NUMBER = {3},
     PAGES = {209--249},
}



\bib{FInt}{book} {
    AUTHOR = {Fulton, W.},
     TITLE = {Intersection theory},
    VOLUME = {2},
 PUBLISHER = {Springer-Verlag, Berlin},
      YEAR = {1998},
     PAGES = {xiv+470},
}

\bib{Kiers_etal}{article}{
  author={Gao, S},
  author={Kiers, J},
  author={Orelowitz, G},
  author={Yong,A.}
   title={The Kostka semigroup and its Hilbert basis}
   note ={ arXiv:2102.00935}

}


\bib{Gepner}{article} {
    AUTHOR = {Gepner, D.},
     TITLE = {Fusion rings and geometry},
   JOURNAL = {Comm. Math. Phys.},
  FJOURNAL = {Communications in Mathematical Physics},
    VOLUME = {141},
      YEAR = {1991},
    NUMBER = {2},
     PAGES = {381--411},
}

\bib{GG}{article} {
    AUTHOR = {Giansiracusa, N.},
    AUTHOR = {Gibney, A.},
     TITLE = {The cone of type {$A$}, level 1, conformal blocks divisors},
   JOURNAL = {Adv. Math.},
  FJOURNAL = {Advances in Mathematics},
    VOLUME = {231},
      YEAR = {2012},
    NUMBER = {2},
     PAGES = {798--814},
}




\bib{EGoursat}{article} {
    AUTHOR = {Goursat, E.},
     TITLE = {Sur les fonctions d'une variable analogues aux fonctions
              hyperg\'{e}om\'{e}triques},
   JOURNAL = {Ann. Sci. \'{E}cole Norm. Sup. (3)},
  FJOURNAL = {Annales Scientifiques de l'\'{E}cole Normale Sup\'{e}rieure. Troisi\`eme
              S\'{e}rie},
    VOLUME = {3},
      YEAR = {1886},
     PAGES = {107--136},
}

\bib{Hobson}{article}{
    AUTHOR = {Hobson, N.},
     TITLE = {Quantum {K}ostka and the rank one problem for
              {$\germ{sl}_{2m}$}},
   JOURNAL = {Adv. Geom.},
  FJOURNAL = {Advances in Geometry},
    VOLUME = {19},
      YEAR = {2019},
    NUMBER = {1},
     PAGES = {71--88},
}
		



\bib{Haraoka}{article} {
    AUTHOR = {Haraoka, Y.},
     TITLE = {Finite monodromy of {P}ochhammer equation},
   JOURNAL = {Ann. Inst. Fourier (Grenoble)},
  FJOURNAL = {Universit\'{e} de Grenoble. Annales de l'Institut Fourier},
    VOLUME = {44},
      YEAR = {1994},
    NUMBER = {3},
     PAGES = {767--810},
}

\bib{Katz}{book} {
    AUTHOR = {Katz, N.},
     TITLE = {Rigid local systems},
    SERIES = {Annals of Mathematics Studies},
    VOLUME = {139},
 PUBLISHER = {Princeton University Press, Princeton, NJ},
      YEAR = {1996},
     PAGES = {viii+223},
}
\bib{Kiers1}{article} {
    AUTHOR = {Kiers, J.},
     TITLE = {On the {S}aturation {C}onjecture for {S}pin(2{$n$})},
   JOURNAL = {Exp. Math.},
  FJOURNAL = {Experimental Mathematics},
    VOLUME = {30},
      YEAR = {2021},
    NUMBER = {2},
     PAGES = {258--267},
      ISSN = {1058-6458},
}

\bib{Kiers}{article}{
   author={Kiers, J.},
   title={Extremal rays of the embedded subgroup saturation cone}
   note ={Ann. Inst. Fourier, to appear, arXiv:1909.09262.}
   YEAR= {2021}

}



\bib{KT}{article} {
    AUTHOR = {Knutson, A.}
    AUTHOR =  {Tao, T.},
     TITLE = {The honeycomb model of {${\rm GL}_n({\bf C})$} tensor
              products. {I}. {P}roof of the saturation conjecture},
   JOURNAL = {J. Amer. Math. Soc.},
  FJOURNAL = {Journal of the American Mathematical Society},
    VOLUME = {12},
      YEAR = {1999},
    NUMBER = {4},
     PAGES = {1055--1090},
}
\bib{KTW}{article} {
    AUTHOR = {Knutson, A.}
    AUTHOR =  {Tao, T.}
    AUTHOR=  {Woodward, C.},
     TITLE = {The honeycomb model of {${\rm GL}_n(\Bbb C)$} tensor products.
              {II}. {P}uzzles determine facets of the
              {L}ittlewood-{R}ichardson cone},
   JOURNAL = {J. Amer. Math. Soc.},
  FJOURNAL = {Journal of the American Mathematical Society},
    VOLUME = {17},
      YEAR = {2004},
    NUMBER = {1},
     PAGES = {19--48},
}
\bib{Kostant}{article}{
    AUTHOR = {Kostant, B.},
     TITLE = {A formula for the multiplicity of a weight},
   JOURNAL = {Trans. Amer. Math. Soc.},
  FJOURNAL = {Transactions of the American Mathematical Society},
    VOLUME = {93},
      YEAR = {1959},
     PAGES = {53--73},
}

\bib{LS}{article} {
    AUTHOR = {Laszlo, Y.}
    AUTHOR = {Sorger, C.},
     TITLE = {The line bundles on the moduli of parabolic {$G$}-bundles over
              curves and their sections},
   JOURNAL = {Ann. Sci. \'{E}cole Norm. Sup. (4)},
  FJOURNAL = {Annales Scientifiques de l'\'{E}cole Normale Sup\'{e}rieure. Quatri\`eme
              S\'{e}rie},
    VOLUME = {30},
      YEAR = {1997},
    NUMBER = {4},
     PAGES = {499--525},
}

\bib{laszlo}{article} {
    AUTHOR = {Laszlo, Y.},
     TITLE = {Hitchin's and {WZW} connections are the same},
   JOURNAL = {J. Differential Geom.},
  FJOURNAL = {Journal of Differential Geometry},
    VOLUME = {49},
      YEAR = {1998},
    NUMBER = {3},
     PAGES = {547--576},
}
\bib{Laumon}{article} {
    AUTHOR = {Laumon, G.},
     TITLE = {Un analogue global du c\^one nilpotent},
   JOURNAL = {Duke Math. J.},
  FJOURNAL = {Duke Mathematical Journal},
    VOLUME = {57},
      YEAR = {1988},
    NUMBER = {2},
     PAGES = {647--671},
}
	
\bib{Loo}{article} {
    AUTHOR = {Looijenga, E.},
     TITLE = {Unitarity of {${\rm SL}(2)$}-conformal blocks in genus zero},
   JOURNAL = {J. Geom. Phys.},
  FJOURNAL = {Journal of Geometry and Physics},
    VOLUME = {59},
      YEAR = {2009},
    NUMBER = {5},
     PAGES = {654--662},
}
\bib{MM}{book} {
    AUTHOR = {Malle, G.},
    AUTHOR = {Matzat, B. H.},
     TITLE = {Inverse {G}alois theory},
    SERIES = {Springer Monographs in Mathematics},
      NOTE = {Second edition},
 PUBLISHER = {Springer, Berlin},
      YEAR = {2018},
     PAGES = {xvii+532},
}
\bib{MO}{article} {
    AUTHOR = {Marian, A.}
    AUTHOR = {Oprea, D.},
     TITLE = {The level-rank duality for non-abelian theta functions},
   JOURNAL = {Invent. Math.},
  FJOURNAL = {Inventiones Mathematicae},
    VOLUME = {168},
      YEAR = {2007},
    NUMBER = {2},
     PAGES = {225--247},
}
	

\bib{Masbaum}{incollection}{
    AUTHOR = {Masbaum, G.},
     TITLE = {An element of infinite order in {TQFT}-representations of
              mapping class groups},
 BOOKTITLE = {Low-dimensional topology ({F}unchal, 1998)},
    SERIES = {Contemp. Math.},
    VOLUME = {233},
     PAGES = {137--139},
 PUBLISHER = {Amer. Math. Soc., Providence, RI},
      YEAR = {1999},
}


\bib{MS}{article} {
    AUTHOR = {Mehta, V. B.}
    AUTHOR = {Seshadri, C. S.},
     TITLE = {Moduli of vector bundles on curves with parabolic structures},
   JOURNAL = {Math. Ann.},
  FJOURNAL = {Mathematische Annalen},
    VOLUME = {248},
      YEAR = {1980},
    NUMBER = {3},
     PAGES = {205--239},
}
\bib{NT}{article} {
    AUTHOR = {Nakanishi, T.}
    AUTHOR=  {Tsuchiya, A.},
     TITLE = {Level-rank duality of {WZW} models in conformal field theory},
   JOURNAL = {Comm. Math. Phys.},
  FJOURNAL = {Communications in Mathematical Physics},
    VOLUME = {144},
      YEAR = {1992},
    NUMBER = {2},
     PAGES = {351--372},
}

\bib{Ou}{article} {
    AUTHOR = {Oudompheng, R.},
     TITLE = {Rank-level duality for conformal blocks of the linear group},
   JOURNAL = {J. Algebraic Geom.},
  FJOURNAL = {Journal of Algebraic Geometry},
    VOLUME = {20},
      YEAR = {2011},
    NUMBER = {3},
     PAGES = {559--597},
}


\bib{OO}{article} {
    AUTHOR = {Orevkov, S. Yu.}
    AUTHOR = {Orevkov, Yu. P.},
     TITLE = {The {A}gnihotri-{W}oodward-{B}elkale polytope and {K}lyachko
              cones},
   JOURNAL = {Mat. Zametki},
  FJOURNAL = {Matematicheskie Zametki},
    VOLUME = {87},
      YEAR = {2010},
    NUMBER = {1},
     PAGES = {101--107},
       URL = {https://doi.org/10.1134/S0001434610010128},
}

\bib{pauly}{article} {
    AUTHOR = {Pauly, C.},
     TITLE = {Espaces de modules de fibr\'{e}s paraboliques et blocs conformes},
   JOURNAL = {Duke Math. J.},
  FJOURNAL = {Duke Mathematical Journal},
    VOLUME = {84},
      YEAR = {1996},
    NUMBER = {1},
     PAGES = {217--235},
}

\bib{paulyB}{incollection} {,
    AUTHOR = {Pauly, C.},
     TITLE = {La dualit\'{e} \'{e}trange [d'apr\`es {P}. {B}elkale, {A}. {M}arian et
              {D}. {O}prea]},
      NOTE = {S\'{e}minaire Bourbaki. Vol. 2007/2008},
   JOURNAL = {Ast\'{e}risque},
  FJOURNAL = {Ast\'{e}risque},
    NUMBER = {326},
      YEAR = {2009},
     PAGES = {Exp. No. 994, ix, 363--377 (2010)},
}

\bib{popa}{incollection} {
    AUTHOR = {Popa, M.},
     TITLE = {Generalized theta linear series on moduli spaces of vector
              bundles on curves},
 BOOKTITLE = {Handbook of moduli. {V}ol. {III}},
    SERIES = {Adv. Lect. Math. (ALM)},
    VOLUME = {26},
     PAGES = {219--255},
 PUBLISHER = {Int. Press, Somerville, MA},
      YEAR = {2013},
}
\bib{Goursat}{article} {
    AUTHOR = {Radchenko, D.}
    AUTHOR = {Rodriguez Villegas, F.},
     TITLE = {Goursat rigid local systems of rank four},
   JOURNAL = {Res. Math. Sci.},
  FJOURNAL = {Research in the Mathematical Sciences},
    VOLUME = {5},
      YEAR = {2018},
    NUMBER = {4},
     PAGES = {Paper No. 38, 34},
}
	
\bib{Ram}{article}{
    AUTHOR = {Ramadas, T. R.},
     TITLE = {The ``{H}arder-{N}arasimhan trace'' and unitarity of the
              {KZ}/{H}itchin connection: genus 0},
   JOURNAL = {Ann. of Math. (2)},
  FJOURNAL = {Annals of Mathematics. Second Series},
    VOLUME = {169},
      YEAR = {2009},
    NUMBER = {1},
     PAGES = {1--39},
}
\bib{R1}{article} {
    AUTHOR = {Ressayre, N.},
     TITLE = {Geometric invariant theory and the generalized eigenvalue
              problem},
   JOURNAL = {Invent. Math.},
  FJOURNAL = {Inventiones Mathematicae},
    VOLUME = {180},
      YEAR = {2010},
    NUMBER = {2},
     PAGES = {389--441},
}

	
\bib{SV}{article}{
    AUTHOR = {Schechtman, V.}
    AUTHOR = {Varchenko, A.},
     TITLE = {Arrangements of hyperplanes and {L}ie algebra homology},
   JOURNAL = {Invent. Math.},
  FJOURNAL = {Inventiones Mathematicae},
    VOLUME = {106},
      YEAR = {1991},
    NUMBER = {1},
     PAGES = {139--194},
}



\bib{Schwarz}{article} {
    AUTHOR = {Schwarz, H. A.},
     TITLE = {Ueber diejenigen {F}\"{a}lle, in welchen die {G}aussische
              hypergeometrische {R}eihe eine algebraische {F}unction ihres
              vierten {E}lementes darstellt},
   JOURNAL = {J. Reine Angew. Math.},
  FJOURNAL = {Journal f\"{u}r die Reine und Angewandte Mathematik. [Crelle's
              Journal]},
    VOLUME = {75},
      YEAR = {1873},
     PAGES = {292--335},
}
\bib{CS}{book} {
    AUTHOR = {Sherman, C.},
     TITLE = {Weight {S}tretching in {M}oduli of {P}arabolic {B}undles and
              {Q}uiver {R}epresentations},
      NOTE = {Thesis (Ph.D.)--The University of North Carolina at Chapel
              Hill},
 PUBLISHER = {ProQuest LLC, Ann Arbor, MI},
      YEAR = {2016},
     PAGES = {77},
      ISBN = {978-1339-81176-5},
   MRCLASS = {Thesis},
  MRNUMBER = {3542240},
       URL =
              {http://gateway.proquest.com/openurl?url_ver=Z39.88-2004&rft_val_fmt=info:ofi/fmt:kev:mtx:dissertation&res_dat=xri:pqm&rft_dat=xri:pqdiss:10119985},
}	

\bib{SimOld}{incollection} {
    AUTHOR = {Simpson, C. T.},
     TITLE = {Products of matrices},
 BOOKTITLE = {Differential geometry, global analysis, and topology
              ({H}alifax, {NS}, 1990)},
    SERIES = {CMS Conf. Proc.},
    VOLUME = {12},
     PAGES = {157--185},
 PUBLISHER = {Amer. Math. Soc., Providence, RI},
      YEAR = {1991},
}

\bib{sorger}{incollection} {,
    AUTHOR = {Sorger, C.},
     TITLE = {La formule de {V}erlinde},
      NOTE = {S\'{e}minaire Bourbaki, Vol. 1994/95},
   JOURNAL = {Ast\'{e}risque},
  FJOURNAL = {Ast\'{e}risque},
    NUMBER = {237},
      YEAR = {1996},
     PAGES = {Exp. No. 794, 3, 87--114},
}
\bib{Swin}{article}{
		author={Swinarski, D.},
		title={\texttt{\upshape ConformalBlocks}: a Macaulay2 package for computing conformal block divisors},
		date={2010},
		note={Version 1.1, {http://www.math.uiuc.edu/Macaulay2/}},
}
\bib{thaddy}{article} {
     AUTHOR = {Thaddeus, M.}

     TITLE = {The universal implosion
and the multiplicative Horn problem},

		note={Lecture at AGNES, slides available at http://www.math.columbia.edu/~thaddeus/},
		year={2014}


}

\bib{tuy}{incollection} {
    AUTHOR = {Tsuchiya, A.}
    AUTHOR = {Ueno, K.}
    AUTHOR = {Yamada, Y.},
     TITLE = {Conformal field theory on universal family of stable curves
              with gauge symmetries},
 BOOKTITLE = {Integrable systems in quantum field theory and statistical
              mechanics},
    SERIES = {Adv. Stud. Pure Math.},
    VOLUME = {19},
     PAGES = {459--566},
 PUBLISHER = {Academic Press, Boston, MA},
      YEAR = {1989},
}

\bib{Ueno}{incollection} {
    AUTHOR = {Ueno, K.},
     TITLE = {Introduction to conformal field theory with gauge symmetries},
 BOOKTITLE = {Geometry and physics ({A}arhus, 1995)},
    SERIES = {Lecture Notes in Pure and Appl. Math.},
    VOLUME = {184},
     PAGES = {603--745},
 PUBLISHER = {Dekker, New York},
      YEAR = {1997},
}

\bib{Witten}{incollection} {
    AUTHOR = {Witten, E.},
     TITLE = {The {V}erlinde algebra and the cohomology of the
              {G}rassmannian},
 BOOKTITLE = {Geometry, topology, \& physics},
    SERIES = {Conf. Proc. Lecture Notes Geom. Topology, IV},
     PAGES = {357--422},
 PUBLISHER = {Int. Press, Cambridge, MA},
      YEAR = {1995},
}

\bib{Z}{incollection} {
    AUTHOR = {Zelevinsky, A.},
     TITLE = {Littlewood-{R}ichardson semigroups},
 BOOKTITLE = {New perspectives in algebraic combinatorics ({B}erkeley, {CA},
              1996--97)},
    SERIES = {Math. Sci. Res. Inst. Publ.},
    VOLUME = {38},
     PAGES = {337--345},
 PUBLISHER = {Cambridge Univ. Press, Cambridge},
      YEAR = {1999},
}


\end{biblist}
\end{bibdiv}
\end{document}